\documentclass[a4paper]{article}
\def\writing{0} 
\def\thesis{0} 

\usepackage[utf8]{inputenc}
\usepackage[T1]{fontenc}
\usepackage{textcomp}
\usepackage{amsmath, amssymb, amsthm}
\usepackage{geometry}
\usepackage{tikz-cd}
\usepackage{mdframed}
\usepackage{microtype}
\usepackage{mathrsfs}
\usepackage{rotating}
\usepackage{comment}
\usepackage{enumerate}
\usepackage{graphicx}
\usepackage{stmaryrd}

\usepackage{import}
\usepackage{xifthen}
\usepackage{pdfpages}
\usepackage{transparent}

\usepackage[backend = biber,style=alphabetic]{biblatex}
\if\writing1
	\usepackage[colorinlistoftodos]{todonotes}
\else
	\usepackage[disable]{todonotes}
\fi
\usepackage{zref-clever}
\newcommand{\cref}[1]{\zcref{#1}}
\newcommand{\Cref}[1]{\zcref[S]{#1}}
\usepackage{hyperref}

\if\writing1
	\usepackage[inline]{showlabels}
\fi
\usepackage{mathrsfs}
\usepackage{dsfont}

\newcommand{\N}{\mathbb{N}}
\newcommand{\Z}{\mathbb{Z}}

\newcommand{\Q}{\mathbb{Q}}
\newcommand{\C}{\mathbb{C}}
\newcommand{\R}{\mathbb{R}}

\newcommand{\fra}{\mathfrak{a}}
\newcommand{\frb}{\mathfrak{b}}

\newcommand{\aff}{\mathbb{A}}
\newcommand{\pro}{\mathbb{P}}

\newcommand{\lctc}{\mathbb{L}^{\mathrm{log}}}
\newcommand{\ctc}{\mathbb{L}}
\newcommand{\lOmega}{\Omega^{\mathrm{log}}}

\newcommand{\ltr}{\par \noindent \framebox[1\width]{ $\implies$ } \hspace{.2cm}}
\newcommand{\rtl}{\par \noindent \framebox[1\width]{ $\impliedby$ } \hspace{.2cm} }

\newcommand{\derotimes}{\otimes^\mathbf{L}}

\DeclareMathOperator*{\colim}{colim}

\DeclareMathOperator{\coker}{coker}

\DeclareMathOperator{\im}{Im}
\DeclareMathOperator{\spec}{Spec}

\DeclareMathOperator{\wt}{wt}

\DeclareMathOperator{\length}{length}
\DeclareMathOperator{\cont}{cont}

\DeclareMathOperator{\vd}{cst}
\DeclareMathOperator{\const}{cst}
\DeclareMathOperator{\bls}{bls}
\DeclareMathOperator{\jac}{Jac}

\DeclareMathOperator{\step}{step}
\DeclareMathOperator{\keys}{key}

\DeclareMathOperator{\GL}{GL}
\DeclareMathOperator{\dlog}{dlog}

\DeclareMathOperator{\disc}{disc}
\DeclareMathOperator{\kv}{kv}

\newcommand{\ldisc}{\operatorname{disc}^{\mathrm{log}}}

\newcommand{\dd}{\mathrm{d}}
\newcommand{\into}{\hookrightarrow}

\newcommand{\diff}{\mathfrak{d}}
\newcommand{\ldiff}{\mathfrak{d}^{\mathrm{log}}}

\newcommand{\aldiff}{\mathfrak{d}^{\mathrm{abs,log}}}

\newcommand{\st}{%
  \nonscript\;
  \ifnum\currentgrouptype=16
    \,\middle|\,
  \else
    \,|\,
  \fi
  \nonscript\;}

\swapnumbers

\NewDocumentCommand{\mynewtheorem}{momo}{%
  \IfValueTF{#4}
    {\newtheorem{#1}{#3}[#4]}
    {%
      \IfValueTF{#2}
        {%
          \AddToHook{env/#1/begin}{\zcsetup{countertype={#2=#1}}}%
          \newtheorem{#1}[#2]{#3}%
        }
        {\newtheorem{#1}{#3}}%
    }%
}

\mynewtheorem{theorem}[subsubsection]{Theorem}
\mynewtheorem{lemma}[subsubsection]{Lemma}
\mynewtheorem{corollary}[subsubsection]{Corollary}
\mynewtheorem{conjecture}[subsubsection]{Conjecture}
\mynewtheorem{proposition}[subsubsection]{Proposition}
\theoremstyle{definition}
\mynewtheorem{definition}[subsubsection]{Definition}
\mynewtheorem{question}[subsubsection]{Question}
\mynewtheorem{remark}[subsubsection]{Remark}
\mynewtheorem{observation}[subsubsection]{Observation}
\mynewtheorem{example}[subsubsection]{Example}
\mynewtheorem{computation}[subsubsection]{Computation}
\mynewtheorem{candef}[subsubsection]{Candidate Definition}
\zcRefTypeSetup{computation}{
Name-sg = Computation ,
name-sg = computation ,
Name-pl = Computations ,
name-pl = computations ,
}
\mynewtheorem{notation}[subsubsection]{Notation}
\newtheoremstyle{para} 
    {0}                    
    {1mm}                    
    {}                   
    {}                           
    {\bfseries}                   
    {}                          
    {.5em}                       
    {}  
\theoremstyle{para}
\zcRefTypeSetup{para}{
Name-sg = § ,
name-sg = § ,
Name-pl = § ,
name-pl = § ,
}
\mynewtheorem{para}[subsubsection]{}

\newcommand\nr{\refstepcounter{equation}\tag{\theequation}}
\numberwithin{equation}{subsection}

\renewcommand{\theequation}{\arabic{section}.\arabic{subsection}.\alph{equation}}

\newcommand{\stacks}[1]{\cite[\href{https://stacks.math.columbia.edu/tag/#1}{Tag #1}]{stacks-project}}

\newcommand{\refactor}[1]{\todo[inline, color=green]{#1}}

\author{Micha\"el Maex}
\date{\today}
\title{The integral (log) cotangent complex of extensions of valued fields}
\begin{document}
\let\temp\phi
\let\phi\varphi
\let\varphi\temp
\maketitle
\todo[inline, color=red]{Before posting anywhere, remove todonotes, parts and later sections.}
\begin{abstract}
	Let $(L, v_L) / (K, v_K)$ be a finite or purely transcendental extension of real valued fields. 
	We construct the associated integral cotangent and log cotangent complexes in terms of a MacLane-Vaquié chain approximating  $v_L$. 
	This leads to explicit formulas for associated invariants such as the (absolute) (log) different, weight norm and Kähler norm. 
	As a corollary of our methods we obtain strong control of the higher homology of the integral (log) cotangent complex, generalizing an important result of Gabber and Ramero to the logarithmic setting.  
\end{abstract}
\tableofcontents

\section{Introduction}\label{sec:introduction}

\subsection{The integral (log) differentials and cotangent complex} \label{sec:the_integral_(log)_differentials_and_cotangent_complex}
Let $(L, v_L) / (K, v_K)$ be an extension of valued fields, and let $L^{\circ}, K^{\circ}$ be their respective rings of integers. 
The integral differentials, $\Omega_{L^\circ / K^{\circ}}$ and the integral log differentials $\lOmega_{L^{\circ} / K^{\circ}}$ are the source of many invariants in birational, analytic, and arithmetic geometry. 
We highlight a few examples. 
\begin{itemize}
	\item If $L / K$ is an algebraic extension, then the different is the length of $\Omega_{L^{\circ} / K^{\circ}}$. The log different is defined analogously. These are important invariants in valuation and ramification theory. For example, see \cite[Chapitre~III]{serreCorpsLocaux1968}. 

	\item On Berkovich analytic spaces, an appropriate notion of log discrepancy defines weight functions that encode both birational information and arithmetic information, see \cite{nicaisexu2016}, and \cite{kayaJumps2025}.
		The Kähler norm on $\Omega_{L / K}$ is the unique norm such that $(\Omega_{L^{\circ} / K^{\circ}})_{tf}$ is an almost unit ball \cite{temkinMetrizationDifferentialPluriforms2016}. 
	\item In the theory of embedded resolutions of a singular map $f: X \to \spec \C[T]$, the log discrepancy at each component $E$ of a resolution with generic point $\eta$ is the order of vanishing of the Gelfand-Leray form in the module $\Omega_{\mathcal{O}_{\eta} / \C[T]_{(T)}}^{\text{log,} \wedge \dim X - 1}$ \cite[Lemma~9.6]{nicaisesebag2007}.  
		One of the most important invariants, the log canonical threshold, is the minimal value of the normalized log discrepancies at each component. 
\end{itemize}
All of these invariants can be defined in terms of the length of specific torsion subquotients of the $L^{\circ}$-modules $\Omega_{L^{\circ} / K^{\circ}}$ or $\lOmega_{L^{\circ} / K^{\circ}}$, and they can all be computed from a good description of these modules. 
When $L$ is not discretely valued (i.e.\ $L^{\circ}$ is non-Noetherian), the length of an $L^{\circ}$-module is no longer a useful invariant. 
Instead, we will use Temkin's \emph{content} \cite[Section 2.6]{temkinMetrizationDifferentialPluriforms2016}, which generalizes length and in particular, agrees with length when the value group of $L$ is $\Z$. 

While the original goal was to obtain a description of the module of integral (log) differentials that is sufficient to compute these invariants, it is actually easier to keep track of the whole (log) cotangent complex throughout the computation.
The cotangent complex of a ring morphism $A \to B$, denoted $\ctc_{B  /A}$, is a derived version of the differentials $\Omega_{B / A}$ and was first introduced in \cite{illusieComplexeCotangentDeformations2009}. 
The log cotangent complex, denoted $\lctc_{B / A}$ is an analogous version for log rings.
In this paper, the log cotangent complex will always refer to Gabber's log cotangent complex as described in \cite[Section 8]{olssonLogarithmicCotangentComplex2005}.
\subsection{Motivation} \label{sec:motivation}

This paper is part of the author's PhD thesis, which aims to define and study weight functions on Berkovich spaces for the purpose of studying skeleta. 
However, this construction of $\ctc_{L^{\circ} / K^{\circ}}$ and $\lctc_{L^{\circ} / K^{\circ}}$ and computation of the associated invariants is an entirely value-theoretic matter, and is interesting in its own right due to the aforementioned connections to number theory and birational geometry. 
For this reason, we publish this separately.  
Except for a few remarks, we defer the applications in Berkovich geometry to a follow-up paper.

\subsection{Main results} \label{sec:main_theorems}

After the technical computation of the (log) cotangent complex of $L^{\circ} / K^{\circ}$ when $L / K$ is finite or purely transcendental (PT) in \Cref{sec:the_torsion_of_Omega_for_primitive_extensions} we conclude with the main results in Sections \ref{sec:extensions_of_valued_fields} and \ref{sec:the_kahler_valuation_and_the_absolute_log_different}. 
These consist of simple formulas for various invariants associated with extensions of valued fields in terms of augmentation chains, as well as control of the higher homology of the integral (log) cotangent complex. 
By the nature of this kind of work, the main results are multiple smaller results. As not all invariants are defined for all extensions, it would be infeasible to repeat every result in the introduction. 
Instead, we have assembled an overview of the main results in \Cref{tab:overview_results}. 

\begin{table}[h!]
	\centering
\if\thesis1
\begin{tabular}{l|ll}
	Extension & Computed object & Theorem \\
	\hline 
	\textbf{finite} &
	Structure of $\ctc_{L^{\circ} / K^{\circ}}$ and $\lctc_{L^{\circ} / K^{\circ}}$ & \Cref{prop:ctc_finite_ext} \\
			&The (log) different $\diff_{L / K}$ and $\ldiff_{L / K}$ & \Cref{prop:comp_log_diff} \\
	\hline 
	\textbf{PT} & 
	Structure of $\ctc_{L^{\circ} / K^{\circ}}$ and $\lctc_{L^{\circ} / K^{\circ}}$ & \Cref{prop:pure_trans_deg_0} \\
		& The (log) discrepancy $\disc \dd T$ and $\ldisc \dd T$ & \Cref{thm:disc_and_log_disc_pure_trans}\\
		& Kähler valuation $\kv \dd T$ & \Cref{thm:comp_kahler}\\
		& The absolute log different $\aldiff_{L / K}$  & \Cref{thm:comp_log_different} \\ 
	\hline
	\textbf{arbitrary} & Homology $H_i(\ctc_{L^{\circ} / K^{\circ}})$ and $H_i(\lctc_{L^{\circ} / K^{\circ}})$ & \Cref{thm:lctc_higher_homology}
\end{tabular}
\caption{Overview of the computed invariants for an extension $L/ K$}
\else
\begin{tabular}{l|ll}
	Extension & Computed object & Theorem \\
	\hline 
	\textbf{$L / K$ is finite} &
	Structure of $\ctc_{L^{\circ} / K^{\circ}}$ and $\lctc_{L^{\circ} / K^{\circ}}$ & \Cref{prop:ctc_finite_ext} \\
			&The (log) different $\diff_{L / K}$ and $\ldiff_{L / K}$ & \Cref{prop:comp_log_diff} \\
	\hline 
	\textbf{$L := K(T) / K$ is PT} & 
	Structure of $\ctc_{L^{\circ} / K^{\circ}}$ and $\lctc_{L^{\circ} / K^{\circ}}$ & \Cref{prop:pure_trans_deg_0} \\
		& The (log) discrepancy $\disc \dd T$ and $\ldisc \dd T$ & \Cref{thm:disc_and_log_disc_pure_trans}\\
		& Kähler valuation $\kv \dd T$ & \Cref{thm:comp_kahler}\\
		& The absolute log different $\aldiff_{L / K}$  & \Cref{thm:comp_log_different} \\ 
	\hline
	\textbf{$L / K$ is arbitrary} & Homology $H_i(\ctc_{L^{\circ} / K^{\circ}})$ and $H_i(\lctc_{L^{\circ} / K^{\circ}})$ & \Cref{thm:lctc_higher_homology}
\end{tabular}
\caption{Overview of the computed invariants}
\fi
\label{tab:overview_results}

\end{table}
\subsection{The approach} \label{sec:the_approach}

Let $(K, v_K)$ be a non-trivially valued field. 
One of the main technical tools we use are chains of ordinary and limit augmentations as introduced by M.\ Vaquié \cite{vaquie2005,vaquie2007a,vaquie2007b} and summarized by E. Nart in \cite{nartMacLaneVaquieChains2021}.
In short, these give an explicit description of any semi-valuation, $w$, on $K[T]$ extending $v_K$ in terms of the data called key polynomials. 
For each intermediate step of the chain, we define a ring $R_i$ such that the union of these rings gives $(K^{\circ}[T]_{\ker w}, w)^{\circ}$ after a localization.
We can then keep track of the complexes $\ctc_{R_i / K^{\circ}}$ and $\lctc_{R_i / K^{\circ}}$, as well as the content of certain subquotients which define the aforementioned invariants. 
The process is analogous to breaking up a complicated field extension into a series of smaller ones. 
However, the rings $R_i$ are not valuation rings.
Instead, we introduce the notion of \emph{scalable} rings, which for these purposes behave similarly. 
This way we can understand any finite and purely transcendental extension.

MacLane-Vaquié chains were chosen primarily because, unlike similar constructions, they work without any further assumptions on $K$.
The theory is ``blind'' to any base change behavior, which, for this work, is a feature as it avoids any wild ramification complications. 
There are some further advantages. 
The data in the chain is closely related to the metric and piecewise linear structure on the valuative tree. This will be studied more extensively in the special case of the Berkovich affine line in a follow up work. 
The explicit nature makes it easy to construct and study examples. 
The key polynomials match well to the structure of the log differentials, allowing us to directly compute the log differentials, whereas previous works have usually studied the log case by comparing it to the (non-log) differentials.
In fact $\{\dlog \phi_i\}_{i \in I}$ where $\phi_i$ runs over the key polynomials in an augmentation chain generate $\lOmega_{(K^{\circ}[T]_{\ker w})^{\circ}/ K^{\circ}}$, see \Cref{rem:total_generators_log_differentials}, and each augmentation corresponds to an extension of the log differentials by a cyclic module.

\subsection{Overview of the paper} \label{sec:overview}

\Cref{sec:maclane-vaquie_valuations} is dedicated to valuation theory. 
We start by fixing some conventions for (semi) valuations on rings and modules. 
Then we give a brief overview of the theory of MacLane-Vaquié. 
Finally, we introduce some notions that will be useful for the computations later on. 

Most of the technical computations are contained in \Cref{sec:the_torsion_of_Omega_for_primitive_extensions}. 
Here we compute the (log) cotangent complex and associated invariants of an extension in terms of a MacLane-Vaquié chain. 
Readers who are not interested in the methods can safely skip this section.

Sections \ref{sec:extensions_of_valued_fields} and \ref{sec:the_kahler_valuation_and_the_absolute_log_different} contain the main results of this paper. 
The results are, for the most part, direct consequences of the computations in \Cref{sec:the_torsion_of_Omega_for_primitive_extensions}. 

\subsection{Further questions} \label{sec:further_questions}

It is natural to ask whether the same computations can be carried out when we remove the assumption that valuations are real valued. 
As the original motivation was Berkovich geometry, where only real valued fields are considered, we did not study this in detail. 
Still, we can make the following conjectural remarks. 
It is likely that a version of the computation in \Cref{sec:the_torsion_of_Omega_for_primitive_extensions} still holds for more general valuations, although one will probably need to perform a more careful analysis of the cases ($\aleph$)-($\daleth$) in \Cref{par:cases_aleph_to_daleth} and add a case when $\mu$ increases the rank of the value group. 
With these modifications, we expect that all results regarding the structure of $\ctc_{(K[T]_{\ker w}, w)^{\circ} / K^{\circ}}$ and $\lctc_{(K[T]_{\ker w}, w)^{\circ} / K^{\circ}}$ will generalize. 
However, any results that concern the content of modules, suffer from the fact that no good notion of content exists for general modules over general valuation rings, see \cite[§~5.1.10-5.1.11]{hubtem}. 

\subsection{Relation to other works} \label{sec:relation_to_other_works}

The idea that the integral log differential can contain interesting information about reductions of analytic spaces goes back to \cite{kontsevichAffineStructuresNonArchimedean2006} and was further developed into the theory of weight functions in \cite{mustataWeightFunctionsNonArchimedean2015,bakerWeightFunctionsBerkovich2016}. 
This is what inspired the study of the log-discrepancy in this work. 
In fact, the log-discrepancy can be extended to the weight valuation, see \Cref{rem:hint_weight}.
Another metrization of differential forms on Berkovich spaces, the Kähler norm/valuation, appears in \cite{temkinMetrizationDifferentialPluriforms2016}. 
We study this in \Cref{sec:the_kahler_norm_and_key_polynomials}. 

As far as the author is aware, the integral (log) differentials of $L / K$ outside the case where $L$ is discretely valued were first studied in \cite{gabberAlmostRingTheory2003}, and this has inspired further exploration, for example \cite{temkinMetrizationDifferentialPluriforms2016} and this paper.  

Other works have related similar invariants to MacLane-Vaquié chains. 
In \cite{nartLocalComputationDifferents2014} the different of a finite extension is computed from a MacLane chain in the discretely valued case. 
In \cite{novacoskiKahlerDifferentialsPure2025} the structure of $\Omega_{L^{\circ} / K^{\circ}}$ is described in the special case when $L$ is a pure extension of $K$.

Towards the end of the writing of this paper, the author was made aware of an unreleased draft by Katharina Hübner and Michael Temkin \cite{hubtem}, which obtains some overlapping results, albeit by different methods. 
More precisely, \cite[Theorem 5.4.1]{hubtem} is the same as \Cref{thm:lctc_higher_homology}. 
And, \cite[Lemma 2.5.3]{hubtem}  is a special case of \Cref{thm:comp_log_different}, when the MacLane-Vaquié chain consists of a single augmentation of degree $p$.

\subsection{Acknowledgements} \label{sec:acknowledgements}

The author would like to thank his PhD advisor, Johannes Nicaise, for the support throughout this project. 
Art Waeterschoot, Tom Biesbrouck, and Antoine Ducros have been very helpful through many informal discussions about this and related material. 
Special gratitude goes to Michael Temkin for taking a lot of time to discuss with the author during his visit at the IHES. 
The author was supported by FWO (grant G0A2923N).

\section{Valuations and augmentations}\label{sec:maclane-vaquie_valuations}
The purpose of this section is to fix some conventions and notation relating to valuations as well as introduce the theory of Vaquié-MacLane, which serves as our main tool for studying finite and purely transcendental extensions of valued fields. 
\Cref{sec:notation_of_valuations} introduces basic concepts on (semi-)valuations, \Cref{sec:a_recap_of_maclane} gives an overview of constructing semi-valuations on $K[T]$. 
Finally, we use \Cref{sec:steps_and_w_optimal} to introduce a few concepts related to augmentation chains that will be useful in the computations later on. 

\subsection{Conventions on valuations} \label{sec:notation_of_valuations}
\begin{definition}
	A semi-valuation on a ring $A$ is a function $v: A \to (-\infty, \infty]$ such that for all $a, b \in A$
	\begin{itemize}
		\item $v(0) = \infty$,
		\item $v(a\cdot b) = v(a) + v(b)$,
		\item  $v(a+ b) \ge \min \{v(a), v(b)\} $.
	\end{itemize}
	A \emph{semi-valued ring} is a pair $(A, v)$. 
	We further define 
	\begin{itemize}
		\item the \emph{value group} $\Gamma_v $ to be the groupification of $v(A)\setminus \{\infty\} $. If the semi-valuation $v$ on $A$ is clear from context we may write $\Gamma_A$ instead of $\Gamma_v$. 
			We write $\Gamma_v(\infty) = \Gamma_v \cup \{\infty\} $. 
		\item We say that $v$ is \emph{trivial} if $\Gamma_v = \{0\} $. 
		\item For any $r \in \R \cup \{-\infty, \infty\}$ we write $A^{\ge r} = v^{-1}([r, \infty)) $ and $A^{> r} = v^{-1}((r, \infty))$. 
		\item  The \emph{kernel of} $v$ is $\ker v := v^{-1}(\infty)$.
		\item We say that  $v$ is a valuation if $\ker v = \{0\} $. 
		\item The \emph{unit ball} $A^{\circ} = A^{\ge 0} $ and the \emph{residue ring} $\widetilde A = A^{\circ} / A^{> 0}$. 
	\end{itemize}
\end{definition}
\begin{para}
The definition of a (semi-)valuation is not entirely standard across the literature. 
Notably, we require that $\Gamma_v \subset \R$, i.e. our semi-valuations are of rank 1. 
We require $\ker v = 0$ in the definition of a valuation, as is standard in the literature on non-archimedean geometry and number theory. 
Most literature on augmentations and key polynomials such as \cite{maclane1936construction,nartMacLaneVaquieChains2021} omit the requirement $\ker v = 0$, and thus their valuation is what we call a semi-valuation.  
\end{para}
\begin{definition}
	Let $(A, v)$ be a semi-valued ring and $M$ an $A$-module. 
	Then a semi-valuation on $M$ is a map $w: M \to \R(\infty)$ such that for any $a, b\in M$ and $c \in A$
	\begin{itemize}
		\item $w(a + b) \ge \min \{w(a), w(b)\} $ and, 
		\item $w(c\cdot a) = v(c) + w(a)$.
	\end{itemize}
	Suppose that $B$  is a $(A, v)$-algebra, then a semi-valuation $w$ on $B$ is a function that is both a semi-valuation when considering $B$ as a ring and as an $A$-module, or equivalently it is a semi-valuation on the ring  $B$ that extends the semi-valuation on  $A$. 
\end{definition}

\subsection{A brief overview of MacLane-Vaquié theory} \label{sec:a_recap_of_maclane}
\begin{para}
	Let $(K, v_K)$ be a non-trivially valued field. Throughout the rest of the paper we will assume that all semi-valuations on $K[T]$ are semi-valuations as a  $K$-algebra, i.e.\ they extend the valuation of $K$. 
	Several efforts have been made to classify or describe the set of all semi-valuations on $K[T]$. 
	When  $K$ is discretely valued Saunders MacLane found a way of constructing all semi-valuations on $K[T]$ in an iterative way \cite{nartMacLaneVaquieChains2021}. 
	When $K$ is complete, the set of semi-valuations on $K[T]$ can be interpreted as the points of the Berkovich affine line.
	In the special case when $K$ is also algebraically closed Vladimir Berkovich classified all semi-valuations on $K[T]$ in terms of decreasing chains of closed disks on $K$ and used this to divide the points into types I-IV \cite[§~1.4.4]{berkovichSpectralTheoryAnalytic1990}. 

	Later work by Michelle Vaquié extends the approach of MacLane to remove any assumptions on $(K, v_K)$ other than  $v$ being non trivially valued \cite{vaquie2005,vaquie2007a,vaquie2007b}.
	A nice overview paper on this was written by Enric Nart \cite{nartMacLaneVaquieChains2021}, which serves as the main source for this subsection. 
	We only study rank 1 semi-valuations on $K[T]$. This simplifies some results in \cite{nartMacLaneVaquieChains2021}, as we can uniformly choose the value group to a subgroup of $\R$ everywhere. 
\end{para}

\begin{para}
	The theory describes three ways of constructing a new semi-valuation on $K$-algebra $K[T]$ from one or more semi-valuations on $K[T]$: \emph{ordinary augmentations}, \emph{limit augmentations} and \emph{stable limits}, whose definitions we will recall shortly.
	The main result of the theory, \Cref{thm:vaquie_main}, states that a semi-valuation $w$ on $K[T]$ can be obtained by starting from a Gauss semi-valuation $v_0$ and the aforementioned building blocks. 
	By imposing simple restrictions on how we build $w$ with these augmentations - known as a MacLane-Vaquié chain - one can even construct $w$ in a canonical way up to some well described equivalence. 
\end{para}
\begin{definition}
	Let $\phi \in K[T]$ be a non-constant polynomial, and $f \in K[T]$ be any polynomial. 
	Then the $\phi$-expansion of $f$ is the unique way of writing \[
		f = \sum_{i = 0}^{n} a_i\phi^{i},
	\] 
	with $a_i \in K[T]$ and $\deg(a_i) < \deg(\phi)$.
	The $\phi$-expansion can be easily computed using Euclid's algorithm. 
\end{definition}

\begin{definition}
	A \emph{Gauss semi-valuation} on $K[T]$ is a semi-valuation of the form 
	 \[
		 w: K[T] \to \R(\infty): \sum_{i = 0}^{n} a_i (T - \alpha)^{i} \mapsto \min_{i = 0}^{n} \left\{v_K(a_i) + i\cdot \mu\right\}
	,\] 
	for a fixed $\alpha \in K$ and $\mu \in \R(\infty)$.
	We say that $\alpha$ is a \emph{center of $w$} and $\mu$ is the \emph{log-radius}. 
	To be consistent with notation later on, we will use the notation  \[
		w = [v_K, T-\alpha, \mu]
	.\] 
\end{definition}
Note that in our additive language the log-radius is simply the negative logarithm of the radius in the multiplicative language. For example a disk of radius $1$ would have log-radius $0$.  

\begin{para}
	Let $v$ be a semi-valuation on $K[T]$ and let $f, g, h \in K[T]$. 
	Then we say that 
	\begin{itemize}
		\item $f$ and $g$ are \emph{$v$-equivalent}, written $f \sim_v g$, if $f = g = 0$ or $v(f-g) > v(f)$, equivalently $v(f-g) > v(g)$, 
		\item $f$ is \emph{$v$-divisible} by $g$, written $g \mid_v f$, if there is a $c \in K[T]$ such that $gc \sim_v f$, 
		\item $f$ is $v$-irreducible\footnote{Perhaps $v$-prime would be a more apt name, but historically this is called $v$-irreducible.} if whenever $f \mid_v h_1h_2$ with $h_1, h_2 \in K[T]$, then $f \mid_v h_1$ or $f \mid_v h_2$,
		\item $f$ is $v$-minimal if whenever $f \mid_v g$ and $g\ne 0$, then $\deg f \le \deg g$,
	\end{itemize}
\end{para}
\begin{definition}
	Let $v$ be a semi-valuation on $K[T]$ and $f \in K[T]$.
	We say that $f$ is a \emph{key polynomial for $v$} if $f$ is monic, $v$-minimal, and $v$-irreducible. 
\end{definition}
Perhaps, the most important property of key polynomials is that they make the semi-valuation in the following definition well-defined. 

\begin{definition}\label{def:ordinary_augmentation_datum}
	An \emph{ordinary augmentation} of degree $m$ is a triple $\fra = (v, \phi, \mu)$ with $v$ a semi-valuation on $K[T]$,  $\phi$ a key polynomial for $v$ with $\deg \phi = m$, and $\mu \ge v(\phi)$. 

	Then we define
	\[
		[\fra]: K[T] \to \R: f = \sum_{i = 0}^{n}a_i \phi^{i} \mapsto \min_{i= 0}^{n} \{v(a_i) + i\cdot \mu\} 
	,\] 
	with $f = \sum_{i = 0}^{n}a_i \phi^{i}$ the $\phi$-expansion of $f$. 
	Then $[\fra]$ is a semi-valuation with $[\fra] \ge v$. 
	We may also write this as $[v, \fra]$ or $[v, \phi, \mu]$. 
\end{definition}
\begin{para}
In works cited on aumentations of semi-valuations the augmented semi-valuations are implicitly identified with the data defining the augmentation. 
But we find it clearer to distinguish the two. 
For us the augmentation will always refer to the defining data, and the resulting semi-valuation is referred to as the augmented semi-valuation. 
\end{para}
\begin{definition}
	Let  $I$ be a totally ordered set and $(v_i)_{i \in I}$ a sequence of semi-valuations on a ring $A$. 
	Then we say that $(v_i)_{i \in I}$ is a \emph{stable family} if $(v_i(a))_{i \in I}$ stabilizes for every $a \in A$, and we say that $\lim_{i} v_i$ is the \emph{stable limit} of $(v_i)_{i \in I}$. 
\end{definition}
\begin{definition}\label{def:continuous_family_datum}
	A \emph{continuous family} of degree $m$ and log-radius $\mu$ is a tuple $\fra = (v, (\psi_i, \gamma_i)_{i \in I})$ where 
	\begin{itemize}
		\item $I$ is a totally ordered set without maximal element, 
		\item each $(v, \psi_i, \gamma_i)$ is an ordinary augmentation of degree $m$. 
			For the remainder of the definition we write $v_i = [v, \psi_i, \gamma_i]$.
		\item For each $i <  j \in I$, $(v_i, \psi_j, \gamma_j)$   is an ordinary augmentation, with $\psi_j \not\sim_{v_i} \psi_i$. 
	\end{itemize}
	Moreover, we say that $\fra$ is
	\begin{itemize}
		\item \emph{stable} if $(v_i)_{i \in I}$ is a stable family,
		\item \emph{almost stable} if it is not stable and $\lim_i \gamma_i  = \infty$, 
		\item \emph{essential} if is not stable and $(v_i(f))_{i \in I}$ stabilizes for every $f \in K[T]$ with  $\deg(f) \le m$.
	\end{itemize} 
	We define $[\fra] = \lim_{i} v_i$ to be the pointwise limit, which also is a semi-valuation.
\end{definition}
\refactor{change almost stable to Cauchy? (Think about this)}

\begin{definition}
	Continue with the notation as in \Cref{def:continuous_family_datum}. 
	If $\fra$ is essential, a \emph{limit key polynomial for $\fra$} is a monic $f \in K[T]$ of minimal degree with the property that $(v_i(f))_{i \in I}$ does not stabilize. 
\end{definition}
\begin{remark}\label{rem:continuous_family}
	Continue with the context of the previous definition. 
	By the line before Lemma 3.4 in \cite{nartMacLaneVaquieChains2021} we see that $(v(\psi_i))_{i \in I}$ is constant and for  $i < j \in I$, it holds that $[v, \psi_i, \gamma_i] = [v, \psi_j, \gamma_i]$. 
	So we see that $(\gamma_i)_{i \in I}$ is strictly increasing. 
	Intuitively, the family of key polynomials $(\psi_i)$ give increasingly accurate directions in the valuative tree, and the $\gamma_i$'s tell us how far in the valuative tree the augmentation moves. 
\end{remark}

\begin{para}\label{rem:almost_stable}
	Note that if $\frb$ is almost stable and $f \in K[T]$ such that $(v_i(f))_{i \in I}$ does not stabilize, then necessarily $\lim_i v_i(f) = \infty$. 
	Indeed, let $f = \sum_{j = 0}^{n}a_j \psi_i^{j}$ be the $\psi_i$ expansion and write $v(\psi)$ for the constant value of $v(\psi_i)$.
	The function 
	\[
		p_{i}:[v(\psi), \gamma_i] \to \R\infty: r\mapsto [v, \psi_i, r](f) = \min_{i = 0}^{n} \{v(a_j) + j r\} 
	,\] 
	is the minimum of non-decreasing affine functions, and as such non-decreasing with decreasing slopes. 
	If $p_i$ would stabilize at some point then there would be $r_1< r_2 \in [v(\psi_i), \gamma_i]$ and $v_{r_1} := [v, \psi_i, r_1], v_{r_2} := [v, \psi_i, r_2]$, and $v \le v_{r_1} < v_{r_2} \le v_i$ such that $v_{r_1}(f) = v_{r_2}(f)$.
	But this contradicts that $(v_i(f))_{i \in I} $ does not stabilize, see \cite[Corollary~2.5(ii)]{nartMacLaneVaquieChains2021}.
	Thus $p_i$ has at least slope 1 at every point and thus $v_i(f) \ge v(f) + (\gamma_i - v(\psi)) $, which tends to $\infty$ as $\gamma_i $ tends to $\infty$. 
\end{para}

\begin{definition}\label{def:limit_augmentation_datum}
	A \emph{limit augmentation} of degree $m$ and log-radius $\mu$ is a tuple $\fra = (v, (\psi_i, \gamma_i)_{i \in I}, \phi, \mu)$ such that $\frb = (v, (\psi_i, \gamma_i)_{i \in I})$ is an essential continuous family, $\phi$ is a limit key polynomial for $\frb$, $\deg \phi = m$,  and $\mu \ge [v, \psi_i, \gamma_i](\phi)$ for all $i \in I$. 
	Sometimes we may write $(\frb, \phi, \mu)$ instead. 

	We define 
	\[
		[\fra]: K[T] \to \R: f = \sum_{i = 0}^{n} a_i \phi^{i} \mapsto \min_{i =0}^{n} \{[\frb](a_i) + i\cdot \mu \}
	,\]
	with $f = \sum_{i = 0}^{n}a_i \phi^{i}$ the $\phi$-expansion of $f$.
	This defines a semi-valuation, see \cite[Proposition~3.5]{nartMacLaneVaquieChains2021}, which we may sometimes denote as $[\frb, \phi, \mu]$.
	Moreover, we say that $\fra$ is \emph{almost stable} if $\frb$ is almost stable. 
\end{definition}

Note that if a limit augmentation $\fra = (v, (\psi_i, \gamma_i)_{i \in I}, \phi, \mu)$ is almost stable, then necessarily $\mu = \infty$ by \cref{rem:almost_stable}.
\begin{definition}\label{def:augmentation_datum}
	An \emph{augmentation } $\mathfrak{a}$ on $v$ is an ordinary augmentation, limit augmentation, or a stable continuous family on $v$. 
	Recall that in each case we denote the augmented semi-valuation by $[\mathfrak{a}]$.
\end{definition}
\begin{definition}\label{def:augmentation_chain}
	Let $v$ be a semi-valuation on $K[T]$. 
	An \emph{augmentation chain} on  $v$ is a sequence of augmentation data $(\mathfrak{a}_i)_{i = 1}^{n}$ with $n \in \Z_{\ge 0}\cup \{\infty\}$, $\fra_1$ an augmentation on $v$ and for all $i < n$, 
	\begin{itemize}
		\item $\fra_i$ is an ordinary or limit augmentation, 
		\item $\fra_{i + 1}$ is an augmentation on $[\fra_i]$.
	\end{itemize}
	In particular, a stable continuous family may only appear at the end of the chain. 

	Let $w$ be a semi-valuation on $K[T]$. 
	We say that an augmentation chain $(\mathfrak{a} _i)_{i = 1}^{n}$ \emph{approximates} $w$ if $w = [\fra_n]$ when $n \ne \infty$ or $w$ is the stable limit of $([\fra_i])_i$ when $n = \infty$. 
	We write this as $w = [(\fra_i)_{i = 1}^{n}]$.
\end{definition}
\begin{notation}
	Let $w, v$ be two semi-valuations on $K[T]$. Then we write 
	\if\thesis1
	 \begin{align*}
		 K[T]_{w = v} &= \{f \in K[T] \mid v(f) = w(f)\}, \\ K[T]_{w \ne v} &= \{f \in K[T] \mid v(f) \ne w(f)\} 
	 ,\end{align*}
	 \else
	 \[
		 K[T]_{w = v} = \{f \in K[T] \mid v(f) = w(f)\}, \quad K[T]_{w \ne v} = \{f \in K[T] \mid v(f) \ne w(f)\} 
	 ,\] 
	\fi 
	 and let $\Phi_{v, w}$ be the subset of monic polynomials of minimal degree in $K[T]_{v \ne w}$.
\end{notation}
If  $w \ge v$ one should think of $\Phi_{v,w}$ as the set of key polynomials on  $v$ that point towards $w$ in the valuative tree.  
\begin{definition}
	An augmentation chain $(\fra_i)_{i = 1}^{n}$ on $v_0$ is a \emph{MacLane-Vaquié chain} if it consists of only ordinary and limit augmentations and for every $i$ with $1 \le  i < n$
	\begin{itemize}
		\item when $\fra_i$ is ordinary, then $\deg \phi_{i-1} < \deg(\Phi_{v_{i-1}, v_i})$,
		\item when $\fra_i$ is limit, then  $\deg \phi_{i} = \deg (\Phi_{v_{i-1}, v_i})$ and $\phi_{i-1} \not\in  \Phi_{v_{i-1}, v_i}$,
	\end{itemize}
	where $v_j = [\fra_j]$, $\phi_j$ is the key polynomial of $\fra_j$ and $\phi_0$ is a key polynomial for $v_0$ of minimal degree. 
\end{definition}

\begin{theorem}[The main theorem of Vaquié, \protect{\cite[Thm~4.3]{nartMacLaneVaquieChains2021}}]\label{thm:vaquie_main}
	Any semi-valuation $w$ on $K[T]$ falls into exactly one of the following cases, denoted (a), (b), and (c) following \textit{loc.\ cit}.
	\begin{itemize}
		\item[(a)] There is a Gauss semi-valuation $v$ and a finite MacLane-Vaquié chain $(\mathfrak{a} _i)_i$ on $v$ such that $(\mathfrak{a}_i)_i$ approximates $w$. 
		\item[(b)] There is a Gauss valuation $v$ and a finite augmentation chain $(\mathfrak{a} _i)_{i = 1}^{n}$ on $v$ approximating $w$ such that $(\mathfrak{a} _i)^{n-1}_{i = 1}$ is a MacLane-Vaquié chain and $\mathfrak{a} _n$ is a stable continuous family augmentation with $\deg v_{n-1} = \deg \Phi_{v_{n-1}, w}$ and $\phi_{n-1} \notin \Phi_{v_{n-1}, w}$, with $\phi_{n-1}$ the key polynomial of $\fra_{n-1}$. 
		\item[(c)] There is a Gauss valuation $v$ and an infinite MacLane-Vaquié chain $(\mathfrak{a} _i)_i$ on $v$ such that $(\mathfrak{a}_i)_i$ approximates $w$. 
	\end{itemize}
\end{theorem}
\begin{para}
	These chains are unique up to some well described equivalence. Its precise formulation is a bit involved, and we will not need it going further. 
	The interested reader may consult \cite[Section~4.3]{nartMacLaneVaquieChains2021}. 
	Some remarks on how \Cref{thm:vaquie_main} relates to other classifications of semi-valuations on $K[T]$ are in order.
\end{para}

\begin{remark}
	If $K$ is discretely valued, the theorem relates to MacLane's original classification \cite[Thm 8.1]{maclane1936construction} as follows. 
	This classification constructs any semi-valuation as the last semi-valuation in a finite chain of ordinary augmentations or a point-wise (not necessarily stable) limit of an infinite chain of ordinary augmentations, called an inductive semi-valuation and a limit semi-valuation respectively. 
	
		Suppose that $w$ is an inductive semi-valuation, then clearly $w$ falls in case (a) and all $\fra_i$ are ordinary augmentations. 
		Suppose that $w$ is a limit semi-valuation, and $([v_{i -1}, \phi_i, \mu_i])_{i \in \N}$ the augmentation chain of ordinary augmentations. 
			\begin{itemize}
				\item If $(\deg \phi_i)_{i}$  is unbounded, then $\lim_i v_{i}$ is necessarily a stable limit. Thus, $w$ belongs to case (c).
				\item If $(\deg \phi_i)_i$ is bounded and stabilizes at $i = n$ then we may put all augmentations of top degree $(\fra_i)_{i = n}^{\infty}$ in a continuous family $\frb$. 
				\begin{itemize}
				\item If  $\frb$ is unstable, then $w$ belongs to case (b) by discreteness. 
				\item If $\frb$ is inessential, then we may replace it by one ordinary augmentation and $w$ belongs to case (a). 
				\item If $\frb$ is essential, we may choose a limit key polynomial  $\psi$ and then the chain $(\fra_0, \ldots, \fra_{n-1}, (\frb, \psi, \infty))$ is a chain approximating $w$. 
				Thus, $w$ belongs to case (a).
			\end{itemize}
	\end{itemize}
	Note that when $K$ is discretely valued, the only limit augmentation that may appear in a chain from \Cref{thm:vaquie_main} is an almost stable limit augmentation at the end of the chain.
	Note that if $K$ is complete, this cannot occur. 
\end{remark}
\begin{remark}
	If $K$ is complete and algebraically closed, then all key polynomials are linear. 
	As a result, all MacLane-Vaquié chains are of length $0$ on a Gauss semi-valuation $v(x - a) = \mu$ which is the supremum norm on the disk with center $a$ and log-radius $\mu$. 
	So case (c) does not occur. In case (a), the semi-valuation is simply the supremum norm on a disk. 
	And in case  (b) it is a limit of the supremum norms on a chain of decreasing disks. 
	Thus, the theorem reduces to the classical Berkovich classification theorem. 
\end{remark}
\begin{remark}
	If $K$ is complete, the semi-valuations on $K[T]$ are points in the Berkovich affine line $\aff_K^{\text{an}}$. 
	There is a correspondence between the cases of \Cref{thm:vaquie_main} and the type I, II, III, IV points on $\aff^{1, \text{an}}_K$. 
	Let $v$ be a semi-valuation on $K[T]$ and $(\fra_i)_{i = 1}^{n}$ be an augmentation chain approximating $v$ as in \Cref{thm:vaquie_main}.
	Let $x$ be the corresponding point of $\aff^{\text{an}}_K$.
	If $v$ falls in case (a), then $x$ is of type I, II, III if $\mu = \infty$, $\mu \in \Gamma_v \otimes\Q$, $\mu \not\in \Gamma_v \otimes \Q \cup \{\infty\} $ respectively with $\mu$ the log-radius of $\fra_n$. 
	If $v$ falls in case (b), then $x$ is of type IV. 
	Finally, if $v$ falls in case (c), then $x$ is of type IV or a non-rigid type I point, which means that $\mathcal{H} (x)$ is not a finite field extension of $K$. 
	The author does not know of a simple criterion to distinguish between type IV and I in this case.  
\end{remark}

\subsection{Steps and \texorpdfstring{$w$}{w}-optimal augmentations} \label{sec:steps_and_w_optimal}
\begin{para}
	In this section, we introduce the concepts of the step of an augmentation and $w$-optimal augmentations, which will be useful for the results and computations in the following sections. 
\end{para}

\begin{definition}\label{def:w-optimal}
	Let $w$ be a semi-valuation on $K[T]$ and let $\mathfrak{a} $ be an augmentation on $v$.
	We say that $\mathfrak{a} $ is \emph{$w$-optimal} if 
	\begin{itemize}
		\item $\mathfrak{a} = (v , \phi, \mu) $ is an ordinary augmentation, and $\mu = w(\phi)$,
		\item or $\mathfrak{a}  = (v, (\psi_i, \gamma_i)_{i \in I})$ is a continuous family, and each ordinary augmentation $(v, \psi_i, \gamma_i) $ is $w$-optimal,
		\item or $\mathfrak{a}  = (v, (\psi_i, \gamma_i)_{i \in I}, \phi, \mu)$ is a limit augmentation, and for each $i$ we have $\gamma_i = w(\psi_i)$ and $\mu = w(\phi)$. 
	\end{itemize}

	We say that an augmentation chain is $w$-optimal if each augmentation in the chain is $w$-optimal. 
\end{definition}

\begin{para}
	By \cite[Corollary~2.5.(2)]{nartMacLaneVaquieChains2021} we see that a chain $(\fra_i)_{i = 1}^{n}$ approximating $w'$ with $w' \le w$ is $w$-optimal if and only if for every $i < n$ we have that $\fra_i$ is $[\fra_i]_{i + 1}$-optimal and, if $n$ is finite, $\fra_n$ is $w$-optimal. 
\end{para}

\begin{proposition}\label{prop:optimal_chain}
	Let $w$ be any semi-valuation on $K[T]$, and let $(\mathfrak{a}_i)_{i = 1}^{n} $ be an augmentation chain approximating $w$ as in \Cref{thm:vaquie_main}.
	Then $(\mathfrak{a}_i)_{i = 1}^{n} $  is $w$-optimal.
\end{proposition}

\begin{proof}
	Let $\fra_i$ be an ordinary or limit augmentation with (limit) key polynomial $\phi_i$ and $\mu_i = [\fra_i](\phi_i)$. 
	If it is the last augmentation in the chain then $w = [\fra_i]$ and clearly $\mu_i = w(\phi_i)$. 
	Else, consider the next augmentation $\fra_{i+1}$.
	If it is an ordinary or limit augmentation then it's part of the MLV-chain, thus $\phi_i \not\in \Phi_{[\fra_i], [\fra_{i + 1}]}$. 
	So by \cite[Corollary 2.5]{nartMacLaneVaquieChains2021} we see that $\phi_i \not\in  \Phi_{[\fra_i], w}$. 
	On the other hand, if $\fra_{i + 1}$ is a continuous family, then we are in case (b) of \Cref{thm:vaquie_main}, and then also $\phi_i \not\in \Phi_{[\fra_i], w}$. 
	In all cases we can conclude that $w(\phi_i) = \mu_i $. 
	So we have already shown that all ordinary augmentations are $w$-optimal and the second half of the definition for limit augmentations. 
	It remains to show that the continuous families (and those in the limit augmentations) are $w$-optimal. 
	Let $\frb = (v', (\psi_i, \gamma_i)_{i \in I})$ be such a continuous family. 
	Take any  $i \in I$ and choose a $j > i$, which exists as $I$ has no maximal element.
	By the properties written between 3.3 and 3.4 in \cite{nartMacLaneVaquieChains2021} we see that $[v', (\psi_j, \gamma_j)](\psi_i) = \gamma_i$. 
	Thus, by \cite[Corollary 2.5]{nartMacLaneVaquieChains2021} we see that $\gamma_i = w(\psi_i)$. 
\end{proof}

\begin{para}
The $w$-optimal condition ensures that all the individual steps in the augmentation process are controlled in a way suitable for the computations in \Cref{sec:the_torsion_of_Omega_for_primitive_extensions}. 
We would also like our chains to start with as simple a semi-valuation as possible, like the Gauss valuation with $v = [v_K, T, 0]$. 
Then we have the nice property that $(K[T], v)^{\circ} = K^{\circ}[T]$.
As this is not always possible, we slightly relax this to the notion of a \emph{simple} Gauss valuation which is defined by $w = [v_K, x-a, \mu]$ with  $\mu \in \Gamma_K$. 
Let $S = b^{-1}(x - a)$. Then $(K[T], w)^{\circ} = (K[S], w)^{\circ} = K^{\circ}[S]$. 
\end{para}
\begin{corollary}\label{cor:optimal_simple_chain}
	Any semi-valuation $w$ on $K[T]$ can be approximated by a chain of $w$-optimal augmentations on a simple Gauss valuation. 
\end{corollary}
\begin{proof}
	By \Cref{thm:vaquie_main,prop:optimal_chain} there exists a $w$-optimal chain approximating $w$ on a Gauss semi-valuation $v_1 = [v_K, x-a, \mu_1]$ . 
	Choose any $r \in \Gamma_K$ with $r < \mu_1$ and let $v_0$ be the simple Gauss valuation defined by $v_0(x-a) = r$. 
	We can now prefix the chain with the ordinary augmentation $\fra_0 = (v_0, x-a, \mu_1)$, to obtain a $w$-optimal chain on $v_0$.
\end{proof}

\begin{definition}\label{def:step}
	Let $\mathfrak{a} $ be an augmentation on $v$. We define the \emph{step of} $\mathfrak{a}$ as follows.
	\begin{itemize}
		\item If $\mathfrak{a} = (v , \phi, \mu)$ is an ordinary augmentation, then \[\step \mathfrak{a}  = \mu - v (\phi) \]
		\item If $\mathfrak{a} = (v, (\psi_i, \gamma_i)_{i\in I})$ is a continuous family, then 
			\[\step \mathfrak{a}  := \lim_{i \in I}  (\gamma_i - v(\psi_i)).\]
		 \item If $\mathfrak{a}  = (v, (\psi_i, \gamma_i)_{i \in I}, \phi, \mu)$ is a limit augmentation with $\mathfrak{b}  = (v , (\psi_i, \gamma_i)_{i \in I}) = (\mathfrak{b} _i)$ the continuous family. 
			 Then \[\step \mathfrak{a} := \step \mathfrak{b}  + \lim_{i\in I}( \mu - [v, \psi_i, \gamma_i](\phi)).\]
	\end{itemize}
	If $(\fra_i)_{i = 1}^{n}$ is an augmentation chain then we define \[
		\step \left((\fra_i)_{i = 1}^{n}\right) = \sum_{i = 1}^{n} \step \fra_i
	.\] 
\end{definition}
\begin{para}
	Note that for a $\mathfrak{a} = (v, (\psi_i, \gamma_i)_{i\in I})$ the sequence $(\gamma_i - v(\psi_i))$ is increasing by \Cref{rem:continuous_family}.
\end{para}
\begin{para}
	Step measures how far in the valuative tree the augmentation has moved the valuation. 
	Actually, the step of a chain on $v$ approximating $w$ depends only on $v$ and $w$. 
	We won't prove this here as it will be an immediate corollary of \Cref{thm:computation_main}.
\end{para}

\section{The (log) cotangent complex associated to augmentations}\label{sec:the_torsion_of_Omega_for_primitive_extensions}
Throughout this section we take $(K, v_K)$ to be a non-trivially valued field, with no further assumptions. 

\begin{para}
	Let $w$ be a semi-valuation on the polynomial ring $K[T]$. 
	The goal of this section is to construct the cotangent complex, and its logarithmic analogue associated to the extensions  $(K[T]_{\ker w}, w)^{\circ} / K^{\circ}$ and  $(K[T]_{\ker w}, w)^{\circ} / K^{\circ}[\phi_0]$ for some appropriate linear polynomial $\phi_0$, using the data in a $w$-optimal augmentation chain approximating $w$.  
	The main idea is to use the augmentation chain to construct the ring $(K[T]_{\ker w}, w)^{\circ}$ from $K^{\circ}[\phi_0]$ using locally complete intersections, directed increasing unions and localizations. 
	These are all operations under which the behavior of (log) differentials is well understood.  
\end{para}

\subsection{Notation and helpful definitions} \label{sec:value_divisible_elements}
\subsubsection{Measuring torsion with content} \label{sec:measuring_torsion_with_content}

When $K$ is discretely valued and $M$ is a torsion $K^{\circ}$-module, then classically the ``size of $M$'' is measured by $\length M$. 
One of the main goals of this section is to understand the torsion part of $\Omega_{\mathcal{H} ^{\circ}(x) / K^{\circ}}$. 
However, the definition of $\length$ only makes sense when $K^{\circ}$ is Noetherian, i.e.\ when $K$ is discretely valued. 

In \cite[§~2.6]{temkinMetrizationDifferentialPluriforms2016}, the \emph{content} of a $K^{\circ}$-module is defined, which generalizes $\length$, even when  $K^{\circ}$ is non-Noetherian. 
Content, like the rest of \cite{temkinMetrizationDifferentialPluriforms2016}, is phrased in the multiplicative language of norms, whereas we use the additive language of valuations. 
For the convenience of the reader, we will rephrase the definition, and repeat and expand some basic properties in the additive language. 
\begin{definition}\label{def:cont}
	Let $M$ be a finitely presented $K^{\circ}$-module.
	There is an isomorphism $M \simeq \bigoplus_{i = 1}^{n} \frac{K^{\circ}}{(\alpha_i)}$ for some $\alpha_i \in K^{\circ}$.
	Then we define $\cont(M) = \sum_{i = 1}^{n} v_K(\alpha_i)$.
	Note that some $\alpha_i$ may be zero, in which case $\cont(M) = \infty$. 
	If $M$ is not finitely presented, we define 
	\[
		\cont(M) = \sup_{i} \cont(M_i)
	,\] 
	where $M_i$ runs through all finitely presented subquotients of $M$. 
\end{definition}

If $\Gamma_K = \gamma\cdot \Z, \gamma > 0$ and $M$ is finitely presented, then $\cont M = \gamma\cdot \length M$.
In particular, when $\gamma = 1$ then $\length$ and content agree, but content has the further advantage of rescaling automatically when $\gamma \ne 1$. 
Like  $\length$, content is also additive under exact sequences. 

\begin{lemma}\label{lem:cont_additive}
	Let  $\ldots \xrightarrow{f_{-2}} M_{-1} \xrightarrow{f_{-1}} M_0 \xrightarrow{f_0} M_1 \xrightarrow{f_1} \ldots$ be an exact sequence of $K^{\circ}$-modules. 
	Then \[
		\sum_{i \text{ even}} \cont(M_i) = \sum_{i \text{ odd}} \cont(M_i). 
	\]
\end{lemma}
\begin{proof}
	For every $i$ we consider the SES \[
	0 \to \im f_{i-1} \to M_i \to \im f_i \to 0
	.\] 
	Then by \cite[Theorem~2.6.7]{temkinMetrizationDifferentialPluriforms2016} we find $\cont(M_i) = \cont(\im f_{i-1}) + \cont (\im f_i)$. 
	Thus, \[
		\sum_{i \text{ even}} \cont(M_i) = \sum_{i} \cont (\im f_i) = \sum_{i \text{ odd}} \cont(M_i)
	.\] 
\end{proof}
\begin{lemma}\label{lem:cont_union}
	Suppose that a $K^{\circ}$-module $M$ is a filtered union of submodules $M_i$. Then, 
	\[
		\cont(M) = \lim_{i} \cont(M_i)
	.\] 
\end{lemma}
\begin{proof}
	This is \cite[Lemma 2.2.6]{temkinMetrizationDifferentialPluriforms2016}.
\end{proof}

\begin{para}
	Content does not commute with arbitrary filtered colimits $(M_i)_{i \in I}$ when it is not an increasing union, i.e. the maps $M_i \to M_j$ are not /injections. For example, let $M$ be any module with non-zero content and take the colimit of $(M)_{i \in \N}$ where all the maps are zero. 
However, if the kernels of the maps $M_i \to M_j$ tend to ``almost vanish'' when $i$ increases, then content still commutes. We make this precise in the following lemma.
\end{para}
\begin{lemma}\label{lem:cont_colim}
	Let $(M_i)_{i \in I}$ be a filtered system of $K^{\circ}$-modules with colimit $M$ and write $f_{ij}: M_i \to M_j$, $f_i: M_i \to M$ for the natural maps.  
	Then, \[
		\cont(M) = \lim_{i} \cont(M_i)
	,\] 
	if one of the following conditions holds
	\begin{enumerate}
		\item[(a)] $\lim_i \cont\left(\ker f_i\right)  = 0 $
		\item[(b)]  $\lim_{i} \lim_{j \ge i}\cont\left(\ker f_{ij}\right)  = 0 $ 
		\item[(c)] $(\cont M_i)_{i \in I}$ converges and $\lim_{i} \lim_{j \ge i}\cont\left(\coker f_{ij}\right)  = 0 $.
	\end{enumerate}
\end{lemma}
\begin{proof}
	We first show $(a)$. 
	By replacing $I$ by a cofinal subsystem, we may assume that $\cont(\ker M_i)$ is finite for all $i$. 
	For any $i \in I$ there is a SES \[
	0 \to \ker f_i \to M_i \to \im f_i \to 0
	.\] 
	Note that $M = \colim_{i} \im f_i$ is an increasing union. 
	Hence, \Cref{lem:cont_union,lem:cont_additive} shows that 
	\[
		\cont(M) = \lim_{i} \cont(\im f_i) = \lim_i \cont(M_i) - \cont (\ker f_i) = \lim_i \cont(M_i)
	.\] 

	We can deduce (b) from (a) by observing that for a fixed $i$, $(\ker f_{ij})_{j \ge i}$ is an increasing union resulting in $\ker f_i$.
	We find that $\lim_{j\ge i} (\cont (\ker f_{ij})) = \cont (\ker f_i)$ from \Cref{lem:cont_union}.

	Now we focus our attention on (c). 
	By passing to a cofinal subsystem of $I$ we assume that $\cont (M_i) < \infty$ for all $i$.
	Then there is an exact sequence 
	\[
	0 \to \ker f_{ij} \to M_i \to M_j \to  \coker f_{ij} \to 0
	.\] 
	Because the kernel and cokernel are subquotients of modules of finite content, they themselves have finite content. 
	By \Cref{lem:cont_additive}, we see that 
	\if\thesis1
	\[
		\resizebox{.98\textwidth}{!}{$\displaystyle\lim_{i} \lim_{j \ge i} \cont (\ker f_{ij})  = \left(\lim_{i} \lim_{j \ge i} \cont M_i - \cont M_j\right) + \left(\lim_{i} \lim_{j \ge i} \cont (\coker f_{ij})\right)$}
	.\] 
	\else
	\[
		\lim_{i} \lim_{j \ge i} \cont (\ker f_{ij})  = \left(\lim_{i} \lim_{j \ge i} \cont M_i - \cont M_j\right) + \left(\lim_{i} \lim_{j \ge i} \cont (\coker f_{ij})\right)
	.\] 
	\fi 
	Thus, we now deduce condition (b) from the fact that the sequence $(\cont M_i)_{i \in I}$ is Cauchy and the condition on the cokernels. 
\end{proof}

The following result is used without proof in \cite{temkinMetrizationDifferentialPluriforms2016}, for example in the proof of \cite[Theorem 5.2.11]{temkinMetrizationDifferentialPluriforms2016}.
For the sake of completeness, we provide a full proof. 
\begin{lemma}\label{lem:cont_invariance}
	Let $(L, v_L) / (K, v_K)$ be an extension of valued fields and let $M$ be a $K^{\circ}$-module. 
	Then 
	\begin{align*}
		\cont_{L^{\circ}}(M \otimes_{K^{\circ}} L^{\circ}) = \cont_{K^{\circ}} (M)
	.\end{align*}
\end{lemma}
\begin{proof}
	Suppose that $M$ is finitely presented. 
	Then we can find an isomorphism $M \simeq \bigoplus_{i = 1}^{n} \frac{K^{\circ}}{(\alpha_i)}$ and then clearly $M \otimes_{K^{\circ}}L^{\circ} \simeq\bigoplus_{i = 1}^{n} \frac{L^{\circ}}{(\alpha_i)}$. 
	Thus,
	\[
		\cont_{L^{\circ}}(M \otimes_{K^{\circ}} L^{\circ}) = \sum_{i = 1}^{n} v(\alpha_i) = \cont_{K^{\circ}} (M)
	.\] 

	Now suppose that $M$ is finitely generated. 
	If $K^{\circ}$ is Noetherian, then $M$ is also finitely presented, so the result follows from the previous case. 
	Now suppose that $K^{\circ}$ is not Noetherian, i.e. $\Gamma_K$ is dense in $\R$. 
	There is some surjection $(K^{\circ})^{n} \to M$ with kernel $A$.
	\[
	0 \to A \to (K^{\circ})^{n} \to M \to 0
	.\] 
	If $A$ is not a semilattice of $K^{n}$, i.e. $\dim_K A\otimes_{K^{\circ}} K < n$, then $\cont(M) =\infty$. 
	Similarly, $A \otimes_{K^{\circ}} L^{\circ}$ is not a semilattice in $L^{n}$ as \[
		\dim_L (A \otimes_{K^{\circ}} L^{\circ})\otimes_{L^{\circ}} L = \dim_L (A \otimes_{K^{\circ}} K)\otimes_{K} L < n
	.\] 
	So $\cont_{L^{\circ}} (M \otimes_{K^{\circ}}L^{\circ}) = \infty = \cont_{K^{\circ}} (M)$.
	Likewise, we see that if $A$ is a semilattice in $K^{n}$, then it is a semilattice in $L^{n}$. 
	Then we see that $[(K^{\circ})^{n}:A] = [(L^{\circ})^{n}: A \otimes_{K^{\circ}}L^{\circ}]$, and we conclude that $\cont_{L^{\circ}}(M \otimes_{K^{\circ}} L^{\circ}) = \cont_{K^{\circ}} (M)$ from \cite[Lemma 2.6.4]{temkinMetrizationDifferentialPluriforms2016}.
	Here $[\cdot, \cdot]$ denotes the \emph{index of lattices} as described in \cite[§~2.5.9]{temkinMetrizationDifferentialPluriforms2016}.

	Finally, let $M$ be any $K^{\circ}$-module. 
	Then $M$ is a filtered union of finitely generated submodules $M_i$.
	By \stacks{00DD} we see that $M \otimes_{K^{\circ}}L^{\circ}$ is also a filtered union of $(M_i \otimes_{K^{\circ}}L^{\circ})_i$ as $L^{\circ}$ is flat over $K^{\circ}$ \cite[Remark~6.1.12(ii)]{gabberAlmostRingTheory2003}. 
	So we may conclude from \Cref{lem:cont_union} that $\cont(M \otimes_{K^{\circ}} L^{\circ}) = \cont (M)$.
\end{proof}

\subsubsection{Constants and scalable rings} \label{sec:value-divisibility}

\begin{para}
	In an effort to make the computations in this section more clear, it will be useful to introduce some new notation and definition, in particular enlargements, constants, and scalable rings. 
	These are quite natural notions, but to the best of the author's knowledge not defined elsewhere.

	An \emph{enlargement} models an ordinary augmentation on the level of unit balls of rings. 
	Roughly speaking, an enlargement adds elements to the unit ball, which corresponds to increasing the valuation of these elements. 
	We also introduce the notions of \emph{constant} elements and \emph{scalable} rings, which can be thought of as an orthogonality condition between the value group and the residue ring. This ensures that we can always rescale an element of the ring to have any valuation of the value group. 
	This is necessary to ensure that some finitely presented rings, that approximate enlargements, are cut out by complete intersections. 

Throughout this subsection, let $(A, v)$ be a semi-valued $K$-algebra with value group $\Gamma$. 
\end{para}
\begin{definition}
	Let $(B, w)$ be a non-trivially semi-valued ring. 
	An element $a \in B$ is a \emph{constant} if $a \notin \ker w$ and for every $b \in B$, it holds that \[
		w(a) \le w(b) \implies a \mid b
	.\] 
	When $w(a) \ge 0$, this is equivalent to the statement that $B^{\ge w(a)}$ is generated by $a$ as a $B^{\circ}$ module. 
	The set of constants is denoted by  $\vd B$. 
	We say that $B$ is \emph{scalable} if for every $r \in w(B) \setminus \{\infty\} $, there exists a constant $a$ with $w(a) = r$, i.e., $w(\const B) = w(B) \setminus \{\infty\}$.
\end{definition}
In a scalable ring, any element can be ``rescaled'' to an element of any valuation by multiplying or dividing by by a constant with an appropriate value. 
\begin{example}
	Let $(K, v_K)$ be a discretely valued field with $\Gamma_K = \Z$, and consider the Gauss valuation $w = [v_K, T, 1 / 2]$.
	Then $T$ is not a constant, nor is there any other constant with valuation $1/2$. 
	Hence, $(K[T], w)$ is not scalable.
	However, $(K[T, T^{-1}], w)^{\circ}$ is scalable.
\end{example}

The notion simplifies a bit when we assume that the rings are $K$-algebras.

\begin{lemma}\label{lem:unit_value_divisor}
	If $A$ is a $K$-algebra, then \[
		\vd(A) = A^{\times}, \quad  \vd(A^{\circ}) = A^{\times } \cap A^{\circ}
	.\] 
\end{lemma}
\begin{proof}
	The inclusions $A^{\times } \subset \vd (A)$, and $A^{\times }\cap A^{\circ} \subset \vd (A^{\circ})$ are immediate. 
	Conversely, suppose that $a$ is a constant in $A^{\circ}$ or $A$.
	Then there is some  $b \in K^{\times}$ with $v(b) \ge v(a)$, so $a \mid b$ and thus $a \in A^{\times }$. 
\end{proof}

\begin{lemma}
	The $K$-algebra $A$ is scalable if and only if $A^{\circ}$ is scalable. 
\end{lemma}
\begin{proof}
	\ltr Let $r \in \Gamma_{v, \ge 0}$. 
	Then there exists a constant $a \in A$ with  $v(a) = r$. 
	It is easy to see that $a$ is also a constant as an element of $A^{\circ}$. 

	\rtl Let $r \in \Gamma_v$ and let $b \in K$ with $-r \le v(b) < \infty$. 
	Then $v(b) + r \ge 0$, so there is a constant $a$ of  $A^{\circ}$ with valuation $r + v(b)$. 
	Then $a$ is a unit in $A$, and $b^{-1} a$ is also a unit in $A$. 
	So $b^{-1}a$ is a constant of valuation $r$. 
\end{proof}
\begin{para}\label{para:constants_are_stable}
Let $w$ be a semi-valuation on $A$ with $w \ge v$. 
	We say that an element $a \in A$ is \emph{$w$-stable} if $w(a)= v(a)$. 
	All constants of $A$ and $A^{\circ}$ are $w$-stable. 
	Indeed, let $a$ be a constant, then $w(a) \ge v(a)$ and $w(a^{-1}) \ge v(a^{-1})$ from which equality follows. 
\end{para}

\begin{lemma}\label{lem:value_divisible_localisation}
	Suppose that $A$ is a scalable domain. 
	Let $S$ be a multiplicative subset of $A$. 
	Then $(S^{-1}A)^\circ = T^{-1}A^\circ$, where $T = \{a^{-1}s \mid s \in S, a \in \vd(A), v(a) = v(s)\}$ is a multiplicative set.
	Moreover, this ring is also scalable.
	In particular, the unit ball of a localization of $A$ is a localization of $A^{\circ}$. 
\end{lemma}
\begin{proof}
	Any element of $(S^{-1} A)^\circ$ can be written as $\frac ab$ with $a \in A, b \in S$ and $v(a) \ge v(b)$. 
Let $c$ be a constant with $v(c) = v(b)$.
Then \[
\frac ab = \frac{a / c}{b / c},\quad a / c \in A^\circ, b / c \in T
\]
So $a / b \in T^{-1}A^\circ$.

Conversely, any element of $T^{-1}A^\circ$ can be written as $\frac{a}{c^{-1}b}$, where $a \in A^\circ$, $b \in S$, $c \in \vd(A)$ with $v(c) = v(b)$.
Then $ac \in A$ and $b\in S$. So $\frac{a}{c^{-1} b}$ is also contained in $(S^{-1} A)^\circ$.
\end{proof}

\subsubsection{Enlargements} \label{sec:enlargements}
\begin{definition}\label{def:enlargement}
	Let $(A, v)$ be a scalable domain and $\phi \in A$ an irreducible element with $\mu \ge v(\phi)$. 
	Then the \emph{$(\phi, \mu)$-enlargement} of $A^{\circ}$ is
	\[
		[A^{\circ}, \phi, \mu] = A^{\circ}[a^{-1} \phi^{n} \mid a \in \vd(A), n \in \Z,  n\cdot \mu - v(a) \ge 0] 
	\]
\end{definition}

\begin{para}
	These enlargements are useful because they help us understand how unit balls change under ordinary augmentations, as is indicated by the following lemma.
\end{para}

\begin{lemma}\label{lem:inductive_model_unit_ball}
	Suppose that $\phi$ is an irreducible element of $A$ and $\mu \in \R$ with $\mu \ge v(\phi)$.
	Let $A^{\circ}, \phi, \mu$ be as in \Cref{def:enlargement} and let $w \ge v$ be a semi-valuation with $w(\phi) = \mu$.
	Suppose that 
	\if\thesis1
	\[
		\resizebox{\linewidth}{!}{$\displaystyle v': A \to \R: a \mapsto \max\left\{\min_{i = 0, \ldots, n} \{v(a_i) + \mu\cdot i\} \st  \forall a_0, \ldots, a_n \in A:\sum_{i = 0}^{n} a_i \phi^{i} = a \right\}$}
	\]
	\else
	\[
		v': A \to \R: a \mapsto \max\left\{\min_{i = 0, \ldots, n} \{v(a_i) + \mu\cdot i\} \st  \forall a_0, \ldots, a_n \in A:\sum_{i = 0}^{n} a_i \phi^{i} = a \right\}
	\]
	\fi 
	is well-defined and defines a semi-valuation on $A$. 
	Here, well-defined means that the maximum is achieved for each $a \in A$.  
	Then \[
		[(A, v)^{\circ}, \phi, \mu] = \begin{cases}
			(A[\phi^{-1}], v')^{\circ} & \text{if } \mu \ne \infty\\
			(A, v')^{\circ} & \text{if } \mu = \infty
		\end{cases} 	
	.\] 

	Moreover, if $(A, v)^{\circ}$ is scalable, then so is $([(A, v)^{\circ}, \phi, \mu], v')$.
\end{lemma}
\begin{proof}
	We only consider the case when $\mu \ne \infty$. The case when $\mu = \infty$ is similar. 
	Let $a \in \vd (A)$ and $n \in \Z$ such that $n\cdot \mu - v(a) \ge 0$. 
	Then $v'(a^{-1}\phi^{n} ) \ge n\cdot \mu - v(a) \ge 0$. 
	Thus, $a^{-1}\phi^{n} \in (A[\phi^{-1}], v')^{\circ}$. 
	So we can conclude that $[(A, v)^{\circ}, \phi, \mu] \subset (A[\phi^{-1}], v')^{\circ} $. 
	Conversely, suppose that $a \in (A[\phi^{-1}], v')^{\circ}$. 
	Then for some $a_n, \ldots, a_m \in A$, we can write  $a = \sum_{i =n}^{m} a_i \phi^{i} $ with $v'(a) \le v(a_i) + \mu \cdot i$ for all $i = n, \ldots, m$. 
	So $v(a_i) + \mu\cdot i \ge 0$. 
	Let $b_i\in A^{\circ}$ be a constant with $v(b_i) = v(a_i)$. 
	Then we can write $a_i \phi^{i} = (a_i / b_i) b_i \phi^{i}$ with $a_i / b_i \in A^{\circ}$ and $v(b_i) + \mu \cdot i \ge 0$.  
	So  $a \in [A^{\circ}, \phi, \mu]$.

	Now suppose that $(A, v)^{\circ}$ is scalable.
	Note that $v \le v' \le w$.  
	Clearly,  $v'([A^{\circ}, \phi, \mu]) = \left<\Gamma_v, \mu \right>_{ \ge 0}$. 
	Let $r \in \left<\Gamma_v, \mu \right>_{ \ge 0}$, with $ r\ne 0$. 
	Then $r$ can be written as $r' + i\cdot \mu$ for some $r \in \Gamma_v$ and $i \in \Z$. 
	Note that as $r < \infty$ either $i = 0$ or $\mu < \infty$. 
	If $i = 0$, then any constant $a$ of  $(A^{\circ}, v)$ is a unit in $A$. 
	So $a$ is also a unit of $A$ or $A[\phi^{-1}]$, thus it is also a constant of the unit ball $([(A, v)^{\circ}, \phi, \mu], v')$.
	It is also a constant because it is a unit in $A[\phi^{-1}]$. 
	So we conclude that $([(A, v)^{\circ}, \phi, \mu], v')$ is scalable.
\end{proof}

\subsection{Explicit descriptions of enlargements and their cotangent complexes} \label{sec:the_computations}

\begin{para}\label{par:cases_aleph_to_daleth}
	In this section we will describe \emph{enlargements} and their relative differentials/cotangent complexes. 
	This is a technical computation. 
	To aid the reader, we will first give an overview of the computation. 
	Suppose that $(A, v)$ is a scalable domain with value group $\Gamma$. 
	Let $\phi \in A$ be a prime element, and let $\mu = w(\phi) \ge v(\phi)$. 
	We will need to consider four cases which we will refer to using the first four letters of the Hebrew alphabet,
	\begin{itemize}
		\item[($\aleph$)] $\mu \in \Q\otimes\Gamma_{v}$,
		\item[($\beth$)] $\mu \not\in \Q \otimes \Gamma_{v}(\infty) $ and $\Gamma_v$ is not discrete,
		\item[($\gimel$)] $\mu \not\in \Q \otimes \Gamma_{v}(\infty)$ and $\Gamma_v$ is discrete,
		\item[($\daleth$)] $\mu  = \infty$.
	\end{itemize}
	Sometimes we may split up case ($\aleph$) into subcases depending on whether $\Gamma_v $ is discrete or not, denoted by ($\aleph$.d) and ($\aleph$.i) respectively.
\end{para}
\begin{para}
	For each step in the computation, we will need to consider the four cases separately, using analogous but slightly different arguments. 
	Each case has its own subtleties that need to be addressed individually. 
	Unfortunately, the author does not know of a way to unify the arguments. 
	Some notation is specific to a single case and will be introduced throughout the computation. 
	This notation remains consistent across all lemmas, propositions, theorems, proofs, etc., and will not be reintroduced each time. 
	\begin{enumerate}
		\item In \Cref{prop:inductive_model_form} we show that $[(A, v)^{\circ}, \phi, \mu]$ can be written as an increasing union of rings $A_i, i \in I$ where $I$ is a directed set and each $A_i$ has an explicit finite presentation as an $A^{\circ}$-algebra. 
		\item In \Cref{comp:quotient_by_gi} we show that the relations in the presentations of $A_i$ obtained in \Cref{prop:inductive_model_form} form a regular sequence. Thus, each $A_i$ is a complete intersection over $A^{\circ}$.
		\item In \Cref{comp:inductive_model_differentials} we use this to show that $\mathbb L_{A_i / A^{\circ}} \simeq \Omega_{A_i / A^{\circ}}[0]$ and give an explicit finite presentation of the latter module. 
	\item In \Cref{lem:lim_dets} we show that the valuation of the determinants of the presentations found in \Cref{comp:inductive_model_differentials} converge as $i$ increases to $\mu - v(\phi)$ up to a fudge term - analogous to the step of an augmentation. 
		These valuations of the determinants are $\cont(\Omega_{A_i / A} \otimes_{A_i} L^{\circ})$  for some suitable valued field $L$. 
	\item In \Cref{lem:enlargement_differentials}, we take a colimit over $i \in I$ to show that \[
			\mathbb L_{[A^{\circ}, \phi, \mu] / A^{\circ}} \derotimes L^{\circ} \simeq \left( \Omega_{[A^{\circ}, \phi, \mu] / A^{\circ}} \otimes L^{\circ}\right)[0] 
		\]
		and that the latter is a module of content $\mu - v(\phi)$ up to some fudge terms. 
		The main difficulty in this step is checking that content commutes with the colimit in this case. 
	\item In \Cref{sec:inducting_on_maclane-vaquie_chains} we interpret the results of \Cref{sec:the_computations} for augmentations on semi-valuations. 
		\Cref{lem:summary_ordinary_augmentation} and \Cref{lem:summary_limit_augmentation} we compute the cotangent complex associated to an ordinary and limit augmentation. 
		In \Cref{thm:computation_main} we use distinguished triangles of cotangent complexes to understand the situation for augmentation chains.  
	\item In Sections \ref{sec:the_log-cotangent_complex_of_an_enlargement} and \ref{sec:the_log-cotangent_complex_for_pure_trancendental_extensions} we prove analogous versions of these statements for the log cotangent complex. 
		We even get a more explicit description of the log differentials than for the (non-log) differentials. 
	\item	Finally, in \Cref{sec:extensions_of_valued_fields} and \Cref{sec:the_kahler_valuation_and_the_absolute_log_different} we use the results of our computations to prove some results about the (log) cotangent complex, (log) differents and (log) discrepancies.
	\end{enumerate}
\end{para}
\begin{notation}
	Let $a, b \in \Z_{\ge 0}$.
	Then we define $\overline{(a, b)}$ to be the minimal tuple $(c, d) \in (\Z_{\ge 0})^2$ such that \[
		ac - bd = \gcd(a, b)
	.\] 
\end{notation}
\begin{proposition}\label{prop:inductive_model_form}
	Suppose that $(A, v)$ is a scalable domain with value group $\Gamma$. 
	Let $\phi \in A$ be a prime element, and let $\mu = w(\phi) \ge v(\phi)$. 
	Let $\theta \in A$ be a constant with $v(\theta) = v(\phi)$.
	Then $[A^{\circ}, \phi, \mu]$ has the following explicit description as an increasing union of finitely presented $A^{\circ}$-algebras over a directed poset. 
	\begin{enumerate}
	\item[($\aleph$)] \textbf{Case $\mu \in \Gamma_{v, \Q}$:}
		Let $m \in \Z_{> 0}$ be the minimal element such that $\mu\cdot m \in \Gamma$.
		Let $\alpha$ be a constant with valuation $\mu\cdot m$. 
		Then $[A^\circ, \phi, \mu]$  is an increasing union 
		\begin{equation}
			[A^\circ, \phi, \mu] = \bigcup_{(\beta, \ell) \in I} A_{(\beta,\ell)}, \quad A_{(\beta, \ell)} = A^{\circ}[\alpha^{-1}\phi^{m}, \alpha \phi^{-m},  \beta^{-1} \phi^{\ell}]
		.\end{equation}
		Here 
		\if\thesis1
		\[
			I = \left\{(\beta,\ell) \in A^{\times }\times \Z \st \begin{aligned}&\gcd(\ell, m) =1, \text{ and}  \\ &\frac{e \mu m+ v(\phi)}{d}\le v(\beta) <  \ell \mu \text{ with } (d, e) = \overline{(\ell, m)} \end{aligned}\right\} 
		\]
		\else 
		\[
			I = \left\{(\beta,\ell) \in A^{\times }\times \Z \st \gcd(\ell, m) =1, \frac{e \mu m+ v(\phi)}{d}\le v(\beta) <  \ell \mu \text{ with } (d, e) = \overline{(\ell, m)} \right\} 
		\]
		\fi
		is the filtered system where $ (\gamma, k) \ge (\beta, \ell)$ if and only if $q \cdot (- v(\gamma)+ k \mu )  \le -v(\beta) + \ell \mu $, with $q$ the smallest positive integer such that $q\cdot k \equiv \ell \mod m$. 

		Moreover, for each $i:= (\beta, \ell) \in I$, the ring $A_i$ has a finite presentation as an  $A^{\circ}$-algebra, 
		\if\thesis1
		\begin{equation}
		\begin{aligned}
			A_i \simeq \frac{A^{\circ}[X, X^{-1}][Y_i]}{(g_i, h_i)}, \quad &X\mapsto \alpha^{-1} \phi^{m}, \quad Y_{i} \mapsto \beta^{-1}\phi^{\ell}, \\
			g_i = X^{-\ell}Y_i^{m} - \alpha^\ell\beta^{-m}, \quad &h_i =  \alpha ^{-e}\beta^{d}\theta^{-1} X^{-e}Y_i^{d} - \theta^{-1}\phi
		,\end{aligned}
		\end{equation}
		\else
		\begin{equation}
		\begin{aligned}
			A_i \simeq \frac{A^{\circ}[X, X^{-1}][Y_i]}{(g_i, h_i)},\quad &g_i = X^{-\ell}Y_i^{m} - \alpha^\ell\beta^{-m}, \quad h_i =  \alpha ^{-e}\beta^{d}\theta^{-1} X^{-e}Y_i^{d} - \theta^{-1}\phi\\
					&X\mapsto \alpha^{-1} \phi^{m}, \quad Y_{i} \mapsto \beta^{-1}\phi^{\ell}
		,\end{aligned}
		\end{equation}
		\fi 
		with $(d, e) = \overline{(\ell, m)}$.
		Note that the conditions on $v(\beta)$ in the definition of $I$ ensure that $v(\alpha ^{-e} \beta^{d} \theta^{-1}) \ge 0$ and $ v(\alpha \beta^{-m}) \ge 0$.
	\item[($\beth$)] \textbf{Case $\mu \not\in \Gamma_{v,\Q}(\infty)$ and $\Gamma_v$ is not discrete:}
		Then $[A^{\circ}, \phi, \mu]$ is an increasing union \[
			[A^{\circ}, \phi, \mu] = \bigcup_{(\alpha, \beta) \in I} A_{(\alpha, \beta) }, \quad A_{(\alpha, \beta)} = A^{\circ}[\alpha^{-1} \phi, \beta \phi^{-1}]
		,\] 
		where $I = \{(\alpha, \beta) \in A^{\times }\times A^{\times } \st v(\phi) \le v(\alpha) \le \mu \le v(\beta)\}$ is the directed set where $(\alpha', \beta') \ge (\alpha, \beta) \iff v(\alpha') \ge v(\alpha), v(\beta') \le v(\beta)$.

		Moreover,  for each $i = (\alpha, \beta)$, the ring $A^{\circ}[\alpha^{-1} \phi, \beta \phi^{-1}]$ admits the following presentation,  
		\begin{align*}
			\frac{A^{\circ}[X_i, Y_i]}{(g_i, h_i)}, \quad & g_i = X_iY_i - \alpha^{-1}\beta, h_i = \theta^{-1}\alpha X_i - \theta^{-1} \phi \\
								& X_i \mapsto \alpha^{-1}\phi, Y_i \mapsto \beta \phi^{-1}
		.\end{align*}
	\item[($\gimel$)] \textbf{Case $\mu \not\in \Gamma_{v,\Q} (\infty)$ and $\Gamma_v$ is discrete:}
		Suppose that $\Gamma_v = \gamma\cdot \Z, \gamma > 0$ and let $\pi$ be a constant with $v(\pi) = \gamma$. 
		Let $\frac{a_0}{b_0}, \frac{a_1}{b_1}, \frac{a_2}{b_2}, \ldots$ be the convergents (in reduced form) of the continued fraction expansion of $\gamma^{-1}(\mu - v(\phi))$. 

		Then $[A^{\circ}, \phi, \mu]$ is an increasing union \[
			[A^{\circ}, \phi, \mu] = \bigcup_{i \in \Z_{\ge 0}} A_{i}, \quad A_{i} = A^{\circ}[\pi^{-a_{2i}}(\theta^{-1}\phi)^{b_{2i}}, \pi^{a_{2i+1}}(\theta^{-1}\phi)^{-b_{2i+1}} ]
		.\] 

		Moreover, for each $i$, the ring $A_i$ admits the following presentation, 
		\begin{align*}
			\frac{A^{\circ}[X_i, Y_i]}{(g_i, h_i)}, \quad & g_i = X_i^{b_{2i+1}}Y_i^{b_{2i}} - \pi, h_i = X_i^{a_{2i+1}}Y_i^{a_{2i}} - \theta^{-1} \phi \\
								& X_i \mapsto \pi^{-a_{2i}}(\theta^{-1}\phi)^{b_{2i}}, Y_i \mapsto \pi^{a_{2i+1}}(\theta^{-1}\phi)^{-b_{2i+1}}
		.\end{align*}

	\item[($\daleth$)] \textbf{Case $\mu = \infty$:}
		Then $[A^\circ, \phi, \mu]$  is an increasing union 
		\begin{align*}
			[A^\circ, \phi, \mu] = \bigcup_{\alpha \in I} A^{\circ}[\alpha^{-1}\phi]
		,\end{align*}
		where $I = \{\alpha \in A^{\times } \mid v(\alpha) > v(\phi)\}$ with $\beta \ge \alpha \iff v(\beta) \ge v(\alpha)$.
		Moreover, for each $\alpha \in I$ the ring  $A^{\circ}[\alpha^{-1}\phi]$ admits the following finite presentation as an $A^{\circ}$-algebra, 
		\begin{align*}
			\frac{A^{\circ}[X_\alpha]}{g_\alpha}, \quad g_\alpha = \theta^{-1}\alpha X_\alpha - \theta^{-1}\phi, \quad X_\alpha \mapsto \alpha^{-1} \phi
		.\end{align*}
\end{enumerate}
\end{proposition}
Before we proceed to the proof, a few remarks are in order. 
\begin{remark}
	Although the proposition involves several cases, they all stem from the same idea. 
	We can write \[
		[A^{\circ}, \phi, \mu]  = A^{\circ}[\alpha^{-1}\phi^{n} \mid \alpha \in \vd(A), (-v(\alpha), n) \in C]
	,\] 
	where $C = \{(a, b) \in \Gamma_v \times \Z \mid b\cdot \mu + a \ge 0\} $ is a submonoid of $\Gamma_{v} \times \Z$. 
	The idea is to describe $[A^{\circ}, \phi, \mu]$ as an increasing union of explicit $A^{\circ}$-algebras by approximating $C$ with an increasing union of convenient submonoids.
\end{remark}
\begin{remark}\label{rem:discrete_filtered_system}
	The case ($\aleph$) simplifies in two different ways in the subcases ($\aleph$.d) and ($\aleph$.i) 
	In subcase ($\aleph$.d), i.e.\ when $\Gamma_v$ is discrete, the filtered system $I$ has a maximal element, $(\pi^{a}, \ell)$, with $a, \ell \in \Z_{\ge 0}$ such that $  \ell \mu -a v(\pi) = v(\pi) / m $, where $\pi$ is a constant with $v(\pi)$ a positive generator of $\Gamma_v$.
	So $[A^{\circ}, \mu, \phi]$ is itself is a finitely presented $A^{\circ}$-algebra, and as we will later see in \Cref{comp:inductive_model_differentials} even a complete intersection over $A^{\circ}$. 

	In subcase ($\aleph$.i), i.e.\ when $\Gamma$ is not discrete, the limit simplifies in another way. 
	In this case $I' = \{(\beta, \ell) \in I \mid \ell = 1\} $ is a cofinal subsequence of $I$, and we may instead consider the limit over $I'$. 
	For $(\beta, 1), (\gamma, 1) \in I'$ we see that simply $(\gamma, 1) \ge (\beta, 1) \iff v(\gamma) \ge v(\beta)$. 
\end{remark}
\begin{remark}
	Assume that $ (K,v_K)$ is complete, $A = K[T]$,  $v = [v_K, T, 0]$  and  $\phi = T$. 
	If $\mu \not\in \Gamma_{v,\Q}(\infty)$ then $w = [v, \phi, \mu]$ is a type III point in $\aff_K^{1, \text{an}}$. 
	Then in both cases (regardless of whether $\Gamma$ is discrete or not) the rings $A_i$ are the intersections of $K[T, T^{-1}]$ with the unit balls of affinoid algebras corresponding to annuli shrinking around the type III point $w$.  
	Intuitively, we can think of this as approximating $(K[T, T^{-1}], w)^{\circ}$ with ``affine models'' of shrinking annuli around $w$.
\end{remark}

\begin{proof}[Proof of \Cref{prop:inductive_model_form}]
In all cases, it is clear that the indexing set, either $I$ or $\Z_{\ge 0}$, is a directed poset. 
And in cases ($\aleph$), ($\beth$), and ($\daleth$), it is also clear that the union is increasing and equals $[A^{\circ}, \phi, \mu]$. 

In case $(\gimel)$, i.e.\ when $\mu \not\in \Gamma_{\Q}(\infty)$ and $\Gamma$ is discrete, this needs an argument. 
We know from the classical theory of continued fractions that $a_{2i} b_{2i + 1} - a_{2i+1}b_{2i} = -1$. 
From the classical theory of finitely generated monoids, we know that $\left<(-a_{2i}, b_{2i}), (a_{2i+1}, -b_{2i+1})  \right>$ is a simplicial cone in $\Z^2$ and thus precisely equal to the cone 
\if\thesis1
\[
C_i := \left\{(a, b) \in \Z^2 \st b\cdot \frac{a_{2i}}{b_{2i}} - a \ge 0,\quad b\cdot \frac{a_{2i + 1}}{b_{2i + 1}} -a\ge 0 \right\} 
.\]
\else
\[
C_i := \left< (-a_{2i}, b_{2i}), (a_{2i+1}, -b_{2i+1})  \right> = \left\{(a, b) \in \Z^2 \st b\cdot \frac{a_{2i}}{b_{2i}} - a \ge 0,\quad b\cdot \frac{a_{2i + 1}}{b_{2i + 1}} -a\ge 0 \right\} 
.\]
\fi 
Let $\psi = \theta^{-1}\phi$. Then we see that 
\[
	A_i = A^{\circ}[\pi^{-a_{2i}} \psi^{b_i}, \pi^{a_{2i+1}}\psi^{-b_{2i+1}}] = A^{\circ}[\pi^{a}\psi^{b} \mid (a, b) \in C_i]
,\]
from which we see that the sequence $A_i$ is increasing because the sequence $C_i$ is increasing. 

Next, consider $\alpha^{-1}\phi^{n}$ with $n \in \Z, \alpha \in \vd(A)$ and $n\mu - v(\alpha) \ge 0$.
The inequality is in fact strict because $\mu \not\in \Gamma_{\Q}$.
Let $b = v(\alpha \theta^{-n}) / v(\pi)$.
Then we can write $\beta \pi^{b} = \alpha\theta^{-n}$ with $\beta$ a constant with $v(\beta) = 0$.
We may rewrite this as $\beta \pi^{b}\psi^{n}$ with $n\gamma^{-1}(\mu -v(\phi)) - b > 0$. 
Because the continued fraction convergents $a_i / b_i$ approximate $\gamma^{-1}(\mu - v(\phi))$ we know that for $i \gg 0$, $n \frac{a_{2i}}{b_{2i}} - b \ge 0$ and $n\frac{a_{2i+ 1}}{b_{2i + 1}} - b \ge 0$. 
Thus, $(n, b) \in C_i$. 
So, there are $\ell, k \in \Z_{\ge 0}$ such that $(n, b) = \ell (a_{2i}, b_{2i}) + k (a_{2i + 1}, b_{2i + 1})$. 
Then we see that
\[
	(\pi^{-a_{2i}}(\theta^{-1}\phi)^{b_{2i}})^{\ell} (\pi^{a_{2i+1}}(\theta^{-1}\phi)^{-b_{2i+1}})^{k} = \pi^{b} (\theta^{-n}\psi^{n})
,\] 
from which we may conclude that $\alpha^{-1}\phi^{n} = \beta \pi^{b}\psi^{n} \in A_i$.
This shows that the union of the rings $A_i$ is indeed equal to $[A^{\circ}, \phi, \mu]$. 

\medskip
This leaves us to show that the given presentations for the $A_i$'s are correct in each case. 
\begin{itemize}
	\item[($\aleph$)] \textbf{Case $\mu \in \Gamma_{\Q}$:}
			Fix an $i = (\beta, \ell) \in I$. 
			Let $f: \frac{A^{\circ}(X)[Y_i]}{(g_i, h_i)} \to A^{\circ}[\alpha^{-1}\phi^{m}, \alpha \phi^{-m},  \beta^{-1} \phi^{\ell}]$ be the induced morphism. 
			It is easy to see that $g_i, h_i$ vanish, thus $f$ is well-defined and that $f(X) = \alpha^{-1}\phi^{m}, f(X^{-1}) = \alpha\phi^{-m}, f(Y_{i}) = \beta^{-1}\phi^{\ell}$ making $f$ surjective.

			We will show that $f$ is also injective. 
			Let $x$ be in the domain of $f$ such that $f(x) = 0$. 
			Let $x = \sum_{j =1}^{n_x}a_j X^{n_j}Y^{m_j}$ be a representative of $x$ with a minimal number of terms. 
			If any of  the coefficients $a_j$ are divisible by $\phi$ in $A$ we may factor out the largest power of $\theta^{-1}\phi$ and use relation $h_i$ to substitute $\theta^{-1}\phi$ with $\alpha ^{-e}\beta^{d} \theta^{-1} X^{-e}Y^{d}_i$. 
			Thus, we may assume all $a_j$'s are indivisible by $\phi$ in $A$.  
			We further assume that all $m_j \le m$ as we can use the relation  $g_i$ to substitute $Y_i^{m}$ for $\alpha\beta^{-m}X^{\ell}$. 
			Note that these substitutions do not change the number of terms, so we still have a representation with a minimal number of terms. 

			No two terms share the same monomial as we chose a representation with a minimal number of terms. 
			Because $\gcd(\ell, m) = 1$ and  $0 \le m_j < m$ we know that the $\phi$-adic order of $f(a_j X^{n_j}Y^{m_j})$  
			\[
				v_\phi(f(a_j X^{n_j}Y_i^{m_j})) = n_j\cdot m + m_j\cdot \ell
			,\] 
			is different for each $j  = 1, \ldots, n_x$. 
			Thus, \[
				v_{\phi}(x) = \inf_{j = 1, \ldots, n_x} \{n_j\cdot m + m_j\} = \infty 
			,\] 
			which means that $n_x = 0$, thus $x = 0$.

			The arguments for the other cases are analogous and can be treated in less detail. 	
		\item[($\beth$)] \textbf{Case $\mu \not\in \Gamma_{\Q}$, $\mu \ne \infty$, and $\Gamma$ is not discrete:}
			Fix an $i = (\alpha, \beta)\in I$ and let $f: \frac{A^{\circ}[X_i, Y_i]}{g_i, h_i} \to A^{\circ}[\alpha^{-1} \phi, \beta \phi^{-1}]$  be the induced morphism. 
			Then, like in the previous case, we see that $f$ is well-defined and surjective. 

			To show injectivity, take an $x$ in the domain of $f$ with $f(x) = 0$, and let  $x = \sum_{j = 1}^{n_x}a_jX_i^{n_j}Y_i^{m_j}$ be a representative with a minimal number of terms.
			Using the relation $h_i$ we may assume that none of the coefficients $a_j$ are divisible by $\phi$.
			Further, using relation $g_i$ we may assume that all monomials are either powers of $X_i$ or powers of $Y_i$, i.e. for every  $j$ either  $n_j = 0$ or $m_j = 0$. 
			As the representation has a minimum number of terms we know that no two terms share the same monomial. 
			So we see again that the $\phi$-adic order of $f(a_j X_i^{n_j}Y_i^{m_j}) = n_j - m_j$ is different for each $j$. 
			Thus, \[
				v_\phi(x) = \inf_{j = 1, \ldots, n_x} \{n_j - m_j\}  = \infty
			.\] 
			Therefore, $n_x = 0$, and hence $x = 0$. 

		\item[($\gimel$)] \textbf{Case $\mu \not\in \Gamma_{\Q}$, $\mu \ne \infty$, and $\Gamma$ is discrete:}
			Fix an $i \in \Z_{\ge 0}$ and let $f: \frac{A^{\circ}[X_i, Y_i]}{g_i, h_i} \to A^{\circ}[\pi^{-a_{2i}}(\theta^{-1}\phi)^{b_{2i}}, \pi^{a_{2i+1}}(\theta^{-1}\phi)^{-b_{2i+1}} ]$ be the induced morphism. 
			Like before, we see that $f$ is well-defined and surjective, and we take $x$ in the domain with $f(x) = 0$. 
			Let $x = \sum_{j = 1}^{n_x}a_jX_i^{n_j}Y_i^{m_j}$ be a representative with a minimal number of terms. 
			Again using relation $h_i$ we may assume that none of the coefficients $a_j$ are divisible by $\phi$. 
			Write $s_j = n_j b_{2i} - m_j b_{2i + 1}$.
			Then using relation $g_i$ we may assume that each pair of exponents $(n_j, m_j)$ is the unique minimal positive solution to the Diophantine equation $n_j b_{2i} - m_j b_{2i + 1} = s_j$. 
			As we assumed that  $n_x$ is minimal we know that no two pairs $(n_j, m_j)$ are the same and thus all $s_j$'s are distinct. 
			We see that $v_\phi(f(a_j X_i^{n_j}Y_i^{m_j})) = s_j$ and thus \[
				v_\phi(x) = \inf_{j = 1, \ldots, n_x} \{s_j\}  = \infty
			\]
			from which we conclude that $n_x =0$, and hence $x = 0$.

	\item[($\daleth$)]  \textbf{Case  $\mu = \infty$:}
		This is a much simpler version of the argument of the previous cases.
	\end{itemize}
\end{proof}

\begin{computation}\label{comp:quotient_by_gi}
	Continue with the context of \Cref{prop:inductive_model_form}. 
	Suppose that we are in case ($\aleph$), ($\beth$), or ($\gimel$), and let $A_i$ be one of the rings with presentation $A^{\circ}[X_i, Y_i] / (g_i, h_i)$  or $A^{\circ}[X_i^{\pm 1}, Y_i] / (g_i, h_i)$. 
	Then the ring $A^{\circ}[X_i, Y_i] / (g_i)$ or   $A^{\circ}[X_i^{\pm 1}, Y_i] / (g_i)$ has a simple explicit form and in particular is a domain. 
	\begin{itemize}
		\item[($\aleph$)] \textbf{Case $\mu \in \Gamma_{\Q}$:}
			Let $i = (\beta, \ell) \in I$ and let $B_i = A^{\circ}[\alpha^{-1}T^{m}, \alpha T^{-m}, \beta^{-1}T^{\ell}]$, where $T$ is a free variable. 
			Then $f: A^{\circ}[X_i^{\pm 1}, Y_i] / (g_i) \to B_i$ with $X_i \mapsto \alpha^{-1} T^{m}, Y_i \mapsto \beta^{-1}T^{\ell}$ is an isomorphism.

			The proof strategy is the same as in the proof of \Cref{prop:inductive_model_form}.
			Clearly $f$ is well-defined and surjective. 
			Let  $x \in \ker f$ and choose a representation $\sum_{j = 1}^{n} a_j X^{n_j}Y^{m_j}$ with a minimal number of terms. 
			Then we may use relation $g_i$ to assume that $m_j \le m$ for all $j = 1, \ldots, n$ by substituting $Y_i^{m}$ for $\alpha \beta^{-m}X^{\ell}$. 
			Then all terms have different $T$-adic orders, and thus $n = 0$ and $x= 0$.

		\item[($\beth$)] \textbf{Case $\mu \not\in \Gamma_{\Q}(\infty)$, and $\Gamma$ is not discrete:}
			Let $i = (\alpha, \beta) \in I$ and let  $B_i = A^{\circ}[\alpha^{-1}T, \beta T^{-1}]$. 
			Then $f: A^{\circ}[X_i, Y_i] / (g_i) \to B_i$ with $X_i \mapsto \alpha^{-1}T, Y_i \mapsto \beta T^{-1}$ is an isomorphism. 
			Again it is well-defined and surjective. 
			Let $x = \sum_{j = 1}^{n} a_j X^{n_j}Y^{m_j} \in \ker f$ be a minimal representative. 
			Using the relation $g_i$ we may assume that for all $j$ either $n_j =  0$ or $m_j = 0$. 
			Hence, the  $T$-adic order of each term differs and so $n = 0$ and $x= 0$. 
		\item[($\gimel$)] \textbf{Case $\mu \not\in \Gamma_{\Q}(\infty)$ and $\Gamma$ is discrete:}
			Let $i \in \Z_{\ge 0}$ and let $B_i = A^{\circ}[\pi^{-a_{2i}}T^{b_{2i}}, \pi^{a_{2i+1}}T^{-b_{2i+1}}]$.
			Then $f: A^{\circ}[X_i, Y_i] / (g_i) \to B_i$ with $X_i \mapsto \pi^{-a_{2i}}T^{b_{2i}}, Y_i \pi^{a_{2i+1}}T^{-b_{2i+1}}$ is an isomorphism.
			Again it is well-defined and surjective. 
			Let $x = \sum_{j = 1}^{n} a_j X^{n_j}Y^{m_j} \in \ker f$ be a minimal representative. 
			Write $s_i = n_j b_{2i} - m_jb_{2i + 1}$. 
			Using the relation $g_i$ we may assume that $(n_j, m_j)$ is the minimal positive solution to the Diophantine equation $n_j b_{2i} - m_j b_{2i+1}$. 
			The $T$-adic order of each term is $s_j$, and they all differ. Hence, $x = 0$. 
	\end{itemize}
\end{computation}
\begin{para}\label{para:Ais_are_CI}
	\Cref{comp:quotient_by_gi} shows that the sequences $g_i, h_i$ from \Cref{prop:inductive_model_form} in case ($\aleph$), ($\beth$), or ($\gimel$) are regular. 
	Indeed, write $A^{\circ}[\underline X]$ for $A^{\circ}[X_i, Y_i]$ or $A^{\circ}[X_i^{\pm 1}, Y_i]$. 
	Then $g_i$ is not a zero divisor, as $A^{\circ}[\underline X] $ is a domain. 
	Then the element $h_i$ in $B_i = A^{\circ}[\underline X] / (g_i)$ is also not a zero divisor, as $B_i$ is a domain. 
	In case ($\daleth$), there is only one relation in the presentation which is clearly a regular element. 
	In conclusion, $A_i$ is always a complete intersection over $A^{\circ}$. 
\end{para} 
\begin{computation}\label{comp:inductive_model_differentials}
	Let $(A, v), w, \Gamma, \phi, \mu, \theta$ be as in \Cref{prop:inductive_model_form}. 
	We can compute the differentials $\Omega_{A_i / A^{\circ}}$ as an explicit finitely presented $A_i$-module and moreover we show that all the higher cotangent groups vanish, i.e.\ $\mathbb L_{A_i / A^{\circ}} \simeq \Omega_{A_i / A^{\circ}}[0]$.
	\begin{enumerate}
		\item[($\aleph$)] \textbf{Case $\mu \in \Gamma_\Q$}. 
			We introduce a dummy variable $Z = X^{-1}$.
			Then we can write the representation as 
			\begin{align*}
				\frac{A^{\circ}[X, Z, Y_i]}{(XZ - 1, g_i, h_i)}, \quad g_i = Z^{\ell} Y_{i}^{m} - \alpha \beta^{-m}, \quad h_i = \alpha ^{-e}\beta^{d}\theta^{-1} Z^{e}Y_i^{d} - \theta^{-1} \phi
			.\end{align*}
			Write $\mathcal{I}  = (XZ - 1, g_i, h_i)$.
			By \cref{para:Ais_are_CI}, $XZ - 1, g_i, h_i$ is a regular sequence. 
			So we know that $\mathbb L_{A_i / A^{\circ}[X, Z, Y_i]} \simeq \mathcal{I}  / \mathcal{I}^2 [-1] $, and $\mathcal{I} / \mathcal{I}^2$ is a free $A_i$ module with generators $XZ-1, g_i, h_i$ \stacks{08SJ}. 

			Consider the ring morphisms $A^{\circ} \to A^{\circ}[X, Z, Y_i] \to A_i$ and consider the corresponding LES of cotangent modules
			\if\thesis1
			\begin{equation}\label{eq:fin_pres_omega_type_2}
				\begin{aligned}
					\ldots \to 0  \to 0 \to H_1(\mathbb L_{A_i / A^{\circ}}) &\to \\
					\underbrace{A_i \cdot \left<XZ -1, h_i, j_i \right>}_{H_1(\mathbb L_{A_i / A^{\circ}[Z, X, Y_i]})}\xrightarrow{\dd_i}\underbrace{A_i\cdot \left<\dd X, \dd Z, \dd Y \right>}_{\Omega_{A^{\circ}[Z, X, Y_i]}} \to \Omega_{A_i / A^{\circ}} &\to 0,
				\end{aligned}
			\end{equation} 
			\else
			\begin{equation}\label{eq:fin_pres_omega_type_2}
				\ldots \to 0 \to H_1(\mathbb L_{A_i / A^{\circ}}) \to \underbrace{A_i \cdot \left<XZ -1, h_i, j_i \right>}_{H_1(\mathbb L_{A_i / A^{\circ}[Z, X, Y_i]})}\xrightarrow{\dd_i} \underbrace{A_i\cdot \left<\dd X, \dd Z, \dd Y \right>}_{\Omega_{A^{\circ}[Z, X, Y_i]}} \to \Omega_{A_i / A^{\circ}} \to 0  
			,\end{equation} 
			\fi
			with $\dd_i$ represented by the matrix: \begin{align*}
				\jac(\dd_i) &= \begin{pmatrix} Z & X & 0 \\
				0 & \ell Z^{\ell-1} Y_i^{m} & m Z^{\ell}Y_i^{m-1} \\
			0 & \alpha ^{-e} \beta^{d}\theta^{-1} e Z^{e-1}Y_i^{d} &  \theta^{-1}\alpha ^{-e} \beta^{d} d Z^{e}Y_i^{d -1}\end{pmatrix}  \\
			\det \jac(\dd_i) &= \alpha ^{-e}\beta^{d}\theta^{-1} (ld-em) Z^{\ell +e}Y_i^{m +d-1} \\
				      &=  Y_i^{-1}(\alpha ^{-e}\beta^{d} \theta^{-1} Z^{e}Y^{d})(Z^{\ell}Y_i^{m}) \\
				       &= (Y_i)^{-1}( \theta^{-1} \phi)Z^{\ell}Y_i^{m} \nr \label{eq:det_type_2}
		.\end{align*}
		In particular $\dd_i$ is injective and thus $H_1(\mathbb L_{A_i / A^{\circ}}) = 0$ and all higher cotangent groups vanish.

	\item[($\beth$)] \textbf{Case $\mu \not\in \Gamma_{\Q}(\infty)$, and $\Gamma$ is not discrete:}
		Write $\mathcal{I} := (g_i, h_i)$. 
		Again, $h_i, g_i$ is a regular sequence, and $\mathbb L_{A_i / A^{\circ}[X_i, Y_i]}  \simeq \mathcal{I}  / \mathcal{I}^2 [-1]$.
		Taking the LES yields 
		\if\thesis1
		\begin{equation}\label{eq:fin_pres_omega_type_3_indisc}
			\begin{aligned}
				\ldots \to 0 \to 0 \to H_1(\mathbb L_{A_i / A^{\circ}}) &\to \\
				\underbrace{A_i \cdot \left<h_i, g_i \right>}_{H_1(\mathbb L_{A_i / A^{\circ}[X_i, Y_i]})}\xrightarrow{\dd_i} \underbrace{A_i\cdot \left<\dd X_i, \ \dd Y_i \right>}_{\Omega_{A^{\circ}[X_i, Y_i]}} \to \Omega_{A_i / A^{\circ}} &\to 0, 
			\end{aligned}
		\end{equation} 
		\else
		\begin{equation}\label{eq:fin_pres_omega_type_3_indisc}
				\ldots \to 0 \to H_1(\mathbb L_{A_i / A^{\circ}}) \to \underbrace{A_i \cdot \left<h_i, g_i \right>}_{H_1(\mathbb L_{A_i / A^{\circ}[X_i, Y_i]})}\xrightarrow{\dd_i} \underbrace{A_i\cdot \left<\dd X_i, \ \dd Y_i \right>}_{\Omega_{A^{\circ}[X_i, Y_i]}} \to \Omega_{A_i / A^{\circ}} \to 0  
		,\end{equation} 
		\fi
		with $\dd_i$ represented by the matrix
		\begin{align*}
			\jac(\dd_i)  &= \begin{pmatrix} Y_i & X_i \\ \theta^{-1}\alpha & 0  \end{pmatrix}  \\
			\det \jac(\dd_i) &= -\theta^{-1}\alpha X_i  \\ 
			&= -\theta^{-1}\phi \nr \label{eq:det_type_3_disc_1}
		.\end{align*}
		Thus, $\dd_i$ is injective and thus $H_1(\mathbb L_{A_i / A^{\circ}}) = 0$ and all higher cotangent groups vanish. 

	\item[($\gimel$)] \textbf{Case $\mu \not\in \Gamma_{\Q}(\infty)$ and $\Gamma$ is discrete:}
		Write $\mathcal{I} := (g_i, h_i) $, which is generated by a regular sequence. 
		We get almost the same LES as in the previous case:
		\if\thesis1
		\begin{equation}\label{eq:fin_pres_omega_type_3_disc}
			\begin{aligned}
				\ldots \to 0 \to 0 \to H_1(\mathbb L_{A_i / A^{\circ}}) &\to \\
				\underbrace{A_i \cdot \left<h_i, g_i \right>}_{H_1(\mathbb L_{A_i / A^{\circ}[X_i, Y_i]})}\xrightarrow{\dd_i} \underbrace{A_i\cdot \left<\dd X_i, \ \dd Y_i \right>}_{\Omega_{A^{\circ}[X_i, Y_i]}} \to \Omega_{A_i / A^{\circ}} &\to 0,  
			\end{aligned}
		\end{equation} 
		\else
		\begin{equation}\label{eq:fin_pres_omega_type_3_disc}
				\ldots \to 0 \to H_1(\mathbb L_{A_i / A^{\circ}}) \to \underbrace{A_i \cdot \left<h_i, g_i \right>}_{H_1(\mathbb L_{A_i / A^{\circ}[X_i, Y_i]})}\xrightarrow{\dd_i} \underbrace{A_i\cdot \left<\dd X_i, \ \dd Y_i \right>}_{\Omega_{A^{\circ}[X_i, Y_i]}} \to \Omega_{A_i / A^{\circ}} \to 0  
		,\end{equation} 
		\fi 
		with $\dd_i$ represented by the matrix
		\begin{align*}
			\jac (\dd_i) &= \begin{pmatrix} b_{2i + 1}X_i^{b_{2i + 1}-1} Y_i^{b_{2i}} & b_{2i}X_i^{b_{2i + 1}} Y_i^{b_{2i}-1} \\
			 a_{2i + 1}X_i^{a_{2i + 1}-1} Y_i^{a_{2i}} & a_{2i}X_i^{a_{2i + 1}} Y_i^{a_{2i}-1} \end{pmatrix}  \\
			 \det\jac(\dd_i) &= (b_{2i + 1} a_{2i} - b_{2i} a_{2i + 1}) (X_i Y_i)^{-1}(X_i^{a_{2i + 1} + b_{2i + 1}}Y_i^{a_{2i} + b_{2i}}) \\
					       &= \left(X_iY_i\right)^{-1}\pi \theta^{-1}\phi \nr \label{eq:det_type_3_disc_2}
		.\end{align*}
		Also, $\dd_i$ is injective and thus $H_1(\mathbb L_{A_i / A^{\circ}}) = 0$ and all higher cotangent groups vanish. 
	\item[($\daleth$)] \textbf{Case $\mu = \infty$:}
		As $A_i$ and $A^{\circ}$ are domains we know that $\theta^{-1}\alpha X_i - \alpha^{-1}\phi$ is not a zero divisor and thus a regular sequence. 
		The same argument as the previous cases gives that 
		\begin{equation}\label{eq:fin_pres_omega_type_1}
			\ldots \to 0 \to 0 \to A_i\cdot \left<g_i \right> \xrightarrow{\dd_i} A_i \cdot \left<\dd X_i \right> \to \Omega_{A_i / A^{\circ}} \to 0
		,\end{equation}
		with $\dd_i$ represented by the  $1 \times  1$ matrix
		\begin{equation}\label{eq:det_type_1}
			\jac(\dd_i) = \theta^{-1}\alpha 
		.\end{equation} 
		Likewise, we see that $\mathbb L _{A_i / A} \simeq \Omega_{A_i / A}[0]$. 
	\end{enumerate}
\end{computation}

Using \Cref{comp:inductive_model_differentials}, we obtain a few results. 

\begin{lemma}\label{lem:lim_dets}
	Continue with the context of \Cref{comp:inductive_model_differentials}.
	In all four cases we have that
	\[
		\lim_{i \in I} w(\det (\jac d_i)) = \mu - v(\phi) +  ( \inf \Gamma_{> 0} - \inf(\left<\Gamma, \mu \right>_{>0}))
	.\] 
\end{lemma}
\begin{para}
Note that if $\Gamma$ not discrete, then the ``fudge terms'' $( \inf \Gamma_{> 0} - \inf(\left<\Gamma, \mu \right>_{>0}))$ vanish. 
If $\Gamma$ is discrete,  then $\inf \Gamma_{> 0}$ is the valuation of a uniformizer. 
Note that if $\phi, \mu$ define a valuation, $v'$ as in \Cref{lem:inductive_model_unit_ball}, then $\Gamma_{v'} = \left<\Gamma, \mu \right>$. 
\end{para}
\begin{proof}[Proof of \Cref{lem:lim_dets}]
	We again treat the four cases separately.
	\begin{enumerate}
		\item[($\aleph$)] \textbf{Case $\mu \in \Gamma_{\Q}:$}
			We will argue that 
			\begin{align*}
				\lim_{(\beta, \ell) \in I} w(\theta^{-1}\phi) &= \mu - v(\phi) \\
				\lim_{(\beta, \ell) \in I} w(Y_i) &= \inf(\left<\Gamma, \mu \right>_{>0})\\			
				\lim_{(\beta, \ell) \in I} w(Z^{\ell} Y_i^{m}) &= \inf  \Gamma_{> 0}
			.\end{align*}
			Then the result follows from \eqref{eq:det_type_2}.
			The first equality holds because $\mu = w(\phi)$ and $v(\phi) = v(\theta)$ and $\theta$ is $w$-stable. 	

			Clearly $w(Y_i) = w(\beta^{-1}\phi^{\ell}) \in \left<\Gamma, \mu \right>_{>0}$ by the bounds on $v(\beta)$. 
			The condition $j \ge i$ translates directly to $w(Y_j) \le w(Y_i)$, so the sequence $(w(Y_i))_{i \in I}$ is non-increasing. 
			If $\Gamma$ is not discrete then we can get $w(Y_i)$ arbitrarily small by choosing $i = (\beta, \ell)$ with $v(\beta) \ge \ell\mu - \epsilon$ for $\epsilon > 0$. 
			If $\Gamma$ is discrete then $I$ has a maximal element $(\pi^{a}, \ell)$ as remarked in \Cref{rem:discrete_filtered_system}.
			Then $w(Y_{(\pi^{a}, \ell)}) = \ell\mu - av(\pi) = v(\pi) / m$. 
			In either case we obtain that $\lim_{(\beta, \ell) \in I} w(Y_i) = \inf(\left<\Gamma, \mu \right>_{>0})$.

			For the last equality we note that $w(Z) = 0$, so it is sufficient to show that $w(Y_i^{m}) = \inf \Gamma_{> 0}$.
			But as $|\left<\Gamma, \mu \right> / \Gamma| = m$ we see that $\Gamma_{> 0} = m\cdot (\left<\Gamma, \mu \right>_{>0}) $ and the result follows from the second equality.

		\item[($\beth$)] \textbf{Case $\mu \not\in \Gamma_{\Q}(\infty)$ and $\Gamma$ is not discrete:}
			From \eqref{eq:det_type_3_disc_1} we see that $w(\det \jac(d_i)) $ is constant and equal to $\mu - v(\phi)$. 
			The fudge terms $( \inf \Gamma_{> 0} - \inf(\left<\Gamma, \mu \right>_{>0}))$ vanish in this case. 
		\item[($\gimel$)] \textbf{Case $\mu \not\in \Gamma_{\Q}(\infty)$, and $\Gamma$ is discrete:}
			Similarly to case ($\aleph$) we will show that 
			\begin{align*}
				w(\theta^{-1}\phi) &= \mu - v(\phi) \\
				w(\pi) &= \inf  \Gamma_{> 0}\\
				w(X_iY_i) &= \inf  \left< \Gamma_{> 0}, \mu \right>_{> 0} = 0
			.\end{align*}
			Then the result follows from \Cref{eq:det_type_3_disc_2}.
			The first equality follows from the exact same arguments as the previous case. 
			By construction $v(\pi) = \inf  \Gamma_{> 0}$ and $\pi$ is constant, whence the second equality holds. 
			Finally, by classical results on continued fractions, we know that 
			\begin{align*}
				\left| \gamma^{-1} (\mu - v(\phi)) - \frac{a_i}{b_i}\right| \le \frac{1}{b_i^2}, \text{ thus }
				\left| b_i (\mu - v(\phi)) - \gamma a_i\right| \le \frac{\gamma}{b_i}
			,\end{align*}
			from which we learn that $|w(X_i)| \le \gamma / b_{2i}, |w(Y_i)| \le \gamma / b_{2i + 1} $. 
			So both tend to $0$ as $i \to \infty$, and we conclude the third inequality. 

		\item[($\daleth$)] \textbf{Case $\mu = \infty$:}
			We easily see that both sides equal $\infty$.
	\end{enumerate}
\end{proof}

\begin{lemma}\label{lem:enlargement_differentials}
	Let $A, \phi, \mu, v, w$ be as in \Cref{comp:inductive_model_differentials},
	and let $(L, v_L)$ be a valued field with a map of semi-valued rings $f:(A, w) \to (L, v_L)$. Then we have \[
		\mathbb L_{[A^{\circ}, \phi, \mu] / A^{\circ}} \otimes^{\mathbf L} L^{\circ} \simeq \Omega_{[A^{\circ}, \phi, \mu] / A^{\circ}} \otimes L^{\circ} [0]
	,\]
	and \[
		\cont \left(\Omega_{[A^{\circ}, \phi, \mu]} \otimes L^{\circ}\right) = \mu - v(\phi) +  ( \inf \Gamma_{> 0} - \inf(\left<\Gamma, \mu \right>_{>0})) 
	.\] 
\end{lemma}
\begin{proof}
	We can write $[A^{\circ}, \phi, \mu]$ as an increasing union of  $A^{\circ}$-algebras $A_i$ as in \Cref{prop:inductive_model_form}. 
	By \Cref{eq:fin_pres_omega_type_2,eq:fin_pres_omega_type_3_indisc,eq:fin_pres_omega_type_3_disc,eq:fin_pres_omega_type_1}, the module $\Omega_{A_i}$ is finitely presented as \[
	 0 \to A_i \cdot \left<\underline f \right> \xrightarrow{d_i}  A_i \cdot \left<\underline X \right> \to   \Omega_{A_i / A^{\circ}} \to 0
	,\] 
	for some variables $\underline X$ and relations $\underline f$.
	By \cite[eq. 1.2.3.4]{illusieComplexeCotangentDeformations2009}, we see that
	\[
		\mathbb L_{[A^{\circ}, \phi, \mu] / A^{\circ}}  = \colim_{i \in I}\left(\mathbb L_{A_i / A^{\circ}} \derotimes_{A_i} [A^{\circ}, \phi, \mu]  \right) 
	.\] 
	The functor $- \derotimes_{[A^{\circ}, \phi, \mu]}L^{\circ}$ is right exact and thus commutes with colimits. 
	Then we compute
	\begin{align*}
		\mathbb L_{[A^{\circ}, \phi, \mu] / A^{\circ}} \otimes L^{\circ} &= \colim_{i \in I} \left(\mathbb L_{A_i / A^{\circ}}\derotimes_{A_i} L^{\circ}\right) \\
		&\simeq \colim_{i\in I} \left(\Omega_{A_i / A^{\circ}}[0] \derotimes_{A_i}L^{\circ} \right) \\
			&\simeq \colim_{i \in I} \left(0 \to A_i \cdot \left<\underline f \right> \xrightarrow{d_i}  A_i \cdot \left<\underline X \right> \to   0\right) \derotimes_{A_i} L^{\circ} \\
			&\simeq \colim_{i \in I} \left(0 \to L^{\circ} \cdot \left<\underline f \right> \xrightarrow{d_i \otimes L^{\circ}}  L^{\circ} \cdot \left<\underline X \right> \to   0\right) \\
			&\simeq \colim_{i \in I} \left(\Omega_{A_i / A} \otimes L^{\circ}[0] \right) \\
			&\simeq \Omega_{[A^{\circ}, \phi, \mu]} \otimes_{[A^{\circ}, \phi, \mu]}L^{\circ} [0]
		.\end{align*}
	This completes the proof of the first part. 

	The second part, which states that $\cont(\Omega_{[A^{\circ}, \phi, \mu]/A} \otimes L^{\circ}) = \mu - v(\phi) +  ( \inf \Gamma_{> 0} - \inf(\left<\Gamma, \mu \right>_{>0}))$, would follow from \Cref{lem:lim_dets} if content commutes with the colimit ${\Omega_{[A^{\circ}, \phi, \mu]/A^{\circ}}} \otimes_{[A^{\circ}, \phi, \mu]} L^{\circ} = \colim  \Omega_{A_i / A} \otimes_{A_i} L^{\circ}$, i.e.
	\begin{equation}\label{eq:commute_1}
		\cont \Omega_{[A^{\circ}, \phi, \mu] / A} = \lim_{i \in I} \cont \Omega_{A_i / A} = \mu - v(\phi) +  ( \inf \Gamma_{> 0} - \inf(\left<\Gamma, \mu \right>_{>0}))
	.\end{equation} 
	Suppose that $\mu \ne \infty$, i.e., we are in case ($\aleph$), ($\beth$), ($\gimel$).
	It is sufficient to show that the system $(\Omega_{A_i / A} \otimes L^{\circ})_{i \in I}$ satisfies condition (c) of \Cref{lem:cont_colim} for \eqref{eq:commute_1} to be true, i.e., 
	\if\thesis1
	\begin{align*}
		&\lim_{i \in I} \lim_{j \ge i} \cont(\coker \Omega_{A_i / A} \otimes_{A_i} L^{\circ} \to \Omega_{A_j / A} \otimes_{A_j} L^{\circ}) \\
		& = \lim_{i \in I} \lim_{j \ge i} \cont( \Omega_{A_j / A_i} \otimes_{A_j} L^{\circ})  = 0 
	.\end{align*}
	\else
	\[
		\lim_{i \in I} \lim_{j \ge i} \cont(\coker \Omega_{A_i / A} \otimes_{A_i} L^{\circ} \to \Omega_{A_j / A} \otimes_{A_j} L^{\circ}) = \lim_{i \in I} \lim_{j \ge i} \cont( \Omega_{A_j / A_i} \otimes_{A_j} L^{\circ})  = 0 
	.\]  
	\fi

	Let $j \ge i$.
	Suppose that we make a choice of morphism $A[\underline{X_i}] \to A[\underline{X_j}]$ such that the following diagram commutes:
	\begin{equation}\label{eq:comm_diag_presentations}
	\begin{tikzcd}
		A[\underline{X_i}] \rar \dar & A[\underline{X_j}] \dar \\
		A_i \rar &  A_j
	\end{tikzcd}
	.\end{equation} 
	This leads to the following diagram with exact rows
	\[
	\begin{tikzcd}
		L^{\circ}\cdot \left<\dd \underline{X_i} \right> \rar \dar[two heads] &  L^{\circ}\cdot \left<\dd \underline{X_j} \right> \dar[two heads] \rar & \frac{L^{\circ}\cdot \left<\dd \underline{X_j} \right>}{L^{\circ}\cdot \left<\dd \underline{X_i} \right>} \dar[two heads] \rar & 0 \\
		\Omega_{A_i / A} \otimes L^{\circ} \rar & \Omega_{A_j / A} \otimes L^{\circ} \rar  &   \Omega_{A_j / A_i} \otimes L^{\circ} \rar & 0 
\end{tikzcd}
	.\] 
	It follows that:
	\[
		\cont (\Omega_{A_j / A_i} \otimes L^{\circ} ) \le \cont\left(  {L^{\circ}\cdot \left<\dd \underline{X_j} \right>} / {L^{\circ}\cdot \left<\dd \underline{X_i} \right>}\right)  
	.\] 
	Thus, it suffices to find morphisms $A[\underline {X_i}] \to A[\underline{X_j}]$ such that \eqref{eq:comm_diag_presentations} commutes and 
	\[
	\lim_{i \in I} \lim_{j \ge i} \cont\left(  {L^{\circ}\cdot \left<\dd \underline{X_j} \right>} / {L^{\circ}\cdot \left<\dd \underline{X_i} \right>}\right) = 0
	.\] 
	Again, we handle this case by case. 
	\begin{itemize}
		\item[($\aleph$)] \textbf{Case $\mu \in \Gamma_\Q$:} 
			We may replace $I$ with any cofinal subsequence. 
			In subcase  ($\aleph$.d), i.e., when $\Gamma$ is discrete, \Cref{rem:discrete_filtered_system} shows that we may replace $I$ by a maximal element of $I$ and the result is trivial. 
			The same remark shows that in subcase ($\aleph$.i), i.e., when $\Gamma$ is not discrete, we may replace $I$ by the subsequence $I' = \{(\beta, \ell) \in I \mid \ell = 1\} $. 
			Let $i = (\beta, 1) \in I', j = (\gamma, 1) \in I'$.
			Then $\underline {X_i} = [X, Z, Y_i], \underline {X_j} = [X, Z, Y_j]$ and the natural mapping $X \mapsto X, Z \mapsto Z, Y_i \mapsto \beta^{-1}\gamma Y_j$ defines a suitable morphism  $A[X, Z, Y_i] \to A[X, Z, Y_j]$.
			So  \[
				\frac{L^{\circ}\cdot \left<\dd \underline{X_j} \right>}{L^{\circ}\cdot \left<\dd \underline{X_i} \right>} = \frac{L^{\circ}\left<\dd X, \dd Z, \dd Y_j \right>}{L^{\circ}\left<\dd X, \dd Z, \dd Y_i \right>} = \frac{L^{\circ}\left< \dd Y_j \right>}{L^{\circ}\left<\beta^{-1}\gamma \dd Y_j \right>}
			,\] 
			which is a module of content $w(\beta^{-1}\gamma) = w(\beta^{-1}\phi) - w(\gamma \phi^{-1})  \le  \mu - v_i(\beta)$, and \[
				\lim_{(\beta, 1) \in I} \lim_{(\gamma, 1) \ge (\beta, 1)} \mu - v_i(\beta) = 0
				.\]
			\item[($\beth$)] \textbf{Case $\mu \not\in \Gamma_\Q(\infty)$ and $\Gamma$ is not discrete:} 
			Let $i = (\alpha, \beta) \in I, j = (\gamma, \delta) \in I' $.
			Then $\underline {X_i} = [X_i, Y_i], \underline {X_j} = [X_j, Y_j]$ and the natural mapping $X_i \mapsto \alpha^{-1} \gamma X_j, Y_i \mapsto \beta^{-1}\delta Y_j$ defines a suitable morphism  $A[X_i, Y_i] \to A[X_j, Y_j]$.
			So  \[
				\frac{L^{\circ}\cdot \left<\dd \underline{X_j} \right>}{L^{\circ}\cdot \left<\dd \underline{X_i} \right>} = \frac{L^{\circ}\left<\dd X_j, \dd Y_j \right>}{L^{\circ}\left<\dd X_i, \dd Y_i \right>} = \frac{L^{\circ}\left< \dd X_j \right>}{L^{\circ}\left<\alpha^{-1}\gamma \dd X_j \right>} \oplus \frac{L^{\circ}\left< \dd Y_j \right>}{L^{\circ}\left<\beta^{-1}\delta \dd Y_j \right>}
			,\] 
			which is a module of content \[
				w(\alpha^{-1}\phi) - w(\gamma^{-1}\phi) +  w(\beta^{-1}\phi) - w(\delta^{-1}\phi) \le  2\mu - v_i(\alpha) - v_i(\beta)
				,\] and \[
			\lim_{(\alpha, \beta) \in I} \lim_{(\gamma, \delta) \ge (\alpha, \beta)} 2\mu - v_i(\alpha) - v_i(\beta) = 0
				.\]
			\item[($\gimel$)] \textbf{Case $\mu \not\in \Gamma_\Q(\infty)$ and $\Gamma$ is discrete:} 
			Let $i, j \in \N$ with $j \ge i$. 
			Then $\underline {X_i} = [X_i, Y_i], \underline {X_j} = [X_j, Y_j]$.				Let  $m, n, m', n' \in \Z_{\ge 1}$ be the unique integers such that 
			\begin{align*}
				(a_{2i}, b_{2i}) &= m\cdot (a_{2j}, b_{2j}) + n(a_{2j + 1}, b_{2j}+ 1), \\ 
				(a_{2i+1}, b_{2i+1}) &= m'\cdot (a_{2j}, b_{2j}) + n'(a_{2j + 1}, b_{2j}+ 1)
			.\end{align*}
				Note that $mn' - m'n = 1$ because $\begin{pmatrix} m & m' \\ n & n' \end{pmatrix} \in \GL_2(\Z)^{+} $ as it maps a basis of $\Z^2$ to a basis of $\Z^2$ with the same orientation. 
			Then we can take the natural morphism $A[X_i, Y_i] \mapsto A[X_j, Y_j]$ induced by $X_i \mapsto X_j^{m}Y_j^{n}, Y_i \mapsto X_j^{m'}Y_j^{n'}$.
			So  
			\if\thesis1
			\begin{align*}
				\frac{L^{\circ}\cdot \left<\dd \underline{X_j} \right>}{L^{\circ}\cdot \left<\dd \underline{X_i} \right>} = \frac{L^{\circ}\left< \dd X_j, \dd Y_j \right>}{L^{\circ}\left< \begin{array}{c} m X_j^{m-1}Y_j^{n} \dd X_j + n X_j^{m}Y_j^{n-1} \dd Y_j,\\  m' X_j^{m'-1}Y_j^{n'} \dd X_j + n' X_j^{m'}Y_j^{n'-1} \dd Y_j \end{array} \right>} 
			,\end{align*}
			\else \[
				\frac{L^{\circ}\cdot \left<\dd \underline{X_j} \right>}{L^{\circ}\cdot \left<\dd \underline{X_i} \right>} = \frac{L^{\circ}\left<\dd X_j, \dd Y_j \right>}{L^{\circ}\left<\dd X_i, \dd Y_i \right>} = \frac{L^{\circ}\left< \dd X_j, \dd Y_j \right>}{L^{\circ}\left< \begin{array}{c} m X_j^{m-1}Y_j^{n} \dd X_j + n X_j^{m}Y_j^{n-1} \dd Y_j,\\  m' X_j^{m'-1}Y_j^{n'} \dd X_j + n' X_j^{m'}Y_j^{n'-1} \dd Y_j \end{array} \right>} 
			,\] 
			\fi 
			which is a module of content
			\begin{align*}
				w\left(\det \begin{pmatrix}  m X_j^{m-1}Y_j^{n}  &  n X_j^{m}Y_j^{n-1} \\  
				m' X_j^{m'-1}Y_j^{n'} &  n' X_j^{m'}Y_j^{n'-1} \end{pmatrix}\right) &= w((X_jY_j)^{-1}(X_iY_i)) 
			.\end{align*}
			Just as in the case ($\gimel$) of the proof of \Cref{lem:lim_dets}, we find that $\lim_{i \in \N} w(X_iY_i) = 0$ and $\lim_{j \in \N} w(X_jY_j) = 0$. 
			Thus, \[
			\lim_{i \in \N} \lim_{j \ge i} w((X_jY_j)^{-1}(X_iY_i)) = 0
			.\] 

	\end{itemize}

	We now turn to the case ($\daleth$), i.e., when $\mu = \infty$.
	We will show that \eqref{eq:commute_1} holds by showing that the system  $(\Omega_{A_i / A})_{i \in I}$ is actually an increasing union and applying \Cref{lem:cont_union}.

	Indeed, let $\alpha, \beta \in I$ with $v(\beta) \ge v(\alpha)$.
	Then we easily verify that \[
		\frac{A_{\alpha}[X_{\beta}]}{\beta\alpha^{-1}X_{\beta} - X_{\alpha}} \simeq A_{\beta}, X_{\beta}\mapsto \beta^{-1}\phi
	,\] 
	where $\beta\alpha^{-1} X_{\beta} - X_{\alpha}$ is a regular sequence.
	Thus, $\mathbb L_{A_{\beta} / A_{\alpha}}$ is concentrated in degree $0$ and $\Omega_{A_{\beta} / A_{\alpha}}$ is finitely presented. 
	So the sequence \[
		0 \to \Omega_{A_{\alpha} / A} \otimes L^{\circ} \to \Omega_{A_\beta /A} \otimes L^{\circ} \to \Omega_{A_\beta / A_\alpha} \otimes L^{\circ} \to 0
	,\] 
	is exact, and the first map is injective.
	Hence, the system $(\Omega_{A_i / A})_{i \in I}$ is increasing. 
\end{proof}

\subsection{Inducting on MacLane-Vaquié chains} \label{sec:inducting_on_maclane-vaquie_chains}

\begin{para}
	We are now at a point where we can interpret our work on enlargements and their differentials for augmentations of semi-valuations on $K[T]$.
	\Cref{lem:summary_ordinary_augmentation} is a straightforward summary of the results of the previous section in the specific case of an ordinary augmentation of a semi-valuation on $K[T]$.
	Most of the technical work left is in \Cref{lem:summary_limit_augmentation} and \Cref{lem:summary_almost_stable} where we show that similar results hold for limit augmentations. This requires a limit argument using enlargements that is not entirely straightforward. 
	Finally, in \Cref{thm:computation_main}, we use distinguished triangles to apply this to augmentation chains. 
\end{para}
\begin{lemma}\label{lem:maximal_localisation}
	Let $w$ be a semi-valuation on $K[T]$ and let $(K[T], w) \to (L, v_L)$  a map of semi-valued rings to a valued field $L$. 
	Let $S$ be a multiplicative subset of $K[T]$ such that $(S^{-1}K[T], w)$ is scalable and let $A \to S^{-1}K[T]$ be any ring morphism.
	Then 
	\[
		\mathbb L_{(S^{-1}K[T], w)^{\circ} / A} \derotimes L^{\circ} \simeq \mathbb L_{(K[T]_{\ker w}, w)^{\circ} / A} \derotimes L^{\circ}
	.\] 
\end{lemma}
\begin{proof}
	The ring $K[T]_{\ker w}$ is a localization of $S^{-1}K[T]$ which is scalable. 
	\Cref{lem:value_divisible_localisation} shows that there is some multiplicative set $S' \subset S^{-1}K[T]$ such that ${S'}^{-1}S^{-1}K[T] = K[T]_{\ker w}$ and $w(S') = \{0\}$.

	Then \[
		 \mathbb L_{(K[T]_{\ker w}, w)^{\circ} / A} \derotimes L^{\circ} \simeq
{S'}^{-1}\mathbb L_{(S^{-1}K[T], w)^{\circ} / A} \derotimes L^{\circ} 	\]
\end{proof}

\begin{lemma}[$\ctc$ for an ordinary augmentation]\label{lem:summary_ordinary_augmentation}
	Let $v \le w$ be two semi-valuations on  $K[T]$, and  let $S$ be a multiplicative system of $w$-stable elements such that $(S^{-1}K[T], v)$ is scalable. 
	Consider $(K[T], w) \to (L, v_L)$ a map of semi-valued rings to a valued field $L$.
	Let $\fra = [v, \phi, \mu]$ be a $w$-optimal ordinary augmentation and let $S'$ be the multiplicative system generated by $S$ and $\phi$ if $\mu \ne \infty$ and $S' = S$ if $\mu = \infty$.
	Write $R_1 = (S^{-1}K[T], v)^{\circ}$ and $R_2 = (S'^{-1}K[T], [\fra])^{\circ}$. 
	Then,
	\begin{enumerate}[(i)]
		\item $R_2 = [(S^{-1}K[T], v), \phi, \mu]$
		\item $(S'^{-1}K[T], [\fra])$ is scalable and on $(S'^{-1}K[T], [\fra])$ it holds that $[\fra] \le w$. 
		\item 
			$\ctc_{R_2 / R_1} \derotimes_{R_2} L^{\circ} \simeq \left(\Omega_{R_2 / R_1} \otimes_{R_2} L^{\circ}\right)[0]$
		\item $\cont\left( \Omega_{R_2 / R_1} \otimes_{R_2} L^{\circ} \right) = \step \fra + \inf(\Gamma_{v, >0}) - \inf(\Gamma_{[\fra], >0}) $
	\end{enumerate}
\end{lemma}
\begin{proof}
	These are straight forward consequences of previous lemmas. 
	(i) and (ii) follows from \Cref{lem:inductive_model_unit_ball}. (iii) and (iv) follow from \Cref{lem:enlargement_differentials}.
\end{proof}
\begin{lemma}[$\ctc$ for a limit augmentation of finite log-radius]\label{lem:summary_limit_augmentation}
	Let $v \le w$ be two semi-valuations on  $K[T]$ and  $S$ a multiplicative system of $w$-stable elements such that $(S^{-1}K[T], v)$ is scalable. 
	Let $(K[T], w) \to (L, v_L)$ be a map of semi-valued rings to a valued field $L$.
	Suppose $\fra = (v, (\psi_i, \gamma_i)_{i \in I}, \phi, \mu)$ is a $w$-optimal limit augmentation with $\mu \ne \infty$. 
	Let $S' \subset K[T]$ be the multiplicative set generated by $S, \phi$ and all $\psi_i, i \in I$. 
	For each $i \in I$, write $S_i$ for the multiplicative set generated by $S$ and all $\psi_j$ with $j \le i$, and $v_i = [v, \psi_i, \gamma_i]$. 
	We write $R_1:= (S^{-1} K[T], v)^{\circ}$ and $R_2 :=({S'}^{-1}K[T], [\fra])^{\circ}$. 
\begin{enumerate}[(i)]
		\item There is an increasing union
			\[
				R_2 = \bigcup_{i \in I} \left[(S_i^{-1} K[T], v_i)^{\circ}, \phi, \mu\right]
			.\] 
		\item $(S'^{-1} K[T], v')$ is scalable and $S '$ is  $w$-stable.
		\item $\mathbb L_{R_2 / R_1} \derotimes_{R_2} L^{\circ} = \left(\Omega_{R_2 / R_1} \otimes_{R_2} L^{\circ}\right)[0]$
		\item $\cont\left( \Omega_{R_2 / R_1} \otimes_{R_2} L^{\circ} \right) = \step \fra $.
			Note that in this case, $\Gamma_v$ is necessarily dense, so the fudge terms  $ \inf(\Gamma_{v, >0}) - \inf(\Gamma_{[\fra], >0})$ vanish automatically.
	\end{enumerate}
\end{lemma}
\begin{proof}
	Unlike the previous lemma, there is still an approximation argument needed to take care of the limit augmentation. 
	\begin{itemize}
		\item[(i)] Let $a \in R_2$ be any element. 
			Define $S'' := \bigcup_{i \in I} S_i$. 
			Then we can write $a = \frac{b}{s}\phi^{-n}$ with  $s \in S'', b \in K[T], n \in \Z$ and $v'(b) \ge v'(s) - n\mu$ and write $\sum_{\ell = 0}^{n}b_\ell \phi^{\ell}$ for the $\phi$-expansion of $b$. 
			Because $\deg b_\ell < \deg \phi$ and $\phi$ is a limit-key polynomial, we know that  $(v_{i}(b_\ell))_{i \in I}$ stabilizes. 
			Choose $j \in I$ sufficiently large, such that  $v_i(b_{\ell})$ has stabilized for all  $\ell = 0, \ldots, n$ and $s \in S_j$. 
		Then $v'(s) = v_j(s)$ because all valuations of the generators of  $S_j$ have stabilized.
		Also, \[
			v'(a) = \min_{\ell = 0}^{n} \{v_j(s^{-1}b_\ell )  + (\ell - n)\cdot \mu\} \ge 0, \text{ thus } v_j(s^{-1} b_\ell) + (\ell -n) \cdot  \mu \ge 0, \forall \ell
		\] 
		So we can write each term $s^{-1}b_\ell \phi^{m-n} =\alpha_\ell \cdot s^{-1}\beta_{\ell}^{-1}\phi^{m-n}$ with $\beta_\ell \in \vd_{v_j(b_\ell)}^{w}$ and $b_\ell = \beta_\ell^{-1}\alpha_\ell$	and $ a \in \left[(S_j^{-1} K[T], v_j), \phi, \mu\right]$.
		So $s ^{-1}b_\ell \phi^{m-n} \in [(S_j^{-1}K[T], v_j), \phi, \mu]$ and thus $a \in [(S_j^{-1}K[T], v_j), \phi, \mu]$.

		By definition 
		\if\thesis1
		\begin{align*}
			&[(S_j^{-1}K[T], v_j), \phi, \mu] \\ &= (S_j^{-1}, K[T])^{\circ}[a^{-1}\phi^{n} \mid a \in \vd(S_j^{-1}K[T]), n \in \Z, n\cdot \mu - v(a) \ge 0]
		,\end{align*}
		\else
		\[
			[(S_j^{-1}K[T], v_j), \phi, \mu] = (S_j^{-1}, K[T])^{\circ}[a^{-1}\phi^{n} \mid a \in \vd(S_j^{-1}K[T]), n \in \Z, n\cdot \mu - v(a) \ge 0]
		,\]
		\fi 
		and, as $v' \ge v_j$ we see that each $a^{-1}\phi \in R_2$. 
	\item [(ii)] The value group $\Gamma_v'$ is generated by $\Gamma_v, \{\gamma_i\} _{i \in I}, \mu$.
		So any $r \in \Gamma_v'$ can be written as a $\Z$-linear combination  $r = r_v + \sum_{i \in I} a_i \gamma_i + b \cdot \mu$ with $a_i = 0$ for all but finitely many $i \in I$ and $r_v = v(\alpha), \alpha \in \vd_{w}((S^{-1}K[T], v))$.
		Then $\alpha \cdot \phi^{b}\cdot \prod_{i \in I} \psi_i^{a_i}$ is a constant with valuation $r$.

	\item[(iii-iv)] 
		Consider the tower of rings $R_1 \to (S^{-1}_i K[T], v_i)^{\circ} \to  \left[(S_i^{-1} K[T], v_i), \phi, \mu\right]$. Note that up to a localization of the middle ring, this is actually two consecutive enlargements. Consider the associated triangle of cotangent complexes
		\if\thesis1
		\[
			\resizebox{.98\linewidth}{!}{$\displaystyle\mathbb L_{\frac{(S^{-1}_i K[T], v_i)^{\circ}}{R_1}} \derotimes \left[(S_i^{-1} K[T], v_i), \phi, \mu\right]  \to \mathbb L_{\frac{\left[(S_i^{-1} K[T], v_i), \phi, \mu\right]}{R_1}} \to \mathbb L_{\frac{\left[(S_i^{-1} K[T], v_i), \phi, \mu\right]}{(S^{-1}_i K[T], v_i)^{\circ} }}$}
		.\] 
		\else
		\[
			\mathbb L_{\frac{(S^{-1}_i K[T], v_i)^{\circ}}{R_1}} \derotimes \left[(S_i^{-1} K[T], v_i), \phi, \mu\right]  \to \mathbb L_{\frac{\left[(S_i^{-1} K[T], v_i), \phi, \mu\right]}{R_1}} \to \mathbb L_{\frac{\left[(S_i^{-1} K[T], v_i), \phi, \mu\right]}{(S^{-1}_i K[T], v_i)^{\circ} }}
		.\] 
		\fi 
		By \Cref{lem:value_divisible_localisation} we know that $(S_i^{-1} K[T], v_i)^{\circ}$ is a localization of $(S^{-1}\psi^{-1} K[T], v_i)^{\circ} = [S^{-1} K[T], \psi, \gamma_i]$  by some multiplicative system $T$, consisting of $w$-stable elements of valuation $0$ which are units in  $L^{\circ}$.

		Thus, taking $-\derotimes L^{\circ}$ gives the triangle
		\if\thesis1
		\[
			\resizebox{.99\linewidth}{!}{$\displaystyle\mathbb L_{\frac{[S^{-1}K[T], \psi, \gamma_i]}{R_1}} \derotimes L^{\circ} \to \mathbb L_{\frac{\left[(S_i^{-1}K[T], v_i), \phi, \mu\right]}{R_1}} \derotimes L^{\circ} \to \mathbb L_{\frac{\left[(S_i^{-1} K[T], v_i), \phi, \mu\right]}{(S^{-1}_i K[T], v_i)^{\circ} }} \derotimes L^{\circ}$}
		.\] 
		\else
		\[
			\mathbb L_{\frac{[S^{-1}K[T], \psi, \gamma_i]}{R_1}} \derotimes L^{\circ} \to \mathbb L_{\frac{\left[(S_i^{-1}K[T], v_i), \phi, \mu\right]}{R_1}} \derotimes L^{\circ} \to \mathbb L_{\frac{\left[(S_i^{-1} K[T], v_i), \phi, \mu\right]}{(S^{-1}_i K[T], v_i)^{\circ} }} \derotimes L^{\circ}
		.\] 
		\fi 
		\Cref{lem:enlargement_differentials,lem:summary_ordinary_augmentation} show that the outer complexes are concentrated in degree $0$, thus so is the middle complex, i.e.\ \[
			\mathbb L_{\frac{\left[(S_i^{-1}K[T], v_i), \phi, \mu\right]}{R_1}} \derotimes L^{\circ} \simeq \Omega_{\frac{\left[(S_i^{-1}K[T], v_i), \phi, \mu\right]}{R_1}} \otimes L^{\circ}[0]
		.\] 
		And by the additivity of content and \Cref{lem:enlargement_differentials,lem:summary_ordinary_augmentation} we find
		\if\thesis1
		\begin{align*}
			\cont\left(\Omega_{\frac{\left[(S_i^{-1}K[T], v_i), \phi, \mu\right]}{R_1}} \otimes L^{\circ}\right) 
															    &= \gamma_i - v(\psi_i) + \mu - v_i(\phi)
		.\end{align*}
		\else		
		\begin{align*}
			\cont\left(\Omega_{\frac{\left[(S_i^{-1}K[T], v_i), \phi, \mu\right]}{R_1}} \otimes L^{\circ}\right) 
			&= \cont\left(\Omega_{\frac{[S^{-1}K[T], \psi, \gamma_i]}{R_1}} \otimes L^{\circ}\right) + \cont \left(\Omega_{\frac{\left[(S_i^{-1} K[T], v_i), \phi, \mu\right]}{(S^{-1}_i K[T], v_i)^{\circ} }} \otimes L^{\circ}\right)\\
															    &= \gamma_i - v(\psi_i) + \mu - v_i(\phi)
		.\end{align*}
		\fi 
		Then (iii) follows from the fact that tensor products and differentials commute with filtered colimits. 
		Statement (iv) would follow if we can show that content commutes with the colimit in this case. 
		We will give a similar argument as in the proof of \Cref{lem:enlargement_differentials}. 
		Recall that limit augmentations only occur when $K$ is not discretely valued. 
		So all semi-valuations involved have dense value group. 
		\begin{itemize}
			\item[($\aleph$)] \textbf{Case $\mu \in \Gamma_{v} \otimes \Q$:}
					For each $i \in I$ let $m_i \in \Z_{> 1}$ be the minimal integer such that $m_i \cdot \mu \in \Gamma_i$. 
					Because $\Gamma_i$ increases with $i$, we see that $m_i$ decreases, and thus must stabilize at some point to $m \in \Z_{>1}$.
					Let $k \in I $ be an element where $m_k = m$ and let $I' = \{j \in I \mid j \ge k\} $. 
					We can also take $\alpha \in \vd(S_k^{-1}K[T], v_k)$ such that $v_k(\alpha) = m\cdot \mu$.
					Note that $\alpha$ satisfies the same properties for all $j \in I'$.

					Write $A_i = (S_i^{-1} K[T], v_i)$ for all $i \in I'$. 
					Then we can write 
					\[
						[A_i^{\circ}, \phi, \mu] = \bigcup_{j \in J_i} A_{(i,j)}
					,\] 
					with the poset $J_i$ and the ring $A_{(i, j)}$ as in \Cref{prop:inductive_model_form}, and uniformly used $\alpha$ as the constant with value $\mu\cdot m$.

					Define $\vd(\mathcal{A} ) = \bigcup_{i \in I'} \vd(S_i^{-1}K[T], v_i) $.
					Note that this is an increasing union. 
					Then we define \[
						J = \{(i,(\beta,1)) \in I' \times  (\vd(\mathcal{A}) \times \Z)  \mid (\beta,1) \in J_i\} 
					,\]
					which is a directed poset where $(j, (\gamma,1)) \ge (i, (\beta, 1)) \iff w(\gamma)\ge w(\beta) \wedge j \ge i$.
					Note that $w(\beta) \to \mu$ as $(i, (\beta, 1))$ increases.

					Then we easily verify that $R_2 = \bigcup_{(i, (\beta, 1)) \in J} A_{i, (\beta, 1)}$. 
					We will argue that $(\Omega_{A_{i, (\beta,1)} / R_1})_{(i, (\beta, 1)) \in J}$ satisfies condition (c) from \Cref{lem:cont_colim}. 
					Let $(j, (\gamma,1)) \ge (i, (\beta, 1))$ then the mapping $X \mapsto X, Z \mapsto Z, Y_{(\beta, 1)} \mapsto  \beta^{-1} \gamma Y_{\gamma, 1}$ induces the following commutative diagram: 
					\[
					\begin{tikzcd}
						A_i[X, Z, Y_{(\beta, 1)}] \rar \dar & A_j[X, Z, Y_{(\gamma, 1)}] \dar \\
						A_{(i, (\beta,1))} \rar &  A_{(j, (\gamma, 1))}
					\end{tikzcd}
					.\] 
					Then exactly the same argument as in the proof of \Cref{lem:enlargement_differentials} shows that \[
					\cont\left(\coker \Omega_{A_{i, (\beta, 1)} / R_1} \to \Omega_{A_{j, (\gamma, 1)} / R_1}\right) \le \mu - w(\beta)\]
					and, that condition (c) of \Cref{lem:cont_colim} holds. 
				\item[($\beth$)] \textbf{Case $\mu \not\in \Gamma_{\mathcal{A} } \otimes \Q(\infty)$:}
					Again we write $A_i = (S_i^{-1} K[T], v_i)$ and \[
						[A_i^{\circ}, \phi, \mu] = \bigcup_{j \in J_i} A_{(i, j)}
					,\] 
					as in \Cref{prop:inductive_model_form}.

					Let 
					\[
						J = \{(i,(\alpha, \beta)) \in I' \times  (\vd(\mathcal{A})^2 \times \Z)  \mid (\alpha, \beta) \in J_i\} 
					,\]
					be the directed poset with $(j, (\gamma, \delta)) \ge (i, (\alpha, \beta)) \iff w(\gamma) \ge w(\alpha) \wedge w(\delta) \ge w(\beta) \wedge j \ge i$.
					Like the previous case, we see that $R_2 = \bigcup_{(i, j) \in J} A_{(i, j)}$. 
					Now let $(j, (\gamma, \delta)) \ge (i, (\alpha, \beta))$ and consider the mapping  $X_{(\alpha, \beta)} \mapsto \alpha^{-1} \gamma X_{(\gamma, \delta)}, Y_{(\alpha, \beta)} \mapsto \beta^{-1} \delta Y_{(\gamma, \delta)}$ yields the following commutative diagram:
					\[
					\begin{tikzcd}
						A_i[X_{(\alpha, \beta)}, Y_{(\alpha, \beta)}  ] \rar \dar & A_j[X_{(\gamma,\delta)}, Y_{(\gamma, \delta)}] \dar \\
						A_{(i, (\alpha, \beta))} \rar &  A_{(j, (\gamma, \delta))}
					\end{tikzcd}.
					\] 
					We conclude like the previous case. 
		\end{itemize}
	\end{itemize}
\end{proof}
\begin{lemma}[$\ctc$ for stable continuous family or almost stable limit augmentation]\label{lem:summary_almost_stable}
	Let $v \le w$ be two semi-valuations on  $K[T]$ and  $S$ a multiplicative system of $w$-stable elements such that $(S^{-1}K[T], v)$ is scalable. 
	Let $(K[T], w) \to (L, v_L)$ be a map of semi-valued rings to a valued field $L$.
	Suppose $\fra = (v, (\psi_i, \gamma_i)_{i \in I})$ is a $w$-optimal stable continuous family or $\fra = (v, (\psi_i, \gamma_i)_{i \in I}, \phi, \mu)$ is a $w$-optimal almost stable limit augmentation. 
	Let $S' \subset K[T]$ be the multiplicative set generated by $S$ and all $\psi_i, i \in I$. 
	For each $i \in I$, write $S_i$ for the multiplicative set generated by $S$ and all $\psi_j$ with $j \le i$, and $v_i = [v, \psi_i, \gamma_i]$. 
	We write $R_1:= (S^{-1} K[T], v)^{\circ}$ and $R_2 :=({S'}^{-1}K[T], [\fra])^{\circ}$. 
	\begin{enumerate}[(i)]
		\item There is an increasing union
			\[
				R_2 = \bigcup_{i \in I} (S_i^{-1} K[T], v_i)^{\circ}			.\] 
			\item $(S'^{-1} K[T], [\fra])$ is scalable and $[\fra] = w$.
		\item $\mathbb L_{R_2 / R_1} \derotimes_{R_2} L^{\circ} = \left(\Omega_{R_2 / R_1} \otimes_{R_2} L^{\circ}\right)[0]$
		\item $\cont\left( \Omega_{R_2 / R_1} \otimes_{R_2} L^{\circ} \right) = \step\fra$.
	\end{enumerate}
\end{lemma}
\begin{proof}
	\begin{enumerate}[(i)]
		\item 	For an almost stable limit augmentation we know that  $w = [\fra] = \lim_{i} v_i$ where this is the point wise limit, and it stabilizes for every $f \in K[T]$ with $v'(f) < \infty$.
			Let $f / s$ be any element in $R_2$ with $f \in K[T]$ and $s \in S'$.
		Then there is some sufficiently large $j \in I$ such that $s \in S_j$ and $v_j(f) = w(f)$ if $f\not\in \ker [\fra]$ or $v_j(f) > v_j(s)$ when $f \in \ker [\fra]$. 
	As  $f / s \in R$ we see that $[\fra](f) - [\fra](s) \ge 0$ from which we can conclude that $f / s \in (S_j^{-1}K[T], v_j)^{\circ}$. 
			The argument when $\fra$ is a stable continuous family is the same. 
		\item Let $r \in \Gamma_w$ be any element. 
			Then there is some  $j \in J$ such that $r \in \Gamma_{v_j}$. 
			Then there is a constant $\alpha \in (S_j^{-1}K[T], v_j)^{\times }$ with $v_j(\alpha) = r$. 
			Thus, $\alpha$ is a constant of $(S'^{-1}K[T], w)^{\times }$ with $w(\alpha) = r$. 
		\item[(iii-iv)] Write $A_i = (S_i^{-1}K[T], v_i)^{\circ}$ for every $i \in I$. 
			Then  $A_i$ is a localization of $[R_1, \psi_i, \gamma_i]$ by elements of value $0$, which map to invertible elements of  $L^{\circ}$. 
			Thus,
			\if\thesis1
			\begin{align*}
				\ctc_{\frac{A_i}{R_1}} \derotimes L^{\circ} &\simeq \ctc_{\frac{[R_1, \psi_i, \gamma_i]}{R_1}} \simeq ( \Omega_{\frac{A_1}{R_1}} )[0], \text{ and} \\
				 \cont(\Omega_{\frac{A_i}{R_1}}\otimes L^{\circ}) &= \gamma_i - v(\psi_i) + \inf(\Gamma_{v, >0}) - \inf(\Gamma_{[v, \psi_i, \gamma_i], >0})
			.\end{align*}
			\else
			\[
				\ctc_{\frac{A_i}{R_1}} \derotimes L^{\circ} \simeq \ctc_{\frac{[R_1, \psi_i, \gamma_i]}{R_1}} \simeq ( \Omega_{\frac{A_1}{R_1}} )[0], \quad \cont(\Omega_{\frac{A_i}{R_1}}\otimes L^{\circ}) = \gamma_i - v(\psi_i) + \inf(\Gamma_{v, >0}) - \inf(\Gamma_{[v, \psi_i, \gamma_i], >0})
			.\] 
			\fi
			Similarly, for any $j > i$ we see that $\ctc_{\frac{A_i}{R_1}} \derotimes L^{\circ} \simeq ( \Omega_{\frac{A_i}{R_1}} )[0]$ and thus that the system $(\Omega_{A_i / R_1} \otimes L^{\circ})_{i \in I}$ is an increasing union. 
			We conclude by taking the filtered colimits over $i$. 
	\end{enumerate}
\end{proof}

\begin{theorem}\label{thm:computation_main}
	Let $w$ be any semi-valuation on $K[T]$ and let $(\fra_i)_{i = 1}^{n}$ be a $w$-optimal augmentation chain starting in a simple Gauss semi-valuation $v_0$. 
	Consider $(K[T], w) \to (L, v_L)$ a map of semi-valued rings to a valued field $L$.
	Let $S_0$ be a multiplicative set of $w$-stable elements such that $(S_0^{-1}K[T], v_0)$ is scalable. 
	Further, let $S_i$ be the multiplicative set generated by all key polynomials $\phi_j, \psi_{j, k}$ for all $j \le i, k \in J_{j}$, and let $S$ be the multiplicative system generated by all $S_i$.
	Let $R_i = (S_i^{-1}K[T], v_i)^{\circ}$ and $R_w = (S^{-1}K[T], w)^{\circ}$
	Then for each $i \in [0, n]\cap \Z$ we have 
	\begin{enumerate}[(i)]
		\item $\mathbb L _{R_i / R_0} \derotimes_{R_n} L^{\circ} = \left(\Omega_{R_i / R_0}\otimes_{R_i} L^{\circ}\right)[0] $,
		\item  $\cont \left(\Omega_{R_i / R_0}\otimes_{R_i} L^{\circ}\right) = \step (\fra_j)_{j = 1}^{i} + \inf(\Gamma_{v_0, > 0}) - \inf( \Gamma_{v_i, > 0})$,
		\item $R_i$ is scalable and $R_w$ is scalable.
		\item $\mathbb L _{R_w / R_0} \derotimes_{R_w} L^{\circ} \simeq \mathbb L_{(K[T]_{\ker w}, w)^{\circ} / R_0} \derotimes L^{\circ} \simeq  \left(\Omega_{R_w / R_0}\otimes_{R_w} L^{\circ}\right)[0] $,
		\item  $\cont \left(\Omega_{R_w / R_0}\otimes_{R_w} L^{\circ}\right) =  \cont\left(\Omega_{\frac{(K[T]_{\ker w}, w)^{\circ}}{R_0}} \derotimes L^{\circ}\right) = \step (\fra_j)_{j = 1}^{n} +  \inf(\Gamma_{v_0, > 0}) - \inf( \Gamma_{w, > 0})$.
	\end{enumerate}
\end{theorem}
\begin{proof}
	We may assume that none of the augmentations in the chain are limit augmentations that are not almost stable and have log-radius $\infty$. 
	Suppose that there is such an augmentation then necessarily the chain is finite and the only such augmentation is the last augmentation $\fra_n = (v_{n-1}, (\psi_i, \gamma_i)_i, \phi_n, \infty)$. 

	Then we can split this limit augmentation up into one limit augmentation with finite log-radius and one ordinary augmentation with infinite log-radius. 
	Indeed, let $\alpha \in K^{\circ}$ with $\mu_n' = v(\alpha)$ be such that $\lim_i [v_{n-1}, \psi_i, \gamma](\phi_n) < \mu_n' < \infty$. 
	Write $\phi_n'=\phi_n + \alpha$. 
	Then $\phi'$ is also a limit key polynomial of $(v_{n-1}, (\psi_i, \gamma_i))$. 
	Let $\fra_n' = (v_{n-1}, (\psi_i, \gamma_i)_i, \phi_n', \mu_n')$ and $\fra_{n+1}' = ([\fra_n'], \phi_n, \infty)$. 
	Then the chain $(\fra_1, \ldots, \fra_{n-1}, \fra_n', \fra_{n+1}')$ is also a $w $-optimal augmentation chain approximating $w$, and containing no limit augmentations of infinite log-radius.
	We write $\mathfrak{A}  = (\fra_i)_{i = 1}^{n}$ and $\mathfrak{A}' = (\fra_1, \ldots, \fra_{n-1}, \fra_n', \fra_{n+1}')$.

	Suppose that the theorem holds for the chain $\mathfrak{A} '$. 
	Then, (i), (ii), and (iii) are equivalent to the theorem for $\mathfrak{A} $ when $i < n$. 
	When  $i = n$, these are equivalent to (iv) and (v) which are equivalent for both $\mathfrak{A}$ and $\mathfrak{A}' $ by \Cref{lem:maximal_localisation}.

	From now on, we assume that there are no non-almost-stable limit augmentations of infinite log-radius in $(\fra_i)_{i = 1}^{n}$. 
	We will first show (i), (ii), (iii) using induction and then obtain (iv) and (v) by taking colimits.
	Clearly $\mathbb L_{R_0 / R_0} = 0$ and $\Omega_{R_0 / R_0} = 0$, thus the base case is trivial. 
	Suppose that (i), (ii) hold for $i$, then consider the tower of rings $R_0 \to R_i \to R_{i + 1}$ and the corresponding distinguished triangle of cotangent complexes tensored by $L^{\circ}$:
	\[
		\mathbb{L}_{R_i / R_0} \derotimes_{R_n} L^{\circ} \to \mathbb{L}_{R_{i + 1} / R_0} \derotimes_{R_{i + 1}} L^{\circ} \to \mathbb{L}_{R_{i + 1} / R_i} \derotimes_{R_{i + 1}} L^{\circ}
	.\] 
	From \Cref{lem:summary_ordinary_augmentation,lem:summary_limit_augmentation} and the induction hypothesis we see that the outer complexes are concentrated in degree $0$, so we obtain  (i). We conclude (ii) and (iii) from the same lemmas and the additivity of content. 
	From this triangle we also see that the map $\Omega_{R_i / R_0} \otimes L^{\circ} \to \Omega_{R_{i+1} / R_0} \otimes L^{\circ}$ is injective. 

	If $n$ is finite, then $R_w = R_n$ thus (iv) and (v) hold immediately. 
	If $n = \infty$, then we obtain $R_w = \bigcup_{i \in I}R_i$, which is an increasing union, and as we've just argued $(\Omega_{R_i / R_0} \otimes L^{\circ})_{i = 1}^{\infty}$ is an increasing union too. 
	Then we obtain (iv) and (v) by taking colimits as before. 
\end{proof}

\begin{remark}\label{rem:explicit_complete_intersection_type_II}
	We continue with the notation as in \Cref{thm:computation_main}. 
	When $K$ is discretely valued, we know that no limit augmentations occur. 
	Suppose that the augmentation chain $(\fra_i)_{i = 1}^{n}$ is finite and only consists of ordinary augmentations with log-radius in $\Gamma_K \otimes \Q$.
	In the Berkovich setting, this is equivalent to $w$ being a type II point. 
	Then each $R_{i + 1}$ is an enlargement of $R_i$ of type ($\aleph$.d) and thus by \Cref{rem:discrete_filtered_system} we know that  $R_{i + 1}$ is complete intersection over $R_i$ given by three variables $[X_i, Y_i, Z_i]$ and relations  $X_iZ_i-1, f_i, g_i$ as in \Cref{prop:inductive_model_form}.
	Note that $Y_i$ maps to a constant of minimal valuation in $R_{i + 1}$ so $Y_i$ can take the role of $\pi$ in \Cref{rem:discrete_filtered_system} when finding the presentation of $R_{i + 2}$ as an $R_{i + 1}$ algebra.  

	So we can find an explicit presentation of $R_n$ as a complete intersection over $R_0$ as 
	\[
		R_n = \frac{R_0[X_0, Y_0, Z_0, \ldots, X_{n-1}, Y_{n-1}, Z_{n-1}]}{(X_0Z_0-1, f_0, g_0, \ldots, X_{n-1}Y_{n-1}-1, f_{n-1}, g_{n-1})}
	.\] 
	It might also be possible to compute other invariants like the precise structure of $\Omega_{R_n / R_0}$ and not only its content. 
\end{remark}

\subsection{The log cotangent complex of an enlargement} \label{sec:the_log-cotangent_complex_of_an_enlargement}

\begin{para}
	Using the previous computation we may also compute the log differentials and more generally the log cotangent complex of enlargements. 
	Recall that a pre-log structure on a ring $A$ is the data of a monoid  $M$ and a morphism of monoids $\alpha: M \to (A, \cdot)$. 
	A ring equipped with a pre-log structure is called a pre-log ring, and we will write it as $(M \xrightarrow{\alpha} A)$. Some sources may write this as $(A, M)$. 
	A morphism of pre-log rings $f:(M_A \to A) \to (M_B \to B) $ is the data of a map of monoids  $f_M: M_A \to M_B$ and a map of rings $f_R: A \to B$ such that the following diagram commutes. \[
	\begin{tikzcd}
		M_A \rar \dar{f_M} & A \dar{f_R} \\
		M_B \rar & B
	\end{tikzcd}
	\]

	We say that the pre-log structure or pre-log ring $(M \xrightarrow{\alpha} A)$ is a log structure or log ring, respectively, if $\alpha$ induces an isomorphism $\alpha^{-1}(A^{\times }) \to A^{\times }$. 
	Any pre-log ring or pre-log structure can be turned into a log ring or log structure in a canonical way, called logification, which is left adjoint to the inclusion of the category of log rings to the category of pre-log rings, see \cite[Chapter II, Proposition 1.1.5]{ogusLecturesLogarithmicAlgebraic2018}. 
	Intuitively, one can think of logification as adjoining all units of $A$ to $M$. 
\end{para}
\begin{para}
	For a morphism of pre-log rings $(M_A \to A) \to (M_B \xrightarrow{\beta} B)$, the module $\lOmega_{((M_B) \to B) / (M_A \to A)}$ is the logarithmic analogue of the ordinary differentials. 
	It is the universal log derivation, meaning that it has a differential $\dd: B \to \lOmega_{(M_B \to B) / (M_A \to A)}$ and the log differential map $\dlog: M_B \to \lOmega_{(M_B\to B) / (M_A \to A)}$ which satisfies $\dd \beta(m) = \beta(m)\cdot \dlog m$ for every $m \in M_B$.

	Let $(M_A \to A) \to (M_B \to B) \to (M_C \to C)$ be a tower of pre-log rings. 
	Analogous to ordinary differentials, there is an exact sequence \[
		\lOmega_{\scriptstyle\frac{(M_B \to B)}{(M_A \to A)}} \otimes_{B} C \to \lOmega_{\scriptstyle\frac{(M_C \to C)}{(M_A \to A)}} \to \lOmega_{\scriptstyle\frac{(M_C \to C)}{(M_B \to B)}} \to 0
	,\] 
	and one may wonder whether this sequence can be extended to the left. 
	Analogous to the cotangent complex, the log cotangent complex provides an answer to this question. 
	In this text, the log cotangent will always mean Gabber's log cotangent complex.
	The alternative, Olsson's log cotangent complex, is unsuitable for our purposes as it is only defined for fs log rings. 
	Gabber's log cotangent complex is described in \cite[§~8]{olssonLogarithmicCotangentComplex2005}, and together with \cite[§~2]{hubnerLogarithmicDifferentialsDiscretely2024} it gives a good description of how to use it in practice. 
	For any tower of pre-log rings as above, there is a distinguished triangle \[
		\lctc_{\scriptstyle \frac{(M_B \to B)}{(M_A \to A)}} \derotimes_{B} C \to \lctc_{\scriptstyle\frac{(M_C \to C)}{(M_A \to A)}} \to \lctc_{\scriptstyle\frac{(M_C \to C)}{(M_B \to B)}}
	.\] 
	The log cotangent complex is also compatible with the ordinary cotangent complex in the sense that for any ring morphism $A \to B$ and pre-log structure $M \to A$ there are natural quasi-isomorphisms \begin{equation}\label{eq:ctc_eq_lctc}
		\ctc_{\scriptstyle\frac{B}{A}} \simeq \lctc_{\scriptstyle\frac{(0 \to B)}{(0 \to A)}} \simeq \lctc_{\scriptstyle\frac{M \to B}{M \to A}}
	,\end{equation}
	see \cite[Lemma 8.21, Lemma 8.22]{olssonLogarithmicCotangentComplex2005}. 
	The log cotangent complex is also invariant under logification, in the sense that if two morphisms $(M_A \to A) \to (M_B \to B) $ and $(M_A' \to A) \to (M_B' \to B)$ logify to the same morphism then there is a natural quasi-isomorphism \[
	\lctc_{\scriptstyle\frac{(M_B \to B)}{(M_A \to A)}} \simeq \lctc_{\scriptstyle\frac{(M_B' \to B)}{(M_A' \to A)}}
	,\] 
	see \cite[Lemma 8.20]{olssonLogarithmicCotangentComplex2005}.
\end{para}
\begin{para}\label{par:total_and_boundar_log_structure}
	Let $(A, v)$ be a semi-valued $K$-algebra without zero divisors. 
	The \textit{total log structure on $A^{\circ}$} is the inclusion of monoids  $\vd(A^{\circ}) \into A^{\circ}$.
	Also note that if $A = L$ is a field, then the total log structure on $L^{\circ}$ is $L^{\circ} \setminus \{0\} $ and thus agrees with the usual notion of the total or integral log structure.

	If $\ker v\subset A$ is a principal ideal generated by $f$, then the \emph{boundary log structure} is defined as $\bls (A^{\circ}) = (A[f^{-1}])^{\times } \cap A^{\circ}$.
\end{para}
\begin{para}
	Let $(A, v_A) \to (B, v_B)$ be a map of scalable $K$-algebras and let $M_B \to B^{\circ}$ be the total or boundary log structure if that is well-defined. 
	Let $\phi, \psi \in B^{\circ}$ be two elements that only differ by a multiplication by a constant in $A$, say without loss of generality $\psi = \alpha \phi$ with $\alpha \in \const A^{\circ}$. 
	Then as elements of $\lOmega_{(M_B \to B^{\circ}) / (\const A^{\circ}\to A^{\circ})}$ we see that $\dlog \psi = \dlog \alpha + \dlog \phi = \dlog \phi$. 
	So by abuse of notation, for $\phi \in B$ such that there is a  $\theta \in \const A^{\circ}$ such that $\theta\cdot \phi \in M_B$  we will write $\dlog \phi = \dlog \theta \phi$. 
\end{para}
\begin{para}\label{par:log_non_log_same_val_grp}
	Let $(A, v_A) \to (B, v_B)$ be a map of scalable $K$-algebras. 
	Suppose that $\Gamma_{A} = \Gamma_B$, then it is easily verified that $(\const A \to B)$ logifies to $(\const B \to B)$. 
	Then \[
	\lctc_{\textstyle \frac{(\const B \to B)}{(\const A \to A)}} \simeq \lctc_{\textstyle \frac{(\const A \to B)}{(\const A \to A)}} \simeq \ctc_{\textstyle \frac{B}{A}}
	.\] 
\end{para}

\begin{lemma}\label{lem:lctc_add_to_the_log_structure}
	Let $(M_A \xrightarrow{\alpha} A)$ be a pre-log ring and let $a \in A$ be any element. 
	Consider the pre-log ring $(M_A \oplus \N \to A)$ where the monoid map is given by $(m, n) \mapsto \alpha(m) \cdot a^{n}$. 
	Then \[
		\lctc_\frac{(M_A \oplus \N \to A)}{(M_A \to A)} \simeq (A\cdot \dd a \to A\cdot \dlog a)[-1,0]
	.\] 
	In particular, when $a$ is not a zero divisor, this is quasi-isomorphic to $(\frac{A}{a\cdot A}\cdot \dlog a)[0]$.
\end{lemma}
\begin{proof}
	Consider the pre-log ring $(M_A \oplus \N \to A[T])$ with the monoid map given by $(m, n) \mapsto \alpha(m) \cdot T^{n}$. 
	Note that $A \simeq A[T] / (T-a)$ and that  $T-a$ is clearly a regular sequence in $A[T]$. 
	By \cite[Proposition 2.7]{hubnerLogarithmicDifferentialsDiscretely2024} we see that \[
		\lctc_\frac{(M_A \oplus \N \to A)}{(M_A \to A)} \simeq \left( \frac{(T-a)}{(T-a)^2} \xrightarrow{-\dd} \lOmega_\frac{(M_A \oplus \N \to A[T])}{(M_A \to A)} \otimes A \right)[-1,0] 
	.\]
	By \cite[Proposition 2.6]{hubnerLogarithmicDifferentialsDiscretely2024} the module in degree $0$ is $A\cdot \dlog a$. 
	As $T-a$ is regular, $(T-a) / (T-a)^2$ is a rank 1 free $A$-module and its generator $(T-a)$ is mapped to $-\dd a$. 
	So we may identify it with $A\cdot \dd a$.  
\end{proof}
\begin{proposition}\label{prop:lctc_enlargement}
	Let $ (A, v)$ be a scalable $K$-algebra with $\ker v = \{0\} $ and let $w \ge v$ be another semi-valuation on $A$.  
	Let  $\phi \in A$ be a prime element, and let $\mu = w(\phi)$. 
	Consider the enlargement $[A^{\circ}, \phi, \mu]$, and let $M_{A, \phi, \mu}$ be the submodule generated by $\vd(A^{\circ})$ and $\alpha^{-1}\phi^{m}$ with $n\cdot \mu - v(\alpha)\ge 0$. 
	Then there is a natural quasi-isomorphism 
	\begin{align*}
		\lctc_{\textstyle \frac{(M_{A^{\circ}, \phi, \mu} \to[A^{\circ}, \phi, \mu])}{(\vd(A^{\circ}) \to A^{\circ})}} &\simeq \left( [A^{\circ}, \phi, \mu] \cdot \dd(\theta^{-1}\phi) \to  [A^{\circ}, \phi, \mu] \cdot \dlog(\theta^{-1} \phi)\right) \\
																       &\simeq \left( \frac{[A^{\circ}, \phi, \mu]}{(\theta^{-1} \phi) } \cdot \dlog \phi  \right)[0] .
	\end{align*}

	Moreover, if the semi-valuation $v'$ from \Cref{lem:inductive_model_unit_ball} is well-defined, then $M_{[A^{\circ}, \phi, \mu]}$ is the total log structure on $([A^{\circ}, \phi, \mu], v')$ when $\mu \ne \infty$ and the boundary log structure on when $\mu = \infty$. Recall that $[A^{\circ}, \phi, \mu]$ equals $(A, v')^{\circ}$ when $\mu = \infty$ and $(A[\phi^{-1}], v')^{\circ}$  when $\mu < \infty$.
\end{proposition}
\begin{proof}
	By \Cref{prop:inductive_model_form}, we know that $[A^{\circ}, \phi, \mu]$ is the increasing union of a filtered system $(A_i)_{i \in I}$ of $A^{\circ}$-algebras.
	Each is $A_i$ has an explicit presentation as $A_i \simeq A^{\circ}[\underline{X_{i, j}}] / (\underline{f_{i, j}}) $, where the relations $\underline{f_{i, j}}$ form a regular sequence, see \Cref{comp:inductive_model_differentials}.
	Consider the pre-log structure $M_i \to A_i$ where \[
		M_i = (\vd(A) \oplus \phi^{\Z}) \cap A_i = \{\alpha^{-1} \phi^{n} \mid \alpha \in \vd(A), n \in \Z, n \cdot \mu - v(\alpha) \ge 0 \} 
	.\]
	Define $M_0 := \const(A^{\circ}) \oplus (\theta^{-1}\phi)^{\N}$ which is a log structure on $A^{\circ}$.

	Note that the relations $f_{i, j}$ from \Cref{prop:inductive_model_form} identify elements from the monoid $M_0 \oplus \underline{X_{i,j}}^{\N}$.
	We write $\underline {F_{i, j}} $ for the corresponding relations on $\Z \{M_0 \oplus \underline{X_{i,j}}^{\N}\}.$\footnote{We use the notation $R\{M\}$ for a monoid ring and reserve square brackets for adjoining free variables or elements of a larger ring in which $R$ is contained.
	For example, here $A_i = A^{\circ}[M_i]$ but this is not the monoid ring $A^{\circ} \{M_i\} $.}
	The exact same arguments as in \Cref{prop:inductive_model_form}, and \Cref{comp:inductive_model_differentials} show that $\Z \{M_i\}  \simeq \Z \{M_0 \oplus \underline{X_{i,j}}^{\N}\} / (\underline{F_{i,j}})$ and that $(\underline{F_{i,j}})$ is a regular sequence in $\Z \{M_0 \oplus \underline{X_{i,j}}^{\N}\}$.
	Now we see that 
	\begin{align*}
		A^{\circ} \derotimes_{\Z \{M_0\} } \Z \{M_i\} &\simeq A^{\circ} \otimes_{\Z \{M_0\} } K_{\bullet}(\Z \{M_0 \oplus \underline{X_{i,j}}\}, \underline{F_{i, j}})\\
				     &=K_{\bullet}(A^{\circ}[\underline{X_{i, j}}], \underline{f_{i, j}})\\
				     &\simeq A_i
	.\end{align*}
	Here $K_{\bullet}$ denotes the Kozul complex, and we use the well known fact that the Kozul complex associated to a regular sequence is a free resolution. 
	This shows that the diagram 
	\[
		\begin{tikzcd}
			\Z \{M_0\} \dar \rar & \Z \{M_i\} \dar \\
			A^{\circ} \rar & A_i
		\end{tikzcd}
	,\] 
	is a homotopy push-out. Moreover, the pre-log structures $M_0, M_i$ are integral, thus by \cite[Corollary 2.4, Lemma 2.3]{hubnerLogarithmicDifferentialsDiscretely2024} the corresponding diagram of pre-log rings is a homotopy pushout. 
	Then  
	\[
		\lctc_{\textstyle \frac{(M_i \to A_i)}{(M_0 \to A^{\circ})}} \simeq \lctc_{\textstyle \frac{(M_i \to \Z \{M_i\} )}{(M_0 \to \Z \{M_0\})}} \derotimes_{\Z \{M_i\}} A_i \simeq \frac{M_i^{\text{gp}}}{M_0^{\text{gp}}} \derotimes_{\Z}A_i = 0
	,\] 
	where the first quasi-isomorphism is due to Lemma 2.2 of loc.\ cit., the second is due to \cite[Lemma 8.23.(ii)]{olssonLogarithmicCotangentComplex2005} and the last equality holds because both $M_i$ and $M_0$ groupify to $\vd(A) \oplus \phi^{\Z}$.

	By \cite[Proposition 2.9]{hubnerLogarithmicDifferentialsDiscretely2024} the log cotangent complex commutes with filtered colimits, thus by taking the colimit of the complex above we see that \[
		\lctc_{\textstyle \frac{(M_{A^{\circ}, \phi, \mu} \to [A^{\circ}, \phi, \mu])}{(M_0 \to A^{\circ})}} \simeq 0
	.\] 
	Consider the tower of pre-log rings $(\const A^{\circ} \to A^{\circ}) \to (M_0 \to A^{\circ}) \to (M^{w}_{A^{\circ}, \phi, \mu} \to [A^{\circ}, \phi, \mu])$. 
	Then, by the distinguished triangle and the vanishing above, we see that 
	\begin{align*}
		\lctc_{\textstyle \frac{(M^{w}_{A^{\circ}, \phi, \mu} \to [A^{\circ}, \phi, \mu])}{(\const A^{\circ} \to A^{\circ})}} &\simeq \lctc_{\textstyle \frac{(M_0 \to A^{\circ})}{(\const A^{\circ}\to A^{\circ})}} \derotimes [A^{\circ}, \phi, \mu]\\
														 &\simeq \left( [A^{\circ}, \phi, \mu] \cdot \dd(\theta^{-1}\phi) \to  [A^{\circ}, \phi, \mu] \cdot \dlog(\theta^{-1} \phi)\right) 
	,\end{align*} 
	where the last quasi-isomorphism is due to \Cref{lem:lctc_add_to_the_log_structure}.

	We still have to show that the log structures are as claimed when $v'$ is well-defined. 
	Note that $v'(\alpha^{-1}\phi^{n}) = \mu\cdot n - v(\alpha)$, with $\alpha \in A^{\times }$. 
	Then we easily verify that $M_{A^{\circ}, \phi, \mu} = A[\phi^{-1}]^{\times } \cap (A[\phi^{-1}], v')^{\circ } $ if $\mu < \infty$, and $M_{A^{\circ}, \phi, \mu} = A[\phi^{-1}]^{\times } \cap (A, v')^{\circ } $  when $\mu = \infty$. 
\end{proof}

\begin{remark}\label{rem:total_log_structure}
	Continue with the context of \Cref{prop:lctc_enlargement}. 
	The statement leaves open what the log cotangent complex of $[A^{\circ}, \phi, \mu] / A^{\circ}$ is when $v'$ is well-defined, $\mu$ is infinite, and we take the total log structure on $[A^{\circ}, \phi, \mu]$. 
	This easily reduces to the non-log case. 
	Indeed, in this case $\const A^{\circ} = \const (A, v')^{\circ}$. 
	Both are the constants of $A^{\times }$, with positive valuation with respect to $v, v'$ respectively. 
	But the constants are $v'$ stable, see \cref{para:constants_are_stable}, hence both sets are identical. 
	Then as in \eqref{eq:ctc_eq_lctc}, we see that there is a quasi-isomorphism of complexes
	\[
		\lctc_{\textstyle \frac{(\const A^{\circ} \to [A^{\circ}, \phi, \infty])}{(\const A^{\circ} \to A^{\circ})}} \simeq \ctc_{\textstyle \frac{[A^{\circ}, \phi, \infty]}{A^{\circ}}}
	,\] 
	and the latter is described in \Cref{lem:enlargement_differentials} after tensoring with a suitable field $L^{\circ}$. 
\end{remark}
\todo[]{I feel like the following corollary this should hold for $\lctc$ as well, and somehow the proof can be translated. Maybe the 3x3 lemma works?}
\begin{corollary}\label{cor:boundary_log_structure_quotient}
	Continue with the context of \Cref{prop:lctc_enlargement} and suppose that $\mu = \infty$.
	Let $(M_B \to B) \to (\const A^{\circ} \to A^{\circ})$ be any morphism of log rings such that $H_1(\lctc_{\frac{(\const A^{\circ} \to A^{\circ})}{(M_B \to B)}}) = 0$
	Then 
	\begin{align*}
		\lOmega_{\textstyle \frac{(\const A^{\circ} \to A^{\circ} / (\theta^{-1}\phi))}{(M_B \to B)}} \otimes [A^{\circ}, \phi, \mu] \simeq \lOmega_{\textstyle \frac{(M_{A^{\circ}, \phi, \infty} \to [A^{\circ}, \phi, \infty])}{(M_B \oplus \Phi^{\N} \to B[\Phi])}} 
	,\end{align*}
	where $\Phi$ is a free variable mapping to $\theta^{-1}\phi$. 
\end{corollary}
\begin{proof}
	For notational clarity, we omit writing the log structures in the following commutative diagram. 
	\if\thesis1
	\[
	\adjustbox{scale=.9,center}{\begin{tikzcd}[sep=small]
		0 \rar & {[A^{\circ}, \phi, \mu]}\cdot \dd(\theta^{-1}\phi) \rar \dar & {[A^{\circ},\phi, \mu]}\cdot \dlog(\theta^{-1} \phi)\rar\dar & \frac{[A^{\circ}, \phi, \mu]}{(\theta^{-1}\phi)}\cdot \dlog(\theta^{-1}\phi) \dar \rar &0\\
		0 \rar & \lOmega_{\frac{A^{\circ}}{B}} \otimes [A^{\circ},\phi, \mu] \rar \dar & \lOmega _{\frac{[A^{\circ}, \phi, \mu]}{B}} \rar \dar & \lOmega_{\frac{[A^{\circ}, \phi, \mu]}{A^{\circ}}} \rar \dar & 0 \\
		0 \rar   & \lOmega_{\frac{A^{\circ} / (\theta^{-1} \phi)}{B}} \otimes [A^{\circ}, \phi, \mu] \rar & \lOmega_{\frac{[A^{\circ}, \phi, \mu]}{B[\Phi]}} \rar & 0
	\end{tikzcd}}
	\]
	\else 
		\[
	\begin{tikzcd}[sep=small]
		0 \rar & {[A^{\circ}, \phi, \mu]}\cdot \dd(\theta^{-1}\phi) \rar \dar & {[A^{\circ},\phi, \mu]}\cdot \dlog(\theta^{-1} \phi)\rar\dar & \frac{[A^{\circ}, \phi, \mu]}{(\theta^{-1}\phi)}\cdot \dlog(\theta^{-1}\phi) \dar \rar &0\\
		0 \rar & \lOmega_{\frac{A^{\circ}}{B}} \otimes [A^{\circ},\phi, \mu] \rar \dar & \lOmega _{\frac{[A^{\circ}, \phi, \mu]}{B}} \rar \dar & \lOmega_{\frac{[A^{\circ}, \phi, \mu]}{A^{\circ}}} \rar \dar & 0 \\
		0 \rar   & \lOmega_{\frac{A^{\circ} / (\theta^{-1} \phi)}{B}} \otimes [A^{\circ}, \phi, \mu] \rar & \lOmega_{\frac{[A^{\circ}, \phi, \mu]}{B[\Phi]}} \rar & 0
	\end{tikzcd}
	\]
	\fi
	The first two rows are is exact and each column is short exact.  
	By the snake lemma, the third row is also exact.
\end{proof}

\subsection{The log cotangent complex for pure transcendental extensions}\label{sec:the_log-cotangent_complex_for_pure_trancendental_extensions}

\begin{lemma}\label{lem:summary_log_ordinary_augmentation}
	Let $v \le w$ be two semi-valuations on $K[T]$, and let $S$ be a multiplicative system of $w$-stable elements such that $(S^{-1} K[T], v)$ is scalable. 
	Consider $(K[T], w) \to (L, v_L)$  a map of semi-valued rings to a valued field $L$.
	Let $\fra = [v, \phi, \mu]$ be a $w$-optimal ordinary augmentation and let $S'$ be the multiplicative system generated by $S$ and $\phi$ if $\mu \ne \infty$ and $S' = S$ if $\mu = \infty$.
	Write $R_1 = (S^{-1}K[T], v)^{\circ}$ and $R_2 = (S'^{-1}K[T], [\fra])^{\circ}$. 
	Let $\theta \in \const (S^{-1}K[T],v)$ with  $v(\theta) = v(\phi)$.
	\begin{itemize}
		\item 
	If $\mu \ne \infty$, 	
	\if\thesis1
	\begin{align*}
		\lctc_{\textstyle \frac{(\const R_2 \to R_2)}{(\const R_1 \to R_1)}} &\simeq \left( R_2 \dd(\theta^{-1}\phi) \to R_2 \dlog (\theta^{-1}\phi)\right)\\
		& \simeq \left( \frac{R_2}{\theta^{-1}\phi} \cdot \dlog (\theta^{-1}\phi)\right) [0]
	\end{align*}
	\else	
	\begin{align*}
		\lctc_{\textstyle \frac{(\const R_2 \to R_2)}{(\const R_1 \to R_1)}} \simeq \left( R_2 \dd(\theta^{-1}\phi) \to R_2 \dlog (\theta^{-1}\phi)\right)  \simeq \left( \frac{R_2}{\theta^{-1}\phi} \cdot \dlog (\theta^{-1}\phi)\right) [0]
	\end{align*}
	\fi
	Thus $\lOmega_{{(\const R_2 \to R_2)/(\const R_1 \to R_1)}} \otimes L^{\circ}$ is a cyclic module of content $\step \fra$.
		\item If $\mu = \infty$, then for the boundary log structure it is the case that

			\begin{align*}
				\lctc_{\textstyle \frac{(\bls R_2 \to R_2)}{(\const R_1 \to R_1)}} \derotimes_{R_2} L^{\circ} \simeq \left( L^{\circ} \dd(\theta^{-1}\phi) \xrightarrow{0} L^{\circ} \dlog (\theta^{-1}\phi)\right)[-1,0]
			\end{align*}
			whose zeroth and first homology are isomorphic to $L^{\circ}$.
			For the total log structure, we find
			\begin{align*}
				\lctc_{\textstyle \frac{(\const R_2 \to R_2)}{(\const R_1 \to R_1)}} \derotimes_{R_2} L^{\circ} \simeq \left(\frac{L\cdot \dd\phi}{L^{\circ}\cdot \dd(\theta^{-1}\phi)}\right)[0]
			.\end{align*}
	\end{itemize}
\end{lemma}
\begin{remark}
	When $\mu = \infty$ we see that for both the boundary log structure and the total log structure, the zeroth homology of the log cotangent complex has infinite content. 
	However, the former is a free indivisible module, whereas the latter is a divisible torsion module. 
\end{remark}
\begin{proof}[Proof of \Cref{lem:summary_log_ordinary_augmentation}]
	The case where $\mu \ne \infty$ and the result for the boundary log structure when $\mu = \infty$ follow immediately from \Cref{prop:lctc_enlargement}.
	The if $\mu = \infty$ and we consider the total log structure on $R_2$, then we know from \Cref{rem:total_log_structure} that this is the same as the non-log case.
	The result now follows from tensoring \eqref{eq:fin_pres_omega_type_1} with $L^{\circ}$ and taking the colimit. 
\end{proof}

The proof of the analogous statement for limit augmentations in the logarithmic case is fortunately a bit easier than in the non-log case. 
\begin{lemma}\label{lem:summary_log_limit_augmentation}
	Let $v \le w$ be two semi-valuations on  $K[T]$ and  $S$ a multiplicative system of $w$-stable elements such that $(S^{-1}K[T], v)$ is scalable. 
	Let $(K[T], w) \to (L, v_L)$ be a map of semi-valued rings to a valued field $L$.
	Suppose $\fra = (v, (\psi_i, \gamma_i)_{i \in I}, \phi, \mu)$ is a $w$-optimal limit augmentation. 
	Let $S' \subset K[T]$ be the multiplicative set generated by $S$, all $\psi_i, i \in I$, and $\phi$ when $\mu \ne \infty$. 
	For each $i \in I$, write $S_i$ for the multiplicative set generated by $S$ and all $\psi_j$ with $j \le i$, and $v_i = [v, \psi_i, \gamma_i]$. 

	\begin{itemize}
		\item If $\mu \ne \infty$ then
		\if\thesis1
	\begin{align*}
		&\lctc_{\textstyle \frac{( M_2 \to R_2)}{(\const R_1 \to R_1)}} \derotimes  L^{\circ} \simeq \left( \lOmega_{\textstyle \frac{( \const R_2 \to R_2)}{(\const R_1 \to R_1)}}  \otimes L^{\circ} \right) [0] \text{ and} \\
		&\cont\left(\lOmega_{\textstyle \frac{(\const R_2\to R_2)}{(\const R_1 \to R_1)}}  \otimes L^{\circ} \right) = \step \fra
	.\end{align*}
	\else
		\begin{align*}
		\lctc_{\textstyle \frac{( M_2 \to R_2)}{(\const R_1 \to R_1)}} \derotimes  L^{\circ}\simeq \left( \lOmega_{\textstyle \frac{( \const R_2 \to R_2)}{(\const R_1 \to R_1)}}  \otimes L^{\circ} \right) [0] \text{ and }
		\cont\left(\lOmega_{\textstyle \frac{(\const R_2\to R_2)}{(\const R_1 \to R_1)}}  \otimes L^{\circ} \right) = \step \fra
	.\end{align*}
	\fi 
	\item When $\mu = \infty$ then 
	\[
		\cont\left(\lOmega_{\textstyle \frac{(\bls R_2\to R_2)}{(\const R_1 \to R_1)}} \otimes L^{\circ} \right) = \step \fra = \infty
	.\] 
	\end{itemize}
\end{lemma}
\begin{proof}
	We write $R_1:= (S^{-1} K[T], v)^{\circ}$ and $R_2 :=({S'}^{-1}K[T], [\fra])^{\circ}$, and let $M_B$ be the total log structure or boundary log structure on $R_2$ when $\mu < \infty$ or $\mu = \infty$ respectively.
	For each $i \in I$ we write $v_i = [v, \psi, \gamma_i]$, $A_i := (S_i^{-1} K[T], v_i)^{\circ}$, $A_i' := [S^{-1}K[T], \psi_i, \gamma_i]$, $B_i= [A_i,\phi, \mu]$ and $A = \bigcup_{i \in I} A_i$.
	Moreover, for each $i \in I$ we choose a constant $\theta_i \in (S_i^{-1}K[T], v_i)$ with $w(\theta) = v_i(\phi)$ and a constant $\lambda \in (S^{-1}K[T], v)$ with $w(\lambda) = v(\psi_i)$, which is independent of $i$ by \Cref{rem:continuous_family}.
	Finally, we write $M_i = (S^{-1}K[T])^{\times }\cap (S^{-1}K[T]) \cap [A_i, \phi, \mu] $. 
	Like in the proof of \Cref{lem:summary_limit_augmentation}, we know that 
	\[
		R_2 = \bigcup_{i \in I} B_i
	,\] 
	is an increasing union and that for each $i \in I$, $A_i$ is a localization of $A_i'$ by inverting elements which map to invertible elements of $L^{\circ}$.
	Moreover, $(\const A_i \to A_i)$ is the logification of $(\const A_i' \to A_i)$. 
	Thus, by \Cref{lem:summary_log_ordinary_augmentation} \[
		\lctc_{\textstyle \frac{(\const A_i \to A_i)}{(\const R_1 \to R_1)}} \derotimes L^{\circ} \simeq \left( \frac{L^{\circ}}{(\lambda^{-1}\psi_i)} \dlog (\psi_i) \right) 	.\] 
	As for any $j \ge i$ we know that $(v_i, \psi_j, \gamma_j)$ is a well-defined augmentation, by the same argument the tower $(\const R_1 \to R_1) \to (\const A_i \to A_i) \to (\const A_j \to A_j)$
	\[
		0 \to \frac{L^{\circ}}{\lambda^{-1} \psi_i}\dlog \psi_i \into \frac{L^{\circ}}{\lambda^{-1} \psi_j}\dlog \psi_j \to \frac{L^{\circ}}{\kappa^{-1} \psi_j}\dlog \psi_j \to 0
	,\]
	where $\kappa$ is a constant of value $\gamma_j-\gamma_i$. 
	So we see that $\dlog\psi_i = \kappa_{ij} \dlog \psi_j$ where $\kappa_{ij} \in L$ with $v_L(\kappa_{ij})  = \gamma_j - \gamma_i$. 
	More suggestively, we may write the middle module as $\frac{\kappa_{ij}^{-1}\cdot L^{\circ} \dlog \psi_i}{(\lambda^{-1}\psi_i)}$. 
	So taking the colimit over $j$ we see that 
	\[
		\lctc_{\frac{(\const A \to A)}{(\const R_1 \to R_1)}} \derotimes L^{\circ} \simeq \left(\frac{L_{> \lim_j \gamma_j - \gamma_i}}{L_{\ge v_i - v(\lambda)} }\cdot \dlog \psi_i \right)[0] 
	,\] 
	which is a module of content $\lim_{i} \gamma_j - v(\lambda) = \step \frb$, where $\frb$ is the continuous family $(v, \psi_i, \gamma_i)_{i \in I}$. 
	It is not finitely generated, but it is generated by all $\dlog \psi_j$.

	Moreover, by \Cref{prop:lctc_enlargement}, we know that 
	\[
		\lctc_{\textstyle \frac{(M_i \to B_i)}{(\const A_i \to A_i)}} \derotimes L^{\circ} \simeq \theta_i^{-1} L^{\circ} \cdot (\phi\dlog(\phi))\to L^{\circ}\dlog(\phi)[-1, 0]
	.\] 
	Taking the colimits over $i$ of this complex, we see that 
	\begin{align*}
		\lctc_{\frac{(M_2 \to R_2)}{(\const A \to A)}} \derotimes L^{\circ} &\simeq \left(L_{> -\lim_i {v_i(\phi)}} \cdot \phi \dlog \phi \to L^{\circ} \dlog \phi\right)[-1, 0] \\
		&\simeq \begin{cases}
			\displaystyle \left( \frac{L^{\circ}}{L_{> -\lim_i v_i(\phi)+\mu}} \cdot \dlog\phi \right) &\text{ if } \mu < \infty\\
			\displaystyle \left(L_{>-\lim_i v_i(\phi)} \xrightarrow{0} L ^{\circ}\cdot \dlog  \phi \right)[-1, 0] & \text{ if } \mu = \infty
		\end{cases}
	.\end{align*}
	In both cases we see that the content of the zeroth homology is  $\lim_i \mu - v_i(\phi)$. 
	Then the distinguished triangle associated to the tower of log rings $(\const R_1 \to R_1) \to (\const A \to A) \to (M_2 \to R_2)$ shows 
	\[
		\lctc_{\frac{(\const A \to A)}{(\const R_1 \to R_1)}} \derotimes L^{\circ} \to \lctc_{\frac{(M_2 \to R_2)}{(\const R_1 \to R_1)}} \derotimes L^{\circ} \to \lctc_{\frac{(M_2 \to R_2)}{(\const R_1 \to R_1)}} \derotimes L^{\circ} 
	.\] 
	When $\mu < \infty$ then $M_2$ is the total log structure, the outer complexes are concentrated in degree $0$, thus so is the middle one. 
	By the additivity of content we see that \[
		\cont\left(\lOmega_{\textstyle \frac{(\const R_2\to R_2)}{(\const R_1 \to R_1)}} \otimes L^{\circ} \right) = \step \frb + \lim_{i} \mu - v_i(\phi)=  \step \fra
	.\]
	When $\mu = \infty$, the distinguished triangle shows that $L^{\circ}\cdot \dlog \phi$ is a quotient of $\lOmega_{{(\bls R_2\to R_2) / (\const R_1 \to R_1)}} \otimes L^{\circ}$, and thus it has infinite content. 
\end{proof}
\begin{remark}\label{rem:generators_log_differentials}
	From the proof of \Cref{lem:summary_log_limit_augmentation}, we see that
	$\lOmega_{\frac{(M_2\to R_2)}{(\const R_1 \to R_1)}} \otimes L^{\circ}$ is generated by $\dlog \phi$ and $\dlog \psi_i, i \in I$, where $M_2$ is the total or boundary log structure when $\mu < \infty$ or $\mu = \infty$ respectively. 
\end{remark}
\begin{lemma}[$\lctc$ for a continuous family or an almost stable limit augmentation]\label{lem:summary_log_almost_stable}
	Let $v \le w$ be two semi-valuations on  $K[T]$ and  $S$ a multiplicative system of $w$-stable elements such that $(S^{-1}K[T], v)$ is scalable. 
	Let $(K[T], w) \to (L, v_L)$ be a map of semi-valued rings to a valued field $L$.
	Suppose $\fra = (v, (\psi_i, \gamma_i)_{i \in I})$ is a $w$-optimal stable continuous family or $\fra = (v, (\psi_i, \gamma_i)_{i \in I}, \phi, \mu)$ is a $w$-optimal almost stable limit augmentation. 
	Let $S' \subset K[T]$ be the multiplicative set generated by $S$ and all $\psi_i, i \in I$. 
	For each $i \in I$, write $S_i$ for the multiplicative set generated by $S$ and all $\psi_j$ with $j \le i$, and $v_i = [v, \psi_i, \gamma_i]$. 
	We write $R_1:= (S^{-1} K[T], v)^{\circ}$ and $R_2 :=({S'}^{-1}K[T], [\fra])^{\circ}$. 
	\if\thesis1
	\begin{align*}
			&\lctc_{\textstyle \frac{(\const R_2 \to R_2)}{(\const R_1 \to  R_1)}} \derotimes_{R_2} L^{\circ} = \left(\lOmega_{\textstyle \frac{(\const R_2 \to R_2)}{(\const R_1 \to  R_1)}} \otimes_{R_2} L^{\circ}\right)[0], \text{ and}\\ &\cont\left(\lOmega_{\textstyle \frac{(\const R_2 \to R_2)}{(\const R_1 \to  R_1)}} \otimes_{R_2} L^{\circ}\right) = \step \fra
	\end{align*}
	\else 
	\[
		\lctc_{\textstyle \frac{(\const R_2 \to R_2)}{(\const R_1 \to  R_1)}} \derotimes_{R_2} L^{\circ} = \left(\lOmega_{\textstyle \frac{(\const R_2 \to R_2)}{(\const R_1 \to  R_1)}} \otimes_{R_2} L^{\circ}\right)[0], \quad \cont\left(\lOmega_{\textstyle \frac{(\const R_2 \to R_2)}{(\const R_1 \to  R_1)}} \otimes_{R_2} L^{\circ}\right) = \step \fra
	\]
	\fi 
\end{lemma}
\begin{proof}
	The proof is exactly the same as \Cref{lem:summary_almost_stable}, but uses the log cotangent complex instead of the cotangent complex. 
\end{proof}
\begin{lemma}\label{lem:maximal_localisation_log}
	Let $w$ be a semi-valuation on $K[T]$ and let $(K[T], w) \to (L, v_L)$ be a map of semi-valued rings to a valued field $L$. 
	Let $S$ be a multiplicative subset of $K[T]$ such that $A = (S^{-1}K[T], w)$ is scalable and let $(M_B \to B) \to (\const A^{\circ} \to A^{\circ})$ be any morphism of log rings.
	Then 
	\[
		\lctc_{\displaystyle \frac{(\const A^{\circ} \to A^{\circ})}{(M_B \to B)}} \derotimes L^{\circ} \simeq \lctc_{\displaystyle \frac{(\const (K[T]_{\ker w}, w)^{\circ} \to (K[T]_{\ker w}, w)^{\circ})}{(M_B \to B)}} \derotimes L^{\circ}
	.\] 
	Moreover, when $\ker w \ne \{0\} $ then 
	\[
		\lctc_{\displaystyle \frac{(\bls A^{\circ} \to A^{\circ})}{(M_B \to B)}} \derotimes L^{\circ} \simeq \lctc_{\displaystyle \frac{(\bls (K[T]_{\ker w}, w)^{\circ} \to (K[T]_{\ker w}, w)^{\circ})}{(M_B \to B)}} \derotimes L^{\circ}
	.\] 
\end{lemma}
\begin{proof}
	The proof is exactly the same as the proof of \cref{lem:maximal_localisation}. 
	The only thing to note is that $(\const (K[T]_{\ker w}, w)^{\circ} \to (K[T]_{\ker w}, w)^{\circ})$ and  $(\bls (K[T]_{\ker w}, w)^{\circ} \to (K[T]_{\ker w}, w)^{\circ})$ are the logifications of $(\const A^{\circ} \to (K[T]_{\ker w}, w)^{\circ})$ and $(\bls A^{\circ} \to (K[T]_{\ker w}, w)^{\circ})$ respectively. 
\end{proof}

\begin{theorem}\label{thm:computation_main_log}
	Let $w$ be any semi-valuation on $K[T]$ extending  $v_K$ and let $(\fra_i)_{i = 1}^{n}$ be a $w$-optimal augmentation chain on $v_0$ approximating $w$. 
	Consider $(K[T], w) \to (L, v_L)$ a map of semi-valued rings to a valued field $L$.
	Let $S_0$ be a multiplicative set of $w$-stable elements such that $(S_0^{-1}K[T], v_0)$ is scalable. 
	Further, let $S_i$ be the multiplicative set generated by all key polynomials $\phi_j, \psi_{j, k}$ for all $j \le i, k \in J_{j}$, and let $S$ be the multiplicative set generated by all $S_i$.
	Let $R_i = (S_i^{-1}K[T], v_i)^{\circ}$ and $R_w = (S^{-1}K[T], w)^{\circ}$
	Then for each $i \in [0, n]\cap \Z$ we have 
	\begin{enumerate}[(i)]
		\item $\displaystyle \lctc_{\textstyle \frac{(\const R_i \to R_i)}{(\const R_0 \to R_0)}} \derotimes_{R_i} L^{\circ} = \lOmega_{\textstyle \frac{(\const R_i \to R_i)}{(\const R_0 \to R_0)}}\otimes_{R_i} L^{\circ}[0] $,
		\item $\displaystyle \cont \left(\lOmega_{\textstyle \frac{(\const R_i \to R_i)}{(\const R_0 \to R_0)}}\otimes_{R_i} L^{\circ}\right) = \step (\fra_j)_{j = 1}^{i}$.
	\end{enumerate}
	We write $A_w := (K[T]_{\ker w}, w)^{\circ} $, then
	\if\thesis1
	\begin{enumerate}
		\item[(iii)] \resizebox{.95\linewidth}{!}{$\displaystyle 
			\begin{aligned}\lctc_{\textstyle \frac{(\const R_w \to R_w)}{(\const R_0 \to R_0)}}\derotimes_{R_w} L^{\circ} \simeq \lctc_{\textstyle \frac{(\const A_w \to A_w} {(\const R_0 \to R_0)}} \derotimes L^{\circ} 
				\simeq  \left(\lOmega_{\textstyle \frac{(\const R_w \to R_w)}{(\const R_0 \to R_0)}}\otimes_{R_w} L^{\circ}\right)[0] 
			\end{aligned}$},
		\item[(iv)]  \resizebox{.95\linewidth}{!}{$\displaystyle \cont \left(\lOmega_{\textstyle \frac{(\const R_w \to R_w)}{(\const R_0 \to R_0)}}\otimes_{R_w} L^{\circ}\right) 
			=  \cont\left(\lOmega_{\textstyle \frac{(\const A_w \to A_w)}{(\const R_0 \to R_0)}} \derotimes L^{\circ}\right)
			= \step (\fra_j)_{j = 1}^{n} $}.
	\end{enumerate}
	\else
		\begin{enumerate}
		\item[(iii)] {$\displaystyle 
			\begin{aligned}\lctc_{\textstyle \frac{(\const R_w \to R_w)}{(\const R_0 \to R_0)}}\derotimes_{R_w} L^{\circ} \simeq \lctc_{\textstyle \frac{(\const A_w \to A_w} {(\const R_0 \to R_0)}} \derotimes L^{\circ} 
				\simeq  \left(\lOmega_{\textstyle \frac{(\const R_w \to R_w)}{(\const R_0 \to R_0)}}\otimes_{R_w} L^{\circ}\right)[0] 
			\end{aligned}$},
		\item[(iv)]  {$\displaystyle \cont \left(\lOmega_{\textstyle \frac{(\const R_w \to R_w)}{(\const R_0 \to R_0)}}\otimes_{R_w} L^{\circ}\right) 
			=  \cont\left(\lOmega_{\textstyle \frac{(\const A_w \to A_w)}{(\const R_0 \to R_0)}} \derotimes L^{\circ}\right)
			= \step (\fra_j)_{j = 1}^{n} $}.
	\end{enumerate}
	\fi 
	Moreover, when $\ker w \ne \{0\} $, then 
	\begin{itemize}
		\item[(v)]  $\displaystyle \cont \left(\lOmega_{\textstyle \frac{(\bls R_w \to R_w)}{(\const R_0 \to R_0)}}\otimes_{R_w} L^{\circ}\right) = \cont \left(\lOmega_{\textstyle \frac{(\bls A_w \to A_w)}{(\const R_0 \to R_0)}}\otimes_{R_w} L^{\circ}\right) = \infty$.
	\end{itemize}
\end{theorem}
\begin{proof}
	The proof of (i)-(iv) is the same as \Cref{thm:computation_main} with the log cotangent complex instead of the cotangent complex. 

	For (v), note that - as argued at the start of the proof of \Cref{thm:computation_main} - we may assume that $n$ is finite and that $\fra_n$ is an ordinary or almost stable limit augmentation with log-radius equal to $\infty$.
	Then we know that \[
		\cont\left(\lOmega_{\frac{(\bls R_n \to R_n)}{(\const R_{n-1} \to R_{n-1})}} \otimes L^{\circ}\right) = \infty
	,\] by \Cref{lem:summary_log_ordinary_augmentation} and \Cref{lem:summary_log_limit_augmentation}. 
	By the fundamental exact sequence, this is a quotient of \(\cont\left(\lOmega_{\frac{(\bls R_n \to R_n)}{(\const R_{0} \to R_{0})}}\right) \) which also must have infinite content. 
\end{proof}

\begin{remark}\label{rem:total_generators_log_differentials}
	Continue with the notation of \Cref{thm:computation_main_log}. 
	From \Cref{lem:summary_log_ordinary_augmentation}, \Cref{rem:generators_log_differentials} and the proof of \Cref{thm:computation_main_log} we see that $\lOmega_{(\const R_i \to R_i) / (\const R_0 \to R_0)}\otimes L^{\circ}$ is generated by all $\dlog \phi$  where $\phi$ is a key polynomial in $(\fra_j)_{j = 1}^{i}$ including the key polynomials in the continuous families. 
	Likewise, when $\ker w = \{0\}$ we see that $\lOmega_{(\const K(T)^{\circ} \to K(T)^{\circ}) / (\const R_0 \to R_0)}$ is generated by all key polynomials in $(\fra_i)_{i = 1}^{n}$.

	Recall that by  $\dlog \phi$ we actually mean $\dlog(\kappa\cdot \phi)$ with $\kappa \in K$ such that $w(\kappa\cdot \phi) \ge 0$, which is independent of the choice of $\kappa$.
\end{remark}

\section{Extensions of valued fields, (log) differents, and discrepancies}\label{sec:extensions_of_valued_fields}
\todo[inline]{Write some more remarks. I'm afraid readers might get lost.}
\subsection{Finite extensions and their (log) differents} \label{sec:the_different_and_log_different_of_finite_field_extensions}
\begin{para}
Let $(L, v_L) / (K, v_K)$ be a finite extension of valued fields.
The \emph{different} of $L / K$ is defined as $\diff_{L / K} := \cont(\Omega_{L^{\circ} / K^{\circ}})$ and the \emph{log different} is defined likewise as $\ldiff_{L/ K}:= \cont(\lOmega_{L^{\circ} /K^{\circ}})$, where both $K^{\circ}$ and $L^{\circ}$ are equipped with their respective total log structures. 
	The computations from the previous section allow us to give formulas for the (log) differents of finite extensions without much extra work. 

Suppose that $L = K[T] / (f)$ where $f$ is some irreducible polynomial. We may assume $f$ to be monic and have coefficients in $K^{\circ}$, by substituting $f(T)$ for $\alpha^{\deg f}f(\alpha^{-1} T)$ with $v(\alpha)$ sufficiently large. 
Then the valuation  $v_L$ induces a semi-valuation $w$ on $K[T]$ with kernel  $(f)$. 

Let $v_0$ be the Gauss valuation with center $0$ and log-radius $0$. 
Then we may approximate $w$ with a $w$-optimal augmentation chain $(\fra_i)_{i = 1}^{n}$ on $v_0$, see \Cref{cor:optimal_simple_chain}.
Like in the beginning of the proof of \Cref{thm:computation_main}, we may assume that $\fra_n = ([\fra_{n-1}], f, \infty)$ is an ordinary augmentation or an almost stable limit augmentation $\fra_n = ([\fra_{n-1}], (\psi_i, \gamma_i)_{i \in I}, f, \infty)$.
\end{para}
\begin{definition}
	Let $\fra = (v, (\psi_i, \gamma_i)_{i \in I}, \phi, \infty) $ be an almost stable limit augmentation.
	Write $v(\psi) = v(\psi_i)$, which we recall is independent of $i$, see \Cref{rem:continuous_family}.
	We define 
	\[
	p_{\fra}: [v(\psi), \infty] \to \mathcal{V} (K[T]): \lambda \mapsto 
		\begin{cases}
			[v, \psi_j, \lambda] &\text{if } \lambda \le \gamma_j  \\
			[\fra] & \text{if } \lambda = \infty
		\end{cases}
	,\] 
	where $\mathcal{V} (K[T])$ is the set of all semi-valuations on $K[T]$ that extend $v_K$.
	Note that $p_{\fra}$ is well-defined by \cite[Lemma~2.8]{nartMacLaneVaquieChains2021}. 
	Moreover, for any $f \in K[T]$, we define \[
		|f|_{\fra}: [v(\psi), \infty] \to \R\infty, \lambda \mapsto (p_{\fra}(\lambda))(f)
	.\]
\end{definition}
\begin{para}\label{par:slope_almost_stable_limit}
	Let $\fra$ be an almost stable limit augmentation and let $f \in K[T]$ be any polynomial. 
	Take any $b > v(\psi)$ and  $i \in I $ such that $\gamma_i \ge  b$. 
	Let  $f = \sum_{k= 0}^{m} a_k \psi_i^{k}$ be the $\psi_i$ expansion of $f$. 
	Then 
	\[
		|f|_{\fra} |_{[v(\psi), b)]}: \lambda \mapsto \min_{k = 0}^{m} \{v(a_k) + k\cdot \lambda\} 
	,\] 

	from which we easily see that $|f|_{\fra} $ is a non-decreasing piecewise $\Z$-linear function with non-increasing slopes. 
	In particular, the slope of $|f|_{\fra}$ stabilizes to a minimal value. 
\end{para}

\begin{proposition}\label{prop:ctc_finite_ext}
	With notation as above, if $L / K$ is separable (i.e. $f' \ne 0$) then
	\[
		\ctc_{\textstyle \frac{L^{\circ}}{ K^{\circ}}} \simeq \left( \Omega_{\textstyle \frac{L^{\circ}}{ K^{\circ}}} \right)[0], \quad  \text{and} \quad
		\lctc_{\textstyle \frac{(\const L^{\circ} \to L^{\circ})}{(\const K^{\circ} \to K^{\circ})}} \simeq \left( \lOmega_{\textstyle \frac{(L^{\circ} \to L^{\circ})}{(K^{\circ} \to K^{\circ})}} \right)[0]
	.\] 
	If $L / K$ is inseparable, then 
	\begin{align*}
		H_i\left(\ctc_{\textstyle \frac{L^{\circ}}{ K^{\circ}}}\right) =   H_i\left(\lctc_{\textstyle \frac{(\const L^{\circ} \to L^{\circ})}{(\const K^{\circ} \to K^{\circ})}}   \right)  &= 0, \quad i \ge 2 
	,\end{align*} 
	and the first homology is torsion free of rank $1$. 
\end{proposition}
\begin{proof}
	Continue with the notation of \Cref{thm:computation_main}. 
	By our choice of $v_0$ we see that $R_0 = K^{\circ}[T]$.
	Note that  $\const R_0 = \const K^{\circ}$. 
	Then \[
		\ctc_{\frac{R_0}{K^{0}}} \simeq
		\lctc_{\frac{(\const R_0 \to R_0)}{(\const K^{0} \to K^{\circ})}} \simeq \left( R_0\cdot \dd T \right) [0]
	.\] 
	Then by the distinguished triangle associated to the tower $K \to R_0 \to R_n$, equipped with the total log structure when considered as log rings, thus we see that 
	\[
		\ctc_{\frac{R_n}{K^{\circ}}}\derotimes L^{\circ} , \quad \lctc_{\frac{(\const R_n \to R_n)}{(\const K^{\circ} \to K^{0})}} \derotimes L^{\circ}
	\] 
	are concentrated in degree $0$ and the zeroth homology has rank $\ge 1$, recall \Cref{thm:computation_main} and \Cref{thm:computation_main_log}. 
	Note that $R_n \to L^{\circ}$ is surjective with kernel the ideal $\mathcal{I} = (\theta^{-1} f \mid \theta \in \const R_w)$. 
	Let $(\theta_i)$ be a sequence of constants with $v_i(\theta_i)$ increasing and going to  $\infty$, let $\mathcal{I}_i = \theta_i^{-1}f  \cdot R_n$. 
	Then this is an increasing sequence of ideals generated by regular elements $\theta_i^{-1}f$. 
	Thus,
	\[
		\ctc_{\textstyle \frac{R_n / \mathcal{I}_i }{R_n}} \simeq \lctc_{\textstyle \frac{(\const R_n \to R_n / \mathcal{I}_i )}{(\const R_n \to R_n)}}  \simeq \left(\frac{R_n}{(\theta_i^{-1} f)}\theta_i^{-1} f\right)[-1]
	.\] 
	Taking the colimit over $i$ we see that \[
		\ctc_{\textstyle \frac{L^{\circ}}{R_n}} \simeq \lctc_{\textstyle \frac{(\const L^{\circ} \to L^{\circ})}{(\const R_n \to R_n)}} \simeq \left(L\cdot f\right)[-1]
	.\]
	Take the tower of rings $K_0 \to R_n \to L^{\circ}$ and the associated long exact sequence 
	\begin{equation}\label{eq:finite_extension_exact_sequence}
		0 \to H^{1}(\ctc_{\frac{L^{\circ}}{K^{\circ}}}) \to L\cdot f \xrightarrow{\dd}  \Omega_{\frac{R_n}{K^{\circ}}} \otimes L^{\circ} \to \Omega_{L^{\circ} / K^{\circ}} \to 0
	.\end{equation} 
	If $L / K$ is separable, then $f' \ne 0$ and $\dd f$ does not vanish in $\Omega_{L / K}$. 
	Then $L \cdot f$ is not torsion, thus $\dd$ is injective and thus $H_{1}( \ctc_{L^{\circ} / K^{0}})$ vanishes. 

	If $L / K$ is inseparable, then $f' = 0$ and $\dd f$ does vanish in $\Omega_{L / K}$. 
	So the morphism $\dd$ is not injective and therefore $H_{1}(\ctc_{L^{\circ} / K^{\circ}})$ is a non-trivial submodule of $L$, whence it is of the prescribed form. 
	The exact same argument holds for the log cotangent complex.
\end{proof}

\begin{theorem}\label{prop:comp_log_diff}
	Continue with the notation as above. 
	When $\fra_n$ is an ordinary augmentation, the (log)different equals, 
	\begin{equation}\label{eq:comp_log_diff_1}
	\begin{aligned}
		\diff_{L / K} &= \step (\fra_i)_{i =1}^{n-1}- v_{n-1}(f) + v_{n-1}(f') + \left( \inf \Gamma_{K, > 0} - \inf \Gamma_{L, > 0} \right),  \\
		\ldiff_{L / K} &= \step (\fra_i)_{i = 1}^{n - 1} - v_{n-1}(f) + v_{n-1}(f'). 
	\end{aligned}
	\end{equation}
	When $\fra_n$ is an almost stable limit augmentation on a continuous family $\frb = (\frb_i)_{i \in I}$, then 
	\begin{equation}\label{eq:comp_log_diff_2}
	\begin{aligned}
		\diff_{L / K} &= \lim_{i \in I } \step (\fra_j)_{j =1}^{i-1} + \step \frb_i - [\frb_i](f) + [\frb_i](f') + \left( \inf \Gamma_{K, > 0} - \inf \Gamma_{L, > 0} \right),  \\
		\ldiff_{L / K} &= \lim_{i \in I} \step (\fra_j)_{j = 1}^{i - 1} + \step \frb_i- [\frb_i](f) + [\frb_i](f'). 
	\end{aligned}
	\end{equation}
	Moreover, these limits stabilize.
\end{theorem}
\begin{proof}
	Throughout the proof we assume that log structure on the rings in the arguments of the log cotangent complex and log differentials are equipped with the total log structure. 
	Any quotients like $A / \mathcal{I} $ we also equip with the log structure $\const A$. 
	Let $\mathfrak{c} = \fra_{n-1}$ when $\fra_n$ is an ordinary augmentation and let $\mathfrak{c}  = \frb_i$ for some $i \in I$ when $\fra_n$ is an almost stable limit augmentation. 
	Write $v' = [\mathfrak{c}]$.

	If $L / K$ is inseparable, then we see from \eqref{eq:finite_extension_exact_sequence} that $\Omega_{L^{\circ} / K^{\circ}}$ and $\lOmega_{L^{\circ} / K^{\circ}}$ must be rank 1 modules. 
	Therefore, the (log)-different is $\infty$, and the right-hand sides are also $\infty$ as $f' = 0$. 

	Assume that $L / K $ is separable. 
	By \Cref{thm:computation_main} and \Cref{thm:computation_main_log}, there is a multiplicative system $S$ such that $(S^{-1}K[T], v')$ is scalable, and the content of the (log)-differentials of $R' / R_0$ is given by these theorems, where $R' = (S^{-1}K[T], v')^{\circ}, R_0 = K^{\circ}[T] = (K[T], v_0)^{\circ}$. 
	Let $\theta$ be a constant in $(S^{-1}, K[T], v')$ with $w(\theta) = v'(f)$. 
	We will compute  $\ctc_{\frac{R' / (\theta^{-1}f)}{K^{\circ}}}$ by understanding $\ctc_{\frac{R_0 / (f)}{K^{\circ}}}$, $\ctc_{\frac{R' / (f)}{R_0 / (f)}} \derotimes L^{\circ}$, and $\ctc_\frac{R' / (\theta^{-1} f)}{R' / (f)}$. Then we can put this together using distinguished triangles. The same works for the log cotangent complexes. 
	\begin{itemize}
		\item Let's focus our attention on $\ctc_{\frac{R_0 / (f)}{K^{\circ}}}$ and $\lctc_{\frac{R_0 / (f)}{K^{\circ}}}$. 
			We know that  $R_0 = K^{\circ}[T]$  and therefore $f \in R_0$ is not a zero divisor and $f$ is a regular sequence. 
			Hence, $\ctc{\frac{R_0 / (f)}{K^{\circ}} }\simeq \lctc_{\frac{R_0 / (f)}{K^{\circ}} }  \simeq \left((R_0 / (f)) \cdot f \xrightarrow{\dd} (R_0 / (f)) \dd T\right) $ in degree $-1$ and  $0$.
			Applying $-\derotimes_{R_0 / (f)} L^{\circ}$ shows that 
			\if\thesis1 \begin{align*}
				\ctc_{\frac{R_0 / (f)}{K^{\circ}}}\derotimes L^{\circ} &\simeq \lctc_{\frac{R_0 / (f)}{K^{\circ}}} (L^{\circ} \cdot f\xrightarrow{\cdot f'} L^{\circ} \dd T), \text{ and } \\ \cont\left(\frac{L^{\circ} \dd T}{L^{\circ} f' \dd T}\right)&= w(f') = v_{n-1}(f')
			.\end{align*}
			\else
			\[
				\ctc_{\frac{R_0 / (f)}{K^{\circ}}}\derotimes L^{\circ} \simeq \lctc_{\frac{R_0 / (f)}{K^{\circ}}} (L^{\circ} \cdot f\xrightarrow{\cdot f'} L^{\circ} \dd T), \text{ and } \cont\left(\frac{L^{\circ} \dd T}{L^{\circ} f' \dd T}\right)= w(f') = v_{n-1}(f')
			.\] 
			\fi 
			As $L / K$ separable, the map  $\dd$ is injective, and thus these complexes are concentrated in degree $0$. 
		\item Note that $R'/ (f\cdot R') \simeq R' \derotimes_{R_0} R_0 / (f\cdot R_0)$.	
			So by \stacks{08QQ}, we have that $\mathbb L_{\frac{R'}{R_0}} \derotimes _{R'} \frac{R'}{ (f)} \simeq \mathbb L_{\frac{R' / (f)}{R_0 / (f)}}$. 
			Applying $- \derotimes_{R' / f} L^{\circ}$ to both sides and simplifying, we see that $\mathbb L_{\frac{R' / (f)}{R_0 / (f)}} \derotimes_{R' / (f)} L^{\circ} \simeq \mathbb L_{\frac{R'}{R_0}} \otimes_{R'} L^{\circ}$. 
			So by \Cref{thm:computation_main} we find that 
			\begin{align*}
				\mathbb L_{\frac{R' / (f)}{R_0 / (f)}} \derotimes_{R'/ (f)} L^{\circ} &\simeq \Omega_{\frac{R'}{R_0}} \otimes L^{\circ}[0], \\
				\cont\left(\Omega_{\frac{R'}{R_0}}\otimes L^{\circ} \right) &= \step (\fra_i)_{i = 1}^{n-2} + \step \mathfrak{c}  + \left( \inf \Gamma_{K, > 0} - \inf \Gamma_{v', > 0} \right) 
			.\end{align*}
			Using \cite[Proposition~2.1, Lemmas~2.2-2.3]{hubnerLogarithmicDifferentialsDiscretely2024}, the same argument shows that 
			\if\thesis1
			\begin{align*}
				\lctc_{\frac{R' / (f)}{R_0 / (f)}} \derotimes_{R'/ (f)} L^{\circ} &\simeq \lOmega_{\frac{R'}{R_0}} \otimes L^{\circ}[0], \text{ and} \\ 
				\cont\left(\lOmega_{\frac{R'}{R_0}}\otimes L^{\circ} \right) &= \step (\fra_i)_{i = 1}^{n-2} +  \step \mathfrak{c}  
			.\end{align*}
			\else
			\begin{align*}
				\lctc_{\frac{R' / (f)}{R_0 / (f)}} \derotimes_{R'/ (f)} L^{\circ} \simeq \lOmega_{\frac{R'}{R_0}} \otimes L^{\circ}[0], \quad 
				\cont\left(\lOmega_{\frac{R'}{R_0}}\otimes L^{\circ} \right) = \step (\fra_i)_{i = 1}^{n-2} +  \step \mathfrak{c}  
			.\end{align*}
			\fi 
		\item Finally, we turn our attention to $\ctc_{\frac{R' / (\theta^{-1} f)}{R'  / (f)}}$.
			Consider the tower $R' \to R' / (f) \to  R' / (\theta^{-1} f)$, and the corresponding triangle of cotangent complexes, 
			\[
				\ctc_{\frac{R' / (f) }{ R'}} \derotimes L^{\circ} \to \ctc_{\frac{R' / (\theta^{-1} f)}{R'}} \to \ctc_{\frac{R' / (\theta^{-1} f)}{R' / (f)}}
			.\] 
			Note that the both $f$ and $\theta^{-1} f$ are regular sequences in $R'$, therefore we see that 
			\begin{align*}
			\ctc_{\frac{R' / (f) }{ R'}}  \derotimes L^{\circ} \simeq \lctc_{\frac{R' / (f)}{R'}}\simeq  \left(f\cdot L^{\circ}\right)[-1], \\
			\ctc_{\frac{R' / (\theta^{-1} f) }{ R'}}  \derotimes L^{\circ} \simeq \lctc_{\frac{R' / (\theta^{-1}f)}{R'}}\simeq  \left(\theta^{-1} f\cdot L^{\circ}\right)[-1], 		
			\end{align*} 
		Filling this in the triangle above (and the logarithmic analogue), we see that \[
			\mathbb L_\frac{R' / (\theta^{-1} f)}{R' / (f)} \simeq \frac{\theta^{-1} f L^{\circ}}{f L^{\circ}} [-1], \text{ and } \cont\left( \frac{\theta^{-1} f\cdot L^{\circ}}{f\cdot L^{\circ}} \right) = w(\theta) = v'(f)
		.\] 
	\end{itemize}
	The tower of rings $K^{\circ} \to R_0 / (f) \to R' /(f)$ gives the triangle \[
		\mathbb L_{\frac{R_0 / (f)}{K^{\circ}}} \derotimes L^{\circ} \to \mathbb L_{\frac{R_{n-1} / (f)}{K^{\circ}}} \derotimes L^{\circ} \to \mathbb L_{\frac{R_{n-1} / (f)}{R_0 / (f)}} \derotimes L^{\circ}
	.\] 
	So we see that 
	\begin{align*}
		\mathbb L_{\frac{R' / (f)}{K^{\circ}}} \derotimes L^{\circ} &\simeq \left(\Omega_{\frac{R'/(f)}{K^{\circ}}}\otimes L^{\circ}\right) [0],\\
		\cont \left(\Omega_{\frac{R'/(f)}{K^{\circ}}}\otimes L^{\circ}\right) &= \step (\fra_i)_{i = 1}^{n-2} + \step \mathfrak{c}  + \left( \inf \Gamma_{K, > 0} - \inf \Gamma_{v', > 0} \right) + v'(f)
	.\end{align*}
	The tower $K^{\circ} \to R' \to R' / (\theta^{-1} f)$ shows that $\ctc_{\frac{R'(\theta^{-1} f)}{K^{\circ}}}$ is concentrated in degree $0$. 
	Next, consider the triangle corresponding to the tower $K^{\circ} \to R' / (f) \to R' / (\theta^{-1}f)$, 
	\[
		\ctc_{\frac{R / (f)}{K^{\circ}}} \derotimes L^{\circ}\to \mathbb L_{\frac{R' / (\theta^{-1} f)}{K^{\circ}}} \to \ctc_{\frac{R' / (\theta^{-1}f)}{R' / (f)}}
	.\] 
	The corresponding LES reads
	\[
		\ldots \to 0 \to \frac{\theta^{-1} f \cdot L^{\circ}}{f\cdot L^{\circ}} \to \Omega_{\frac{R_{n-1} / (f)}{K^{\circ}}} \otimes L^{\circ} \to \Omega_{\frac{L^{\circ}}{K^{\circ}}} \to 0 \to 0 
	.\] 
	Thus,  
	\[\cont \Omega_{\frac{R' / (\theta^{-1}f)}{K^{\circ}}} = \step (\fra_i)_{i = 1}^{n-2} + \step \mathfrak{c}  - v'(f) + v'(f') + \left( \inf \Gamma_{K, > 0} - \inf \Gamma_{v', > 0} \right) . \]
	The exact same argument shows that $\lctc_{\frac{R' / (\theta^{-1} f)}{K^{\circ}}}$ is concentrated in degree $0$ and that
	\[\cont \lOmega_{\frac{R' / (\theta^{-1}f)}{K^{\circ}}} = \step (\fra_i)_{i = 1}^{n-2} + \step \mathfrak{c}  - v'(f) + v'(f') . 
	\]
	If $\fra_n$ is an ordinary augmentation, then the result follows from noting that  $R' / (\theta^{-1} f) = L^{\circ}$ and that $\const R' \to L^{\circ}$ logifies to $\const L^{\circ} \to L^{\circ}$. 

	If $\fra_n$ is an almost stable limit augmentation on the continuous family $\frb = (\frb_i)_{i \in I}$, then we may choose multiplicative sets $S_i$ as in \Cref{lem:summary_almost_stable}, write $v_i = [\frb_i]$, and write $R_i = (S_i^{-1}K[T], v_i)$. 
	We write $\tilde R_i := R_i / (\theta_i^{-1} f)$. 
	Note that $(\const R_i \to \tilde R_i)$ logifies to $(\const \tilde R_i \to R_i)$. 
	As we have just computed 
	\begin{equation}\label{eq:_comp_log_diff}
	\if\thesis1
	\begin{aligned}
		\cont \lOmega_{{\tilde R_i}/{K^{\circ}}} &= \step (\fra_i)_{i = 1}^{n-2} + \step \frb_i  - v_i(f) + v_i(f') \\
		\cont \Omega_{{\tilde R_i} / {K^{\circ}}} &= \cont \lOmega_{{\tilde R_i}/{K^{\circ}}} + \left( \inf \Gamma_{K, > 0} - \inf \Gamma_{v_i, > 0} \right)
	.\end{aligned}
	\else
		\begin{aligned}
		\cont \Omega_{{\tilde R_i} / {K^{\circ}}} &= \step (\fra_i)_{i = 1}^{n-2} + \step \frb_i  - v_i(f) + v_i(f') + \left( \inf \Gamma_{K, > 0} - \inf \Gamma_{v_i, > 0} \right) \\
		\cont \lOmega_{{\tilde R_i}/{K^{\circ}}} &= \step (\fra_i)_{i = 1}^{n-2} + \step \frb_i  - v_i(f) + v_i(f')
	.\end{aligned}
	\fi
	\end{equation}
	We will show that  $\left(\Omega_{{\tilde R_i} / {K^{\circ}}}\right)_{i \in I}$, and $\left(\lOmega_{{\tilde R_i} / {K^{\circ}}}\right)_{i \in I}$ stabilize. 

	Consider the function $|f|_{\fra_n}: \left[ [\fra_{n-1}](\psi), \infty\right] \to \R(\infty)$, is a piecewise linear function whose slope becomes constant to $s$ from  $\lambda_0$ onwards, see \Cref{par:slope_almost_stable_limit}. 
	Clearly $s \ne 0$ because $f$ is unstable for $\frb$. 
	Suppose that  $s \ge 2$, by evaluating $|f|_{\fra_n}$ in $(\gamma_i)$ we see that  \[
		v_i(f) - \step \frb_i \ge s\cdot (\gamma_i - [\fra_{n-1}](\psi)) + [\fra_{n-1}](f) - (\gamma_i - [\fra_{n-1}](\psi))
	,\] 
	where the right hand side tends to $\infty$ as $i$ increases. 
	Further we know that 
	\[\lOmega_{\frac{R_i / (\theta^{-1}_if)}{K^{\circ}}} = \step (\fra_i)_{i = 1}^{n-2} + \step \frb_i  - v_i(f) + v_i(f')  \ge 0
	\]
	which is non-negative as it is the content of a module. Thus, $v_i(f') \to \infty$. 
	As $\deg f' < \deg f$, we know that $(v_i(f'))$ stabilizes to $\infty$ and therefore $f' = 0$, thus $L / K$ is inseparable. 

	Thus, $s = 1$ whenever $L / K$ is separable. 
	Choose $i \in I$ sufficiently large such that $\Gamma_{v_i}$ has stabilized had $\gamma_i > \lambda_0$. 
	Then let $j > i$.
	Consider the following commutative diagram with exact rows and columns. 
	\[
		\begin{tikzcd}[sep = small]
			& & & 0  \dar \\
		& 0 \dar \rar & 0 \dar \rar  & H_1\left(\lctc_{\textstyle {\tilde R_i} / {\tilde R_j}} \derotimes L^{\circ}\right) \dar \\
		0 \rar & L^{\circ}\cdot \dd (\theta_i^{-1} f) \dar \rar & \lOmega_{\textstyle {R_i} / {K^{\circ}}} \otimes L^{\circ} \dar \rar & \lOmega_{\textstyle {\tilde R_i} / {K^{\circ}}} \otimes L^{\circ} \dar \rar & 0 \\
		0 \rar & L^{\circ}\cdot \dd (\theta_j^{-1} f) \dar \rar & \lOmega_{\textstyle {R_j} / {K^{\circ}}} \otimes L^{\circ} \dar \rar & \lOmega_{\textstyle {\tilde R_j} / {K^{\circ}}} \otimes L^{\circ} \dar \rar & 0 \\
		 & \frac{L^{\circ}}{(\theta_i^{-1}\theta_j)} \dd (\theta_j^{-1} f) \dar \rar{\alpha} & \lOmega_{\textstyle {R_j} / {R_i}} \otimes L^{\circ} \dar \rar & \lOmega_{\textstyle{\tilde R_j} / {\tilde R_i}} \otimes L^{\circ} \dar \rar & 0 \\
		       & 0 & 0 & 0
	\end{tikzcd}
	\] 
	We will show that the map 
	\[
		\alpha: \frac{L^{\circ}}{(\theta_i^{-1}\theta_j)} \cdot \dd (\theta_j^{-1} f)\to \lOmega_{\textstyle{R_j} / {R_i}} = \frac{L^{\circ}}{\kappa^{-1} \psi_j} \dlog (\kappa^{-1}\psi_j) 
	,\] 
	is an isomorphism, where $\kappa$ is a constant with $v_i(\kappa) = v_i(\psi_j) = \gamma_i$, and we have used that $v_j = [v_i, \psi_j, \gamma_j]$ and \Cref{lem:summary_log_ordinary_augmentation}.
	Let $f = \sum_{s = 0}^{m} a_s \psi_j^{s}$ be the $\psi_j$ expansion of $f$. 
	As the slope of $|f|_{\fra_n}$ has stabilized to $1$ we see that $a_1\psi_j^{1}$ is the term with minimal valuation for $v_j$. 
	Then we see that  $\theta_j ^{-1} f = \sum_{s = 0}^{m} \theta^{-1} a_s \kappa^{s}(\kappa^{-1}\psi_j)^{s}$ and as an element of $\lOmega_{\textstyle{R_j} / {R_i}}$ we compute  
	\if\thesis1
	\begin{align*}
		\dd (\theta_j^{-1} f) &= \left(\sum_{s = 1}^{m}\theta_j^{-1} a_s \kappa^{s}(\kappa^{-1}\psi_j)^{s-1}\right) \cdot \dd(\kappa^{-1}\psi_j) \\
		&=  \left(\sum_{s = 1}^{m}\theta_j^{-1} a_s \kappa^{s}(\kappa^{-1}\psi_j)^{s}\right) \cdot \dlog(\kappa^{-1}\psi_j) 
	\end{align*}
	\else
	\[
		\dd (\theta_j^{-1} f) = \left(\sum_{s = 1}^{m}\theta_j^{-1} a_s \kappa^{s}(\kappa^{-1}\psi_j)^{s-1}\right) \cdot \dd(\kappa^{-1}\psi_j) =  \left(\sum_{s = 1}^{m}\theta_j^{-1} a_s \kappa^{s}(\kappa^{-1}\psi_j)^{s}\right) \cdot \dlog(\kappa^{-1}\psi_j) 
	\]
	\fi 
	We easily check that in the last sum the term $\theta_j^{-1} a_1 \kappa (\kappa^{-1} \psi_j)^{k} $ has the lowest valuation equal to $0$, and thus the sum is a unit in $L^{\circ}$.
	So $\alpha$ takes $d(\theta^{-1}_jf)$ to a generator of $\lOmega_{R_j / R_i}$, and thus $\alpha$ is surjective. 
	We find that $\alpha$ is also injective because \[
	w(\theta_i^{-1} \theta_j) = |f|_{\fra_n}(\gamma_j) - |f|_{\fra_n}(\gamma_i) = \gamma_j - \gamma_i = w(\kappa^{-1}\psi_j).
	\] 
	Here we used that the slope of $|f|_{\fra_n}$ in this region is $1$.  

	As $\alpha$ is an isomorphism, the snake lemma shows that $H_1(\lctc_{\tilde R_j/\tilde R_i} \derotimes L^{\circ}) = \lOmega_{\tilde R_j / \tilde R_i } \otimes L^{\circ} = \{0\} $. 
	The same is also true for $\ctc_{\tilde R_j / \tilde R_i} \derotimes L^{\circ}$ as it is isomorphic to $\lctc_{\tilde R_j/\tilde R_i} \derotimes L^{\circ}$, see \Cref{par:log_non_log_same_val_grp}.
	Then the tower of (log) rings $K^{\circ} \to \tilde R_i \to \tilde R_j$ shows that $\Omega_{\tilde R_j / K^{\circ}} \otimes L^{\circ} = \Omega_{\tilde R_i / K^{\circ}} \otimes L^{\circ} $ and thus the system $(\Omega_{\tilde R_j / K^{\circ}} \otimes L^{\circ})_{j \in I}$ stabilizes for $j \ge i$, likewise for the log differentials. 
	Taking the colimit, we see that 
	\[
	\Omega_{L^{\circ} / K^{\circ}} = \Omega_{\tilde R_i / K^{\circ}} \otimes L^{\circ}, \quad 
	\lOmega_{L^{\circ} / K^{\circ}} = \lOmega_{\tilde R_i / K^{\circ}} \otimes L^{\circ}
	,\] 
	whose content were computed in \eqref{eq:_comp_log_diff}.
\end{proof}
\begin{remark}
	Combining the computations from the previous proof with \Cref{rem:explicit_complete_intersection_type_II} we see that when $K$ is discretely valued and $\fra_n$ is ordinary, can we write $L^{\circ} / K^{\circ}$ as the explicit complete intersection \[
		L^{\circ} = \frac{K^{\circ}[T, X_0, Y_0, Z_0, \ldots, X_{n-2}, Y_{n-2}, Z_{n-2}]}{(X_0Y_0-1, f_0, g_0, \ldots, X_{n-2}Y_{n-2}-1, f_{n-2}, g_{n-2}, \theta^{-1}f)}
	,\] 
with notation as in the previous proof and \Cref{rem:explicit_complete_intersection_type_II}. 
\end{remark}

\subsection{Purely transcendental extensions and (log) discrepancies}\label{sec:purely_trancendental_extensions_and_log_discrepancies}
Now we study the case that $L = (K(T), w)$ is a purely transcendental field extension of $(K, v_K)$. 
\begin{proposition}\label{prop:pure_trans_deg_0}
	With notation as above. 
	\begin{align*}
		\ctc_{\textstyle \frac{L^{\circ}}{K^{\circ}}} \simeq \left( \Omega_{\frac{L^{\circ}}{K^{\circ}}} \right)[0], \quad \lctc_{\textstyle \frac{(\const L^{\circ} \to L^{\circ})}{(\const K^{\circ} \to K^{\circ})}} \simeq \left( \lOmega_{\textstyle \frac{(\const L^{\circ} \to L^{\circ})}{(\const K^{\circ} \to K^{\circ})}} \right)[0]
	.\end{align*}
\end{proposition}
\begin{proof}
	Let $(\fra_i)_{i = 1}^{n}$ be a $w$-optimal augmentation chain on a simple Gauss semi-valuation $v_0$ approximating $w$. 
	Let $S_i, R_i$ as in \Cref{thm:computation_main}. 
	We know that $R_0 = K^{\circ}[T']$ for some linear polynomial $T'$. 
	Then \[
		\ctc_{R_0 / K^{\circ}} \derotimes L^{\circ}\simeq \lctc_{\textstyle \frac{(\const R_0 \to R_0)}{(\const K^{\circ} \to K^{\circ})}} \derotimes L^{\circ} \simeq \left( L \cdot \dd T' \right) [0]
	.\] 
	We also know that $\ctc_{L^{\circ} / R_0}$ and $\lctc_{(\const L^{\circ} \to L^{\circ}) / (\const R_0 \to R_0)}$ are concentrated in degree $0$ by \Cref{thm:computation_main} and \Cref{thm:computation_main_log}.
	Then we conclude by taking the distinguished triangle of (log) cotangent complexes associated to the tower $K^{\circ} \to R_0 \to L^{\circ}$. 
\end{proof}

\begin{para}
	Clearly the module of (log) differentials $\Omega_{L^{\circ} / K^{\circ}}$ is of rank $1$ and hence the content is infinite for any such field extension, and this is not a useful notion. 
	However, the (log) discrepancy still yields a useful invariant of a similar flavor. 
\end{para}

\begin{definition}\label{def:(log)-discrepancy}
	Let $(A, v)$ be any scalable $K$ algebra, and $(L, v_L)$ a valued field with with morphism of semi-valued rings $(A, v) \to (L, v_L)$. 
	Let $\omega \in \Omega_{A^{\circ} / K^{\circ}}$. 
	Then we define the \emph{discrepancy} of $\omega$ as  \[
		\disc a = \cont\left( \frac{\Omega_{A^{\circ} / K^{\circ}}\otimes L^{\circ}}{L^{\circ}\cdot w	} \right)  
	.\] 
	Likewise, let $M_A = \const A^{\circ}$,  or $M_A = \bls A^{\circ}$ when defined, and let $\omega \in \lOmega_{(M_A \to A^{\circ})/(\const K^{\circ} \to K^{\circ})}$ then we define the \emph{log discrepancy} of $\omega$ as 
	\[
		\ldisc \omega = \cont\left( \frac{\lOmega_{(M_A \to A^{\circ})/(\const K^{\circ} \to K^{\circ})}\otimes L^{\circ}}{L^{\circ} \cdot w} \right) 
	.\] 
\end{definition}
\begin{remark}
	Note that the choice of $L$ in \Cref{def:(log)-discrepancy} does not matter. 
	Indeed, consider $F = \mathrm{frac}(A)$, then naturally  $F \subset L$ and $L^{\circ}$ is flat over $F^{\circ}$. 
	So by \Cref{lem:cont_invariance}, the (log) discrepancy is the same for either choice. 
\end{remark}
\begin{theorem}\label{thm:disc_and_log_disc_pure_trans}
	Let $(L, v_L) = (K[T]_{\ker w}, w)$ for some semi-valuation $w$ on $K[T]$. 
	Note that if $\ker w = \{0\} $ then $L$ is a purely transcendental field extension. 
	Let $(\fra_i)_{i = 1}^{n}$ be a $w$-optimal augmentation chain starting on a Gauss semi-valuation $v_0$ with center $a$ and log-radius $\mu$. We write $\phi_0 = T-a$. 
	\begin{enumerate}[(i)]
		\item If $\mu \in \Gamma_K$, i.e. $v_0$ is a simple Gauss valuation,		\begin{align*}
			\disc \dd(\theta^{-1} \phi_0) &= \step (\fra_i)_{i = 1}^{n} + (\inf \Gamma_{v_K, > 0} - \inf\Gamma_{w, > 0})\\
			\ldisc \dd(\theta^{-1} \phi_0) &= \step (\fra_i)_{i = 1}^{n}
		,\end{align*}
		 where $\theta \in K$ with $v_K(\theta) = \mu$.

		\item If $\mu = w(\phi_0)$ then \[
		\ldisc \dlog (\phi_0) = \step (\fra_i)_{i = 1}^{n}
		.\] 
	\end{enumerate}
\end{theorem}
\begin{proof}
	We first show (i). 
	Let $R_0 = (K[T], v_0)^{\circ}$. Then $R_0 = K^{\circ}[\theta^{-1}(T-a)]$. 
	We know that 
	 \[
		 \ctc _{R_0 / K^{\circ}} \simeq \lctc_{\textstyle \frac{(\const R_0 \to R_0)}{(\const K^{\circ} \to K^{\circ})}} \simeq  \left( R_0\cdot \dd(\theta^{-1} \phi_0) \right) 
	.\] 
	Recall that we have computed $\ctc_{L^{\circ}/ R_0}$ in \Cref{thm:computation_main} and the logarithmic version in \Cref{thm:computation_main_log}. 
	By the vanishing of the higher homology in said complexes, the first exact sequence associated to $K^{\circ} \to R_0 \to L^{\circ}$ is short exact and reads \[
	0 \to L^{\circ} \cdot \dd(\theta^{-1} \phi_0) \to \Omega_{L^{\circ} / K^{\circ}} \to \Omega_{L^{\circ} / R_0} \to 0
	.\] 
	Thus,  \[
	\disc \dd(\theta^{-1} \phi_0) = \cont \left( \Omega_{L^{\circ} / R_0} \right)  = \step (\fra_i)_{i = 1}^{n} + (\inf \Gamma_{v_K} - \inf(\left<\Gamma_w )\right>\
	.\] 
	The same argument computes the log discrepancy. 

	We move on to (ii). Let $\mu_{-1} \in \Gamma_{K , \le \mu_0}$, and let $v_{-1}$ be $[K, \phi_0, \mu_{-1}]$. 
	Write $R_{-1} := (K[T], v_{-1})^{\circ} = K^{\circ}[\kappa^{-1} \phi_0]$, and $R_0 = (K[T, \phi^{-1}], v_0)^{\circ}$ with $\kappa \in K$ such that $v_K(\kappa) = \mu_{-1}$. 
	Then $R_0 = [R_{-1}, \phi, \mu_0]$ which is scalable. 
	Let $M_{-1} = \const K^{\circ} \oplus (\kappa^{-1}\phi_0)$ and consider the log ring $(M_{-1} \to R_{-1})$.
	Note that $M_{-1} $ is the same log structure as $M_0$ in the proof of \Cref{prop:lctc_enlargement}, where we have shown that $\lctc_{(\const R_0 \to R_0) / (M_{-1} \to R_{-1})} \simeq 0$. 
	So the tower of log rings $(\const K^{\circ} \to K^{\circ}) \to (M_{-1} \to R_{-1}) \to (\const R_0 \to R_0) $ shows that \[
		\lctc_{\textstyle \frac{(\const R_0 \to R_0)}{(\const K^{\circ} \to K^{\circ})}} \simeq \lctc_{\textstyle \frac{(M_{-1} \to R_{-1})}{(\const K^{\circ} \to K^{\circ})}} \derotimes R_0 \simeq (R_0 \cdot\dlog(\theta^{-1} \phi_0))[0]
	.\] 
	Then, the exact sequence of log differentials corresponding to $K^{\circ} \to R_0 \to L^{\circ}$ reads 
	\[
	0 \to L^{\circ} \cdot \dd(\theta^{-1} \phi_0) \to \lOmega_{\textstyle \frac{(\const L^{\circ} \to L^{\circ})}{(\const K^{\circ} \to K^{\circ})}}  \to \lOmega_{\textstyle \frac{(\const L^{\circ} \to L^{\circ})}{(\const R_0 \to R_0)}}  \to 0
	,\] 
	thus, \[
	\ldisc \dlog(\phi_0) = \cont \left( \lOmega_{\textstyle \frac{(\const L^{\circ} \to L^{\circ})}{(\const R_0 \to R_0)}} \right)  = \step (\fra_i)_{i = 1}^{n}
	.\] 
\end{proof}
\begin{para}
	Note that if $\ker w$ is non-trivial, then $\step (\fra_{i})_{1}^{n} = \infty$. 
	So the (log) discrepancy of $\dd f$ with $f$ a linear polynomial that does not generate $\ker w$ will always be $\infty$. 
	However, if we equip $(K[T]_{\ker w}, w)^{\circ}$ with the boundary log structure then there are still elements for which the log discrepancy is a meaningful invariant. 
\end{para}
\begin{remark}
	If $\ker w = 0$ then the function $\ldisc$ uniquely extends to a semi-valuation on $\Omega_{K(T) / K}$, see \Cref{rem:hint_weight}. 
	As $\Omega_{K(T) / K}$ is a one dimensional vector space, the value of $\ldisc$ of one non-zero element is sufficient to fix it for any element. 
\end{remark}
\begin{remark}
	We can interpret \Cref{thm:disc_and_log_disc_pure_trans} geometrically on the valuative tree. 
	Recall that if $(\fra_i)_{i = 1}^{n}$ is an augmentation chain on $v_0$ approximating $w$, then we may think of $\step (\fra_i)_{i = 1}^{n}$ as the distance between $v_0$ and $w$. 
	So, for any linear polynomial $aT + b$ in $L^{\circ}$ $\ldisc \dd(aT + b)$ measures the distance between $w$ and the Gauss valuation $v_0 = [v_K, T + b/a, -v(a)]$.
	Likewise, $\ldisc \dlog(aT + b)$ measures the minimal distance from $w$ to the line $\{[v_K, aT + b, \lambda] \mid \lambda \in \R(\infty)\}$.

	This can be made precise when $K$ is complete and the valuative tree is identified with the Berkovich affine line, in which case the log discrepancy is a generalization of the weight function as defined in \cite{mustataWeightFunctionsNonArchimedean2015}. 
	Then $\dlog (aT + b)$ is the unique canonical form up to scaling on the marked projective line $(\pro_K^{1, \text{an}}, \{\infty, -b / a\} )$.
	The line $\{[v_K, aT + b, \lambda] \mid \lambda \in \R(\infty)\}$ gives precisely the skeleton of $(\pro_K^{1, \text{an}}, \{\infty, -b / a\} )$, and $\ldisc \dlog (aT + b)$ measures how far a point is from the skeleton. 
	The author will explore this in upcoming work.
\end{remark}

\begin{theorem}
	Let $w$ be a semi-valuation on $K[T]$ with $\ker w = (f)$. 
	We assume $f$ to be a monic polynomial and equip $(K[T]_{\ker w}, w)^{\circ}$ with the boundary log structure. 
	Then $\ldisc \dlog f = \ldiff_{(K[T] / (f), w) / K}$. 
\end{theorem}
\begin{proof}
	Write $(M_A \to A)$ for the log ring $(\bls (K[T]_{\ker w}, w)^{\circ} \to (K[T]_{\ker w}, w)^{\circ})$ and $L = K[T] / (f)$
	When no log structure is written in the arguments of the log differentials, we assume that the rings are equipped with the total log structure. 
	We may again consider a $w$-optimal augmentation chain $(\fra_i)_{i = 1}^{n}$ on a simple Gauss semi-valuation $v_0$, approximating $w$ with $\fra_n$ either an ordinary augmentation or an almost stable limit augmentation, and in both cases with (limit) key polynomial $f$.
	If $\fra_n = ([\fra_{n-1}], f, \infty)$ is an ordinary augmentation then \Cref{cor:boundary_log_structure_quotient} shows that 
	\[
		\frac{\lOmega_{(M_A \to A) / K^{\circ}}}{A \cdot \dlog f} \otimes L^{\circ}\simeq \lOmega_{\textstyle \frac{(M_A \to A)}{(\const K^{\circ} \oplus f^{\N} \to K^{\circ}[f])}} \otimes L^{\circ} \simeq \lOmega_{L^{\circ} / K^{\circ}}
	.\] 
	The equality follows from comparing the content of these modules. 

	Now assume that $\fra_n  = ([\fra_{n-1}], (\psi, \gamma)_{i \in I}, f, \infty)$ is an almost stable limit augmentation. 
	Let $S_i$, $R_i, \tilde R_i, v_i$ and $\theta_i$ be as in the almost stable limit case in the proof of \Cref{prop:comp_log_diff}.
	Let $M_i$ be the boundary log structure on $R_i$. 
	Then like above we see that 
	\[
		\frac{\lOmega_{(M_i \to R_i)/ K^{\circ}}}{ R_i\cdot \dlog f} \otimes L^{\circ}\simeq \lOmega_{\textstyle \frac{(M_i \to A)}{(\const K^{\circ} \oplus f^{\N} \to K^{\circ}[f])}} \otimes L^{\circ} \simeq \lOmega_{R_i / K^{\circ}} \otimes L^{\circ}
	.\]
	The result follows from taking the colimit over  $i$ and comparing content. 
\end{proof}

\subsection{Arbitrary field extensions}\label{sec:arbitrary_field_extensions}

The following theorem is a generalization of \cite[Theorem 6.5.12]{gabberAlmostRingTheory2003} which also includes the analogous result for the log cotangent complex.  
\begin{theorem}\label{thm:lctc_higher_homology}
	Let $(L, v_L) / (K, v_K)$ be any extension of valued fields. 
	Let $\mathcal{L} $ be either $\ctc_{L^{\circ} / K^{\circ}}$ or $\lctc_{(\const L^{\circ} \to L^{\circ}) / (\const K^{\circ} \to K^{\circ})}$. 
	Then $H_1(\mathcal{L} )$ is torsion free and for any $i\ge 2$, $H_i(\mathcal{L} ) = \{0\} $. 
	Moreover, $H_1(\mathcal L)$ is trivial if and only if $L / K$ is separable.
\end{theorem}
\begin{proof}
	We write $\mathcal{L}_{L^{\circ} / K^{\circ}} $ for either the (non-log) cotangent complex or the log cotangent complex, where $L^{\circ}$ and $K^{\circ}$ are equipped with their respective total log structure. 
	We first show that 
	\begin{itemize}
		\item 	$H_{\ge 2}(\mathcal{L}_{L^{\circ} / K^{\circ}})$ vanishes, and 
		\item $H_1(\mathcal{L} _{L^{\circ} / K^{\circ}})$ is torsion free and vanishes if $L / K$ is separable. 
	\end{itemize}
	These properties will be referred to as $(\star)$.
	Later, we will show that $H_1(\mathcal{L} _{L^{\circ} / K^{\circ}})$ does not vanish in the inseparable case.

	Suppose that $L / K$ is finitely generated, i.e.,\ $L = K(x_1, \ldots, x_n)$ for some $x_1, \ldots, x_n \in L$. 
	We induct on $n$. 
	Let $F = K(x_1, \ldots, x_{n-1})$.
	Consider the tower of (log) rings $K^{\circ} \to F^{\circ} \to L^{\circ}$. 
	The associated distinguished triangle of (log) cotangent complexes is \[
	\mathcal{L}_{F^{\circ} / K^{\circ}} \derotimes L^{\circ} \to \mathcal{L} _{L^{\circ} / K^{\circ}} \to \mathcal{L} _{L^{\circ} / F^{\circ}} 
	.\] 
	By induction, $\mathcal{L} _{F^{\circ} / K^{\circ}}$ satisfies $(\star)$. 
	The extension  $L / F$ is either finite or purely transcendental. So $\mathcal{L} _{L^{\circ} / K^{\circ}}$ satisfies $(\star)$ by \Cref{prop:ctc_finite_ext} and \Cref{prop:pure_trans_deg_0} respectively. 
	From the long exact sequence associated with the distinguished triangle, we easily verify $(\star)$ for $\mathcal{L} _{L^{\circ} / K^{\circ}}$.  

	For general $L / K$ we may conclude $(\star)$ by taking the colimit over all subfields of $L$ that are finitely generated over $K$. 

	Suppose that $L / K$ is inseparable. 
	Then $H_1(\mathcal{L} _{L ^{\circ} / K^{\circ}}) \otimes L \simeq H_1(\ctc_{L / K}) \ne 0$. 
	Therefore, $H_1(\mathcal{L} _{L ^{\circ} / K^{\circ}}) \ne 0$. 
\end{proof}

\section{Semi-valuations on differentials and the absolute log different} \label{sec:the_kahler_valuation_and_the_absolute_log_different}
\subsection{The Kähler semi-valuation and key polynomials} \label{sec:the_kahler_norm_and_key_polynomials}

\begin{para}
	Let $(A, v_A) \to (B, v_B)$ be a morphism of semi-valued rings. 
	In \cite{temkinMetrizationDifferentialPluriforms2016}, Temkin defines the Kähler norm on $\Omega_{B / A}$ as the maximal seminorm such that the map $\dd: (B, v_B) \to (\Omega_{B / A})$ is non-expansive. 
	In our additive language, we define the \emph{Kähler semi-valuation}, $\kv: \Omega_{B / A} \to \R(\infty)$ as the minimal semi-valuation on $\Omega_{B/ A}$ such that $\kv(\dd b) \ge v_B(b)$ for all $b \in B$.  
	When considering multiple morphisms of rings, we might decorate $\kv$ as $\kv_{B / A}$ or $\kv_{v_B}$ to avoid confusion. 
\end{para}

\begin{para}
	One easily checks that the Kähler semi-valuation is compatible with localization in the following way. 
	Let $S$ be a multiplicative set of $B$ with $\infty \not\in v_B(S)$ , then $\kv_{S^{-1}B / A}: \Omega_{S^{-1}B / A} \to \R(\infty): \omega / s \mapsto \kv_{B / A}(\omega) - w(s)$. 
\end{para}

\begin{para}
	Now suppose that $(L, v_L) / (K, v_K)$ is an extension of valued fields. 
	Then $\kv_{L / K}$ is related to the $\lOmega_{L^{\circ} / K^{\circ}}$ as follows. 
	The morphism $\lOmega_{L^{\circ} / K^{\circ}} \to \Omega_{L / K}$ naturally identifies $(\lOmega_{L^{\circ} / K^{\circ}})_{tf} $ as a semilattice in $\Omega_{L / K}$, and this semilattice is an almost unit ball of $\kv_{L / K}$.
	This is easily seen from the observation that $\kv(\dlog f) \ge 0$ is equivalent to $\kv f \ge v_L(f)$ for every $f \in L$.
	See \cite[Theorem~5.1.8]{temkinMetrizationDifferentialPluriforms2016} for more details. 
\end{para}

\begin{para}
	A priori, evaluating the Kähler semi-valuation seems like a difficult problem. 
	When $B$ is uncountable, $\kv$ is defined in terms of an uncountable number of inequalities $\kv(\dd b) \ge v_B(b)$ for any $b \in B$. 
	One may wonder whether there is a smaller and more tractable subset of $B$ which is sufficient to define $\kv$. 
	The main result of this section \Cref{thm:comp_kahler} roughly states that when $(A, v_A) = (K, v_K)$ and $(B, v_B) = (K[T], w)$, such a set is given by the key polynomials in a MacLane-Vaquié chain approximating $w$. 
\end{para}
\begin{para}
	For the rest of this section, let $(K[T], w)$ be a semi-valued $(K, v_K)$ algebra with $\ker w = \{0\} $. 
	Let $A \subset K[T]$ be a subset. 
	A semi-value $u$ on $\Omega_{K[T] / K} = K[T] \cdot \dd T$ is \emph{$A$-non-expansive} if for any $a \in A$ we have that $w(\dd a) \ge v(a)$.
	We say $u$ is \emph{non-expansive} if it is $K[T]$-non-expansive.
	Let $(\fra_i)_{i = 1}^{n}$ be a $w$-optimal augmentation chain on a simple Gauss valuation $v_0 = [v_K, T-a, w(T-a)]$ approximating a valuation $w'$. 
	Let $\keys((\fra_i)_{i = 1}^{n})$  be the set consisting of $T-a$ and all key polynomials in $(\fra_i)_{i = 1}^{n}$, including the key polynomials in the unstable families of the limit augmentations.
	Further, we write $K[T]_{v_n = w} = \{f \in K[T] \mid v_n(f) = w\}$.
\end{para}

\begin{lemma}\label{lem:nonexpansive_keys}
	With notation as above, let $u$ be a semi-valuation on $\Omega_{K[T] / K}$.
	Then, the following statements are equivalent,
	\begin{itemize}
		\item[(i)] $u$ is $\keys((\fra_i)_{i = 1}^{n})$ non-expansive,
		\item[(ii)] $u$ is $K[T]_{w' = w}$ non-expansive.
	\end{itemize}
\end{lemma}

\begin{proof}
	As usual, we write $v_i = [\fra_i]$. 
	Note that $\keys((\fra_i)_{i = 1}^{n})$ is a subset of $K[T]_{v_n = w}$ and thus the implication from (ii) $\implies$ (i) is trivial.  
	We will prove the implication (i) $\implies$ (ii) by induction on $n$. 
	First suppose that $n = 0$, and let $f \in (K[T], v_0)$ be $w$-stable. 
	Let $f = \sum_{i = 0}^{m}a_i (T-a)^{i}$ be the $T-a$ expansion. 
	Then $f'\cdot (T-a) = \sum_{i = 0}^{m} i \cdot a_i (T-a)^{i}$.
	Clearly
	\[
	w(f) = v_0(f) \le v_0(f'(T-a)) \le w(f'(T-a)) \le w(f')\cdot u(\dd T-a) = u(\dd f)
	,\]
	thus $u$ is $K[T]_{v_0 = w}$-non-expansive.

	Suppose that the theorem holds for chains of length $n-1$.
	We first consider the case in which $\fra_n = (v_{n-1}, \phi_n, \mu_n)$ is an ordinary augmentation. 
	Let $f \in K[T]_{v_n = w}$ and let $f = \sum_{i = 0}^{m}a_i \phi_n ^{i}$ be the $\phi_n$ expansion. 
	Then 
	\begin{align*}
		w(f) = v_0(f) &= \min_{i} \{i\cdot w(\phi_n) + w(a_i)\}   \\
			      &\le \min_{i} \{i w(\phi_n) + u(\dd a_i), w(a_i) + (i-1) w(\phi_n) + u(\dd \phi_n)\} \\
			      &\le u \left( \sum_{i = 0}^{m} \phi_n^{i} \dd a_i + a_i\cdot i\cdot \phi_n^{i-1} \dd \phi_n \right) = u (\dd f)
	.\end{align*}
	Here we use that $u(\dd \phi_n) \ge w(\phi_n)$ as $u$ is $\phi_n$-non-expansive, and that $u(\dd a_i) \ge w(a_i)$ due to the induction hypothesis and the fact that $\deg a_i < \deg \phi_n$, thus $a_i \in K[T]_{v_{n-1} = w}$. 
	If $\fra_n = (v_{n-1}, (\psi_j, \gamma_j)_{j \in J})$ is a stable continuous family, then the same holds when using the $\psi_j$ expansion of $f$ where $i$ is chosen such that $[v_{n-1}, \psi_j, \gamma_j](f) = w(f)$. 

	Now suppose that $\fra_n = (v_{n-1}, (\psi_j, \gamma_j)_{j \in J}, \phi_n, \mu_n)$ is a limit augmentation. 
	Again, let $f \in K[T]_{v_n = w}$ and let $f = \sum_{i = 1}^{m} a_i \phi_n ^{i}$ be the $\phi_n$ expansion. 
	Choose some $j \in J$ sufficiently large such that $[v_{n-1}, \psi_j, \gamma_j](a_i) = w(a_i)$ for all $a_i$. 
	Write $\frb = (v_{n-1}, \psi_j, \gamma_j)$. 
	Then $(\fra_1, \ldots, \frb)$ is a $w$-optimal augmentation chain of length $n$ for which we have just shown that the theorem holds. 
	As $a_i \in K[T]_{[\frb] = w}$ we know that $u(\dd a_i) \ge w(a_i)$. 
	The same computation as in the ordinary case shows that  $u(\dd f) \ge w(f)$, hence the theorem holds for finite $n$. 

	Finally, we assume that $n = \infty$. 
	Then the values $(v_i)_{i =1}^{n}$ have a stable limit. 
	Let $f \in K[T]_{w' = w}$, then for some sufficiently large $m$, $f \in K[T]_{v_m = w}$. 
	Then $u$ is $\keys((\fra_i)_{i = 1}^{m})$-non-expansive, and thus $u(\dd f) \ge w(f)$. 
\end{proof}

\begin{theorem}\label{thm:comp_kahler}
	Let $w $ be a valuation on $K[T]$ and let $(\fra_i)_{i = 1}^{n}$ be a $w$-optimal augmentation chain approximating $w$ on a Gauss valuation $v_0 = [v_K, T-a, w(T-a)]$.
	Then $\kv_{w}$ is the unique semi-valuation on $\Omega_{K[T] / K}$ as a $(K[T], w)$-module with 
	\[
	\kv_w(\dd T) = \sup \{w(\phi) - w(\phi') \mid \phi \in \keys\left( (\fra_i)_{i = 1}^{n}  \right) \} 
	.\] 
\end{theorem}

\begin{proof}
	By \Cref{lem:nonexpansive_keys} $\kv$ is the minimal semi-valuation for which $\kv(\dd\phi) = w(\phi') + \kv(\dd T) \ge w(\phi)$ for every $\phi \in  \keys\left( (\fra_i)_{i = 1}^{n}\right)$. 
	Equivalently, $\kv$ is the minimal semi-valuation for which $\kv(\dd T) \ge w(\phi) - w(\phi')$ for each $\phi \in \keys\left( (\fra_i)_{i = 1}^{n}\right)$, from which the theorem follows. 
\end{proof}

\begin{remark}
	We give this proof of \Cref{thm:comp_kahler}, because it is very elementary.
	However, we can also deduce this from the computations in \Cref{sec:the_computations}. 
	We sketch an alternative proof.
	The set $\{\dlog \phi \mid \phi \in  \keys((\fra_i)_{i = 1}^{n})\} $ generates $\Omega^{\text{log}}_{K(T)^{\circ} / K^{\circ}}$ by \Cref{rem:total_generators_log_differentials}. 
	For this module to be an almost unit ball of $\kv$, $\kv$ has to be the smallest semi-valuation such that $\kv(\dlog \phi) \ge 0$ for all $\phi \in \keys( (\fra_i)_{i = 1}^{n})$. 
	The condition $\kv(\dlog \phi) \ge 0$ is equivalent to $\kv(\dd\phi) \ge w(\phi)$ which is equivalent to $\kv(\dd T) \ge w(\phi) - w(\phi')$.
\end{remark}

\subsection{The absolute log different} \label{sec:the_absolute_log_different}

\begin{definition}
	Let $L / K$ be an extension of valued fields. 
	The \emph{absolute} log different of $L / K$ is 
	\[
	\aldiff_{L / K} = \cont\left(\left(\Omega_{\frac{(\const L^{\circ} \to L^{\circ})}{(\const K^{\circ} \to K^{\circ})}}^{\text{log}}\right)_{\text{tor}}\right)
	,\]  when $L / K$ is separable and $\aldiff_{L / K} = \infty$ when $L / K$ is inseparable. 
\end{definition}
\begin{para}
	This is a version of the log different that also makes sense for non-algebraic field extensions and was introduced in \cite{temkinMetrizationDifferentialPluriforms2016}.  
	Note that in \cite{temkinMetrizationDifferentialPluriforms2016} this is simply called the log different. 
	But as this is quite a more subtle invariant than the classical log different, we prefer to give it a different name. 
\end{para}

\begin{proposition}\label{prop:ldisc_kv_ldiff}
	Let $(L, v_L) / (K, v_K)$ be a separable extension of transcendence degree 1 (e.g.,  $L = K(T)$) and let $\omega \in \lOmega_{L^{\circ} / K^{\circ}}$. 
	Then $\ldisc \omega = \kv(\omega_L) + \ldiff_{L / K}$, where $\omega_L$ denotes the image of $\omega$ in $\Omega_{L / K}$. 
\end{proposition}
\begin{proof}
	If $\omega$ is torsion, then $\ldisc \omega = \infty$ and  $\omega_L = 0$, thus $\kv \omega_L = \infty$ and the equality holds. 
	Suppose that $\omega$ is not torsion. 
	Then we easily verify that the following sequence is SES. 
	\[
		0 \to \left(\lOmega_{L^{\circ} / K^{\circ}}\right)_{\text{tor}} \to \frac{\lOmega_{L^{\circ} / K^{\circ}}}{L^{\circ}\cdot \omega} \to \frac{\left(\lOmega_{L^{\circ} / K^{\circ}}\right)_{\text{tf}}}{L^{\circ}\cdot \omega} \to 0
	.\] 
	The content of the first two modules are by definition $\ldiff_{L / K}$ and  $\ldisc \omega $ respectively. 
	As $\left(\lOmega_{L^{\circ} / K^{\circ}}\right)_{\text{tf}}$ is an almost unit ball of $\kv$ we see that the content of the third module is $\kv \omega_L $. 
	The result follows from the additivity of content.
\end{proof}
\begin{remark}\label{rem:hint_weight}
	When $K$ is complete and discretely valued and $L$ discretely valued then $\kv(\omega_L) + \ldiff_{L / K}$ is actually equal to the weight of $\omega$ in \cite{mustataWeightFunctionsNonArchimedean2015}, see \cite[Theorem 8.3.3]{temkinMetrizationDifferentialPluriforms2016}.
	So one may characterize the weight semi-valuation $\wt$ on $\Omega_{L / K}$ as the unique semi-valuation such that $\ldisc(\omega) = \wt(\omega_L)$ for every $\omega \in \lOmega _{L^{\circ} / K^{\circ}}$.
\end{remark}

\begin{para}
	If we can find an $\omega \in \Omega_{L^{\circ} / K^{\circ}}$ such that $\ldisc \omega$,  $\kv \omega_L$ are computable and $\kv \omega_L \ne \infty$, then we can also compute the absolute log different $\aldiff_{L / K}$.
	This is essentially what the first part of the following theorem states. 
	The next part shows that even if there are no such elements, we can still compute the absolute log different by considering the log different on intermediate approximations of $w$, which on Berkovich spaces can be seen as a sort of continuity property of the absolute log different at type IV and non-rigid type I points. 
\end{para}
\begin{theorem}\label{thm:comp_log_different}
	Let $w$ be a valuation on $L = K(T)$ and let $(\fra_{i})_{i = 1}^{n}$ be a $w$-optimal augmentation chain on a Gauss valuation $v_0 = [v_K, \phi_0, w(\phi_0)]$ approximating $w$. 
	If $\step (\fra_{i})_{i = 1}^{n}< \infty$, then 
	\[
		\ldiff_{(K(T), w) / K} = \step (\fra_{i})_{i = 1}^{n} + w(\phi_0) - \sup \{w(\phi) - w(\phi') \mid \phi \in \keys\left( (\fra_i)_{i = 1}^{n}  \right) \} 
	.\] 
	If $\step (\fra_i)_{i = 1} = \infty$, then 
	\[
		\aldiff_{(K(T), w) / K} = \lim_{j \in I} \aldiff_{(K(T), v_j) / K}
	,\] 
	where either 
	\begin{itemize}
		\item $n = \infty$, $I = \N$, and $v_j = [\fra_j]$ or
		\item $n$ is finite, $\fra_n = ([\fra_{n-1}], (\psi_i, \gamma_i)_{i \in I})$ is a stable limit, and $v_j = [[\fra_{n-1}], \psi_j, \gamma_j] $. 
	\end{itemize} 
	Moreover, $(\aldiff_{(K(T), v_j) / K})_{j \in I}$ is non-decreasing, and each element of the sequence can be computed as in the first case. 
\end{theorem}
\begin{para}
	Note that the case in which $\step (\fra_i)_{i = 1} = \infty$ and $n < \infty$ does not occur when $K$ is complete.
	When $n = \infty$, then the result can be written in the perhaps simpler form,  
	\[
	\aldiff_{(K(T), w) / K} = \lim_{j \in \N} \step (\fra_{i})_{i = 1}^{j} + w(\phi_0) - \sup \{w(\phi) - w(\phi') \mid \phi \in \keys( (\fra_i)_{i = 1}^{j}  ) \}
	.\] 
\end{para}
\begin{proof}[Proof of \Cref{thm:comp_log_different}]
	First we assume that $\step (\fra_i)_{i = 1}^{n} < \infty$.
	Then by \Cref{prop:ldisc_kv_ldiff}, $\ldiff_{(K(T), w) / K} = \ldisc(\dlog \phi_0) - \kv(\dlog \phi_0)$ as long as one of the terms on the right is finite. 
	The result now follows by filling in the results of \Cref{thm:comp_kahler} and part (ii) of \Cref{thm:disc_and_log_disc_pure_trans}. 

	We now treat the case where $\step(\fra_i) = \infty$. 
	We assume that all log rings are equipped with their standard log structure. 
	Recall that there is an increasing union of $K^{\circ}$ algebras $R_i, i \in I$ such that $(K(T), w)^{\circ}$ is the localization obtained by inverting all elements of valuation $0$ in $R_{\infty} = \bigcup_{i} R_i$, and for each $j \ge i \in I$, $\lctc_{R_j / R_i} \derotimes L^{\circ} = \lOmega_{R_j / R_i}\otimes L^{\circ}[0]$.
	When $n = \infty$, this is due to Theorems \ref{thm:computation_main} and \ref{thm:computation_main_log} and when $n < \infty$ and $\fra_n$ is a stable continuous family, this is due to Lemmas \ref{lem:summary_almost_stable} and \ref{lem:summary_log_almost_stable}.
	Write $R_0 = (K[T, \phi_0^{-1}], v_0)^{\circ}$. 
	Moreover, when $n = \infty$ we let $\mathfrak{A}_j  = (\fra_{i})_{i = 1}^{j}$ and when $n < \infty$ we let  $\mathfrak{A} _j = \left(\fra_1, \ldots, \fra_{n-1}, ([\fra_{n-1}], \psi_j, \gamma_j)\right)$. 
	Then $\mathfrak{A}_j$ approximates $v_j$.

	For $j \ge i$, we may consider the map of SES,
	\if\thesis1
	\[
	\begin{tikzcd}[column sep=small]
		0 \rar & \left( \lOmega_{R_i / K^{\circ}} \derotimes L^{\circ} \right)_{\text{tor}} \rar \dar{f} & \displaystyle \frac{\lOmega_{R_i / K^{\circ}} \derotimes L^{\circ} }{L^{\circ}\cdot \dlog \phi_0} \rar \dar{g} &  \displaystyle \frac{\left(\lOmega_{R_i / K^{\circ}} \derotimes L^{\circ}\right)_{\text{tf}}}{L^{\circ}\cdot \dlog \phi_0} \rar \dar & 0 \\
		0 \rar & \left( \lOmega_{R_j / K^{\circ}} \derotimes L^{\circ} \right)_{\text{tor}} \rar & \displaystyle \frac{\lOmega_{R_j / K^{\circ}} \derotimes L^{\circ} }{L^{\circ}\cdot \dlog \phi_0} \rar &  \displaystyle \frac{\left(\lOmega_{R_j / K^{\circ}} \derotimes L^{\circ}\right)_{\text{tf}}}{L^{\circ} \cdot \dlog \phi_0} \rar & 0 
	\end{tikzcd}
	.\] 
	\else
	\[
	\begin{tikzcd}
		0 \rar & \left( \lOmega_{R_i / K^{\circ}} \derotimes L^{\circ} \right)_{\text{tor}} \rar \dar{f} & \displaystyle \frac{\lOmega_{R_i / K^{\circ}} \derotimes L^{\circ} }{L^{\circ}\cdot \dlog \phi_0} \rar \dar{g} &  \displaystyle \frac{\left(\lOmega_{R_i / K^{\circ}} \derotimes L^{\circ}\right)_{\text{tf}}}{L^{\circ}\cdot \dlog \phi_0} \rar \dar & 0 \\
		0 \rar & \left( \lOmega_{R_j / K^{\circ}} \derotimes L^{\circ} \right)_{\text{tor}} \rar & \displaystyle \frac{\lOmega_{R_j / K^{\circ}} \derotimes L^{\circ} }{L^{\circ}\cdot \dlog \phi_0} \rar &  \displaystyle \frac{\left(\lOmega_{R_j / K^{\circ}} \derotimes L^{\circ}\right)_{\text{tf}}}{L^{\circ} \cdot \dlog \phi_0} \rar & 0 
	\end{tikzcd}
	.\] 
	\fi
	Note that the middle modules are simply $\lOmega_{R_i / R_0} \otimes L^{\circ}$ and $\lOmega_{R_j / R_0} \otimes L^{\circ}$, which have content $\step \mathfrak{A} _i$, $\step \mathfrak{A} _j$ respectively. 
	We then also see that $g$ is injective by the vanishing of $H_1( \lctc_{R_j / R_i} \derotimes L^{\circ} ) $, and thus $f$ is also injective and the family $(( \lOmega_{R_i / K^{\circ}} \derotimes L^{\circ} )_{\text{tor}})_{i \in I}$ is an increasing union with colimit $(\lOmega_{L^{\circ} / K^{\circ}})_{\text{tor}}$. 

	By \Cref{lem:cont_union}, we already see that $\aldiff_{L / K} = \lim_{i} \cont\left( ( \lOmega_{R_i / K^{\circ}} \derotimes L^{\circ} )_{\text{tor}} \right) $ and that the limit is non-decreasing. 
	It remains to be shown that $\cont\left( ( \lOmega_{R_i / K^{\circ}} \derotimes L^{\circ} )_{\text{tor}} \right) = \aldiff_{(K(T), v_i) / K}$.

	Moreover, by \Cref{rem:total_generators_log_differentials}, we know that $\lOmega_{R_i / K^{\circ}} \derotimes L^{\circ}$ is generated by $\{\dlog \phi, \phi \in \keys \mathfrak{A}_i \}$, and therefore $(\lOmega_{R_i / K^{\circ}} \derotimes L^{\circ})_{\text{tf}}$ maps to an almost unit ball of the minimal semi-valuation on $\Omega_{L / K}$ that is $\keys(\mathfrak{A}_i)$-non-expansive, denoted by $u_i$, which like in \Cref{thm:comp_kahler} is determined by \[
	u_i(\dd T) = \sup \{w(\phi) - w(\phi') \mid \phi \in \keys \mathfrak{A}_i\}
	.\]
	Then, the additivity of content states  
	\begin{align*}
		\cont\left( \lOmega_{R_i / K^{\circ}} \derotimes L^{\circ} \right)_{\text{tor}} &= \step \mathfrak{A} _i + w(\phi_0) -   \sup \{w(\phi) - w(\phi') \mid \phi \in \keys \mathfrak{A}_i\}\\
		&= \step \mathfrak{A} _i + v_i(\phi_0) -   \sup \{v_i(\phi) - v_i(\phi') \mid \phi \in \keys \mathfrak{A}_i\}\\
		&=  \aldiff_{(K(T), v_i) / K}
	.\end{align*}
\end{proof}

\printbibliography

\end{document}